\documentclass[final,leqno]{siamltex}
\usepackage{amssymb}
\usepackage{amsmath}
\usepackage{color}
\usepackage{enumerate}
\usepackage{pgfplots}
\usepackage{tikz}
\usetikzlibrary{shapes.geometric}

\usepackage{graphicx}

\newcommand{\hF}{\widehat{F}}

\newcommand{\cS}{\mathcal{S}}
\newcommand{\cT}{\mathcal{T}}

\newcommand{\cM}{\mathcal{M}}
\newcommand{\cN}{\mathcal{N}}

\newcommand{\Ome}{\Omega}
\newcommand{\oOme}{\overline{\Omega}}
\newcommand{\NJ}{\mathbb{N}_J}

\newcommand{\bv}{\mathbf{ v}}

\newcommand{\del}{\delta}
\newcommand{\ou}{\overline{u}}
\newcommand{\uu}{\underline{u}}

\newcommand{\p}{\partial}

\newcommand{\bh}{\mathbf{h}}
\newcommand{\bx}{\mathbf{x}}
\newcommand{\bs}{\mathbf{s}}
\newcommand{\by}{\mathbf{y}}

\newcommand{\bz}{\mathbf{z}}
\newcommand{\bw}{\mathbf{w}}
\newcommand{\bk}{\mathbf{k}}

\newtheorem{remark}{Remark}[section]

\begin{document}

\title{A Narrow-stencil finite difference method for approximating viscosity solutions 
of fully nonlinear elliptic partial differential equations with applications to Hamilton-Jacobi-Bellman equations} 
 
\author{
Xiaobing Feng\thanks{Department of Mathematics, 
The University of Tennessee, Knoxville, 
TN 37996, U.S.A. {\tt xfeng@math.utk.edu}.
The work of this author was partially supported by the NSF grants DMS-1620168 and DMS-1318486.}
\and
Thomas Lewis\thanks{Department of Mathematics and Statistics, 
The University of North Carolina at Greensboro, 
Greensboro, NC 27402, U.S.A.  {\tt tllewis3@uncg.edu}.}
}

\maketitle

\begin{abstract}
This paper presents a new narrow-stencil finite difference method for approximating 
the viscosity solution of second order fully nonlinear elliptic partial differential equations 
including Hamilton-Jacobi-Bellman equations. The proposed finite difference method naturally extends the Lax-Friedrichs method for first order problems 
to second order problems by introducing a key stabilization and guiding term called a ``numerical moment".
The numerical moment uses the difference of two (central) 
Hessian approximations to resolve the potential low-regularity 
of viscosity solutions. It is proved that the proposed scheme is well posed (i.e, 
it has a unique solution) and stable in both the $\ell^2$-norm 
and the $\ell^\infty$-norm. The highlight of the paper is to prove the convergence 
of the proposed scheme to the viscosity solution of the underlying fully nonlinear 
second order problem using a novel discrete comparison argument.   
This paper extends the one-dimensional analogous method of
\cite{Feng_Kao_Lewis13} to the higher-dimensional setting.  
Numerical tests are presented to gauge the performance of the proposed 
finite difference methods and to validate the convergence result of the paper.  
\end{abstract}



\pagestyle{myheadings}
\thispagestyle{plain}
\markboth{XIAOBING FENG AND THOMAS LEWIS}{A NARROW-STENCIL FD METHOD FOR FULLY NONLINEAR PDES} 

\section{Introduction}\label{intro}

This paper develops and analyzes a narrow-stencil finite difference approximation of
the viscosity solution to the following fully nonlinear second order Dirichlet boundary 
value problem:
\begin{subequations}\label{FD_problem}
\begin{alignat}{2}
F[u](\bx)\equiv F \left( D^2 u, \nabla u, u, \bx \right) 
	& = 0, &&\qquad \forall \bx \in \Omega,
	\label{FD_pde} \\
u(\bx) & = g(\bx), && \qquad \forall \bx \in \partial \Omega,  \label{FD_bc} 
\end{alignat}
\end{subequations}
where $\Omega \subset \mathbb{R}^d$ ($d\geq 2$) is a bounded domain and $D^2u(\bx)$ 
denotes the Hessian matrix of $u$ at $\bx$. 
The PDE operator $F$ is a fully nonlinear second order differential operator in the 
sense that $F$ is nonlinear in $D^2 u$ (or in at least one of its components).  
We assume $g$ is continuous and $F$ is Lipschitz in its first three arguments. 
Moreover, $F$ is assumed to be {\em elliptic} and satisfy a comparison principle
(defined in Section \ref{elliptic_section}). 
The paper focusses on a particular class of fully nonlinear PDEs called 
Hamilton-Jacobi-Bellman (HJB) equations 
which correspond to $F[u]=H[u]$, where 
\begin{subequations}\label{HJB}
\begin{align} 
H[u] &\equiv \min_{\theta \in \Theta} \bigl( L_\theta u - f_\theta \bigr), \\
L_\theta u &= A^{\theta}(\bx) : D^2 u + b^{\theta}(\bx) \cdot \nabla u + c^{\theta}(\bx) \, u . 
\end{align}
\end{subequations}
HJB equations arise in various applications including stochastic optimal control, game theory, 
and mathematical finance (cf. \cite{Fleming_Soner06}). 

Numerical fully nonlinear PDEs has experienced some rapid developments in recent years.   
Various numerical methods have been proposed and tested (cf. \cite{Feng_Glowinski_Neilan10,NSZ16} 
and the references therein). In order to guarantee convergence to the viscosity solution of 
problem \eqref{FD_problem}, most of these works follow the framework laid out by 
Barles-Souganidis in \cite{Barles_Souganidis91} which requires that a numerical scheme 
be monotone, consistent, and stable. It follows from a well-known result of \cite{MotzkinWasow53}
that, in general, monotone schemes are necessarily wide-stencil schemes  
as seen in \cite{Debrabant_Jakobsen13,Feng_Jensen,Oberman08,Salgado_Zhang16}. 
Consequently, to design convergent narrow-stencil schemes, one has to abandon the standard 
monotonicity requirement 
(in the sense of Barles-Souganidis \cite{Barles_Souganidis91}) 
and, instead, deal with non-monotone schemes. 
This in turn prohibits one to directly use the powerful and handy Barles-Souganidis (analysis) framework.  
On the other hand, it should be pointed out that monotone narrow-stencil schemes could be 
constructed for problem \eqref{FD_problem} in special cases. For example, when the matrix $A^\theta$ is
diagonally dominant for all $\theta$, the methods proposed in \cite{Kushner_Dupuis,Bonnans_Zidani}
become narrow-stencil schemes even though they require wide-stencils for general negative semi-definite  
matrices $A^\theta$. For isotropic coefficient matrices $A^\theta$ (i.e., $A^\theta$ diagonal), 
a monotone linear finite element method was proposed in \cite{Jensen_Smears}.  
Finally, we also mention that when $A^\theta$ satisfies a Cordes condition, a least 
square-like interior penalty discontinuous Galerkin (IP-DG) method was proposed in \cite{Smears_Suli}.
This method is obviously a narrow-stencil scheme, and it was proved to be convergent 
even though it is not monotone.  

The primary goal of this paper is to develop a convergent narrow-stencil finite difference 
method for problem \eqref{FD_problem} in the general case when $F$ is uniformly elliptic.  
Thus, we only require $A^\theta$ to be uniformly negative definite 
as long as the underlying PDE satisfies the comparison principle.  
The method to be introduced in this paper can be regarded as a high  
dimensional extension of the one-dimensional finite difference method developed and analyzed by 
the authors in \cite{Feng_Kao_Lewis13}. We note that this extension is far from trivial in terms 
of both method design and convergence analysis because of the additional difficulties caused by
the second order mixed derivatives on the off-diagonal entries of the Hessian matrix $D^2u$ for $d\geq 2$.
Indeed, constructing ``correct" finite difference discretizations for those mixed derivatives 
and controlling their ``bad" effect in the convergence analysis poses some significant new 
challenges. Recall that the main reason to use wide-stencils in the works cited above
is exactly to ``correctly" approximate the second order mixed derivatives. Hence, the new 
challenges which must be dealt with in this paper can be regarded as the price we need to pay for 
not using wide-stencils.  

To achieve our goal, several key ideas are introduced and utilized in this paper.  
First, we introduce multiple discrete approximations of $\nabla u$ and $D^2 u$ to locally 
capture the behavior of the underlying function $u$ when it has low regularity.  
The reason for using such multiple approximations is to compensate for the low resolution of a fixed 
grid, especially when using a narrow-stencil finite difference approximation.  
Second, as explained above, using narrow-stencils forces us to abandon the hope of 
constructing monotone finite difference schemes (again, in the sense of Barles-Souganidis 
\cite{Barles_Souganidis91}).  However, in order to obtain a convergent scheme, the anticipated 
finite difference method must have some ``good" structures.  Our key motivating idea is to impose 
such a structure, called {\em generalized monotonicity}, on our
finite difference scheme. We shall demonstrate that besides allowing for the use of narrow-stencils,
generalized monotonicity is a more natural concept/structure than the usual monotonicity concept/structure, 
and it is much easier to verify. 
Third, to obtain practical generalized monotone finite difference methods, two important questions 
to answer are: how to build such a scheme and what is the primary mechanism that makes the method work?  
Inspired by the vanishing moment method developed in \cite{Feng_Neilan09a,Feng_Neilan11}, we propose
a higher order stabilization term/mechanism, called a {\em numerical moment}, as a means to address 
these questions. 
We will see that the stabilization term/mechanism is the key for us to
overcome the lack of monotonicity in the admissibility, stability, and convergence analysis. 

The remainder of this paper is organized as follows. In Section~\ref{visc_sec} we introduce the
basics of viscosity solution theory and elliptic operators as well as the corresponding notation.  
We also define our primary finite difference operators that will be used throughout 
the paper and list some identities for those difference operators.  In Section~\ref{FD_method_sec}
we first formulate our finite difference method and then prove several properties of the 
scheme. Section~\ref{stability_admissibility_sec} is devoted to studying the well-posedness of the scheme. 
A contraction mapping technique based on a pseudo-time explicit Euler scheme is employed to prove the existence 
and uniqueness for the proposed finite difference scheme and to derive an $\ell^2$-norm stability estimate.  
In Section ~\ref{H2-stability_sec}, an $\ell^2$-norm stability estimate for the second-order central difference quotients of the numerical solution is first derived using a Sobolev iteration technique.  
Then an $\ell^\infty$-norm stability estimate is derived based on combining both $\ell^2$ estimates. 
The results in Sections \ref{stability_admissibility_sec} and \ref{H2-stability_sec} rely upon some 
auxiliary linear algebra results that can be found in Appendix~\ref{appendix_sec}. 
In Section \ref{conv_proof_sec} we present 
a detailed convergence analysis for the proposed finite difference method.  
In Section \ref{FD_2D_numerics} we provide 
some numerical experiments that validate and complement the theory. The paper is concluded in 
Section~\ref{conc_sec} with a brief summary and some final remarks.  

\section{Preliminaries} \label{visc_sec}

We begin by recalling the viscosity solution concept for problem \eqref{FD_problem} 
as documented in \cite{Crandall_Lions83, Crandall_Lions_Evans84,
viscosity_guide} and defining basic difference operators that will be used to construct our 
finite difference method.   We define the relevant function space notation in Section~\ref{spaces_sec},  state the definition of viscosity solutions in Section~\ref{viscosity}, 
and recall the definition of ellipticity and the comparison principle in Section~\ref{elliptic_section}. 
The comparison principle ensures uniqueness of solutions to \eqref{FD_problem}. 
The specific assumptions for HJB equations will be given in Section~\ref{hjb-sec}.
Lastly, in Section~\ref{diff_quotients_sec},  we define our basic difference operators  
which will serve as the building blocks for constructing various discrete gradient and Hessian operators.
The discrete gradient and Hessian operators will be used to define 
our narrow-stencil finite difference method in Section~\ref{FD_method_sec}.  

\subsection{Function spaces and notation} \label{spaces_sec}

Standard function and space notation as in \cite{Caffarelli_Cabre95} and \cite{Gilbarg_Trudinger01}
will be adopted in this paper. For example, for a bounded open domain $\Omega \subset \mathbb{R}^d$, 
$B(\Omega)$, $USC(\Omega)$, and  $LSC(\Omega)$ are used to denote, respectively, the spaces of bounded, upper semicontinuous, and lower semicontinuous functions on $\Omega$.
For any $v\in B(\Ome)$, we define 
\[
v^*(\bx)\equiv \limsup_{\by\to \bx} v(\by) \qquad\mbox{and}\qquad
v_*(\bx)\equiv \liminf_{y\to \bx} v(\by). 
\]
Then, $v^*\in USC(\Omega)$ and $v_*\in LSC(\Omega)$, and they are called
the upper and lower semicontinuous envelopes of $v$, respectively. 
We also let $\mathcal{S}^{d \times d} \subset \mathbb{R}^{d \times d}$ 
denote the set of symmetric real-valued matrices.  

In presenting the relevant viscosity solution concept for fully nonlinear second 
order PDEs, we shall use $G$ to denote a general fully nonlinear second order operator.  
More precisely, $G: \cS^{d\times d}\times\mathbb{R}^d\times \mathbb{R}\times \oOme \to \mathbb{R}$. 
The general second order fully nonlinear PDE problem is to seek a locally bounded function 
$u : \oOme \to \mathbb{R}$ such that $u$ is a viscosity solution of
\begin{align} \label{PDE_wbc}
G(D^2 u,\nabla u, u, \bx) & = 0 \qquad\mbox{in } \oOme ,
\end{align}
where we have adopted the convention of writing the boundary condition as a
discontinuity of the PDE (cf. \cite[p.274]{Barles_Souganidis91}).   
We shall treat the Dirichlet boundary condition explicitly in the formulation of our
finite difference method for \eqref{PDE_wbc}.  


\subsection{Viscosity solutions} \label{viscosity}

Due to the full nonlinearity of the differential operator in \eqref{PDE_wbc}, the standard weak 
solution concept based upon lowering the order of derivatives on the solution function 
using integration by parts is not applicable. 
Thus, a different solution concept is necessary for fully nonlinear PDEs. Among several 
available weak solution concepts for fully nonlinear PDEs, the best known one is the 
{\em viscosity solution} notion whose definition is stated below.
 
\begin{definition} \label{visc_def}
Let $G$ denote the second order operator in \eqref{PDE_wbc}.

(i) A locally bounded function $u: \overline{\Omega} \to \mathbb{R}$ 
is called a {\em viscosity subsolution} of \eqref{PDE_wbc}
if $\; \forall \varphi\in C^2 ( \overline{\Omega})$, when $u^*-\varphi$ has a local 
maximum at $\bx_0 \in \overline{\Omega}$, 
\[
G_*(D^2\varphi(\bx_0), \nabla \varphi(\bx_0),u^*(\bx_0),\bx_0) \leq 0 .
\]

(ii) A locally bounded function $u: \overline{\Omega} \to \mathbb{R}$ is 
called a {\em viscosity supersolution} of \eqref{PDE_wbc} if 
$\; \forall \varphi\in C^2 (\overline{\Omega})$, when $u_*-\varphi$ has a local 
minimum at $\bx_0 \in \overline{\Omega}$, 
\[
G^*(D^2\varphi(\bx_0), \nabla \varphi(\bx_0),u_*(\bx_0),\bx_0) \geq 0 .
\]

(iii) A locally bounded function $u: \overline{\Omega} \to \mathbb{R}$ 
is called a {\em viscosity solution} of \eqref{PDE_wbc}
if $u$ is both a viscosity subsolution and a viscosity supersolution of \eqref{PDE_wbc}.
\end{definition}

\medskip 
\begin{remark} \label{rem2.1} 
(a) Applying the above definition to the Dirichlet boundary value problem \eqref{FD_problem}, 
we require the viscosity solution $u$ satisfies the Dirichlet boundary condition \eqref{FD_bc} 
in the {\em viscosity sense}. Using the notation of \cite[Section 7] {viscosity_guide}, this 
can be described as follows. First, define the boundary operator 
$B \in C \left( \mathcal{S}^{d \times d} , \mathbb{R}^d , \mathbb{R} , \partial \Omega \right)$ by 
\begin{equation} \label{PDE_bcB}
	B(A , \mathbf{q} , v , \bx) \equiv v(\bx) - g(\bx) 
\end{equation}
based on the Dirichlet boundary condition \eqref{FD_bc}, and then define the differential operator 
$G : \cS^{d\times d}\times\mathbb{R}^d\times \mathbb{R}\times \oOme \to \mathbb{R}$ by 
\[
	G(D^2u, \nabla u, u, \bx)\equiv \begin{cases}
	F[u](\bx) , & \text{if } \bx \in \Omega , \\ 
	B[u](\bx) , & \text{if } \bx \in \partial \Omega . 
	\end{cases}
\]
Then we have 
\[
	G_* \left( A , \mathbf{q} , v , \bx \right) = \begin{cases} 
	F \left( A , \mathbf{q} , v , \bx \right) , & \text{if } \bx \in \Omega , \\ 
	\min \big\{ F \left( A , \mathbf{q} , v , \bx \right) , B \left( A , \mathbf{q} , v , \bx \right) \big\} , 
		& \text{if } \bx \in \partial \Omega 
	\end{cases}
\]
and 
\[
	G^* \left( A , \mathbf{q} , v , \bx \right) = \begin{cases} 
	F \left( A , \mathbf{q} , v , \bx \right) , & \text{if } \bx \in \Omega , \\ 
	\max \big\{ F \left( A , \mathbf{q} , v , \bx \right) , B \left( A , \mathbf{q} , v , \bx \right) \big\} , 
		& \text{if } \bx \in \partial \Omega . 
	\end{cases}
\]
Using the boundary operator  $B$, we have explicitly written the boundary condition 
as a discontinuity of the PDE operator $G$.   

(b) If $g$ is continuous, then another way to enforce the Dirichlet boundary condition is
in the pointwise sense.  As explained in \cite{viscosity_guide}, the enforcing of the Dirichlet 
boundary condition in the viscosity sense and in the pointwise sense may not be equivalent in general.
\end{remark}

The above definitions can be informally interpreted as follows.
Without a loss of generality, we may assume $u^*(\bx_0) = \varphi(\bx_0)$ whenever $u^*-\varphi$ has 
a local maximum at $\bx_0$ or $u_*(\bx_0) = \varphi(\bx_0)$ whenever $u_*-\varphi$ has 
a local minimum at $\bx_0$. Then, $u$ is a viscosity solution of \eqref{PDE_wbc} if for all smooth 
functions $\varphi$ such that $\varphi$ ``touches" the graph of $u^*$ from above at $\bx_0$, 
we have $G_*(D^2\varphi(\bx_0), \nabla \varphi(\bx_0),\varphi(\bx_0),\bx_0) \leq 0$, 
and for all smooth functions $\varphi$ such that $\varphi$ ``touches" the graph of $u_*$ from 
below at $\bx_0$, we have $G^*(D^2\varphi(\bx_0), \nabla \varphi(\bx_0),\varphi(\bx_0),\bx_0) \geq 0$.  
A geometric interpretation of the definition of viscosity solutions can be found in 
Figure~\ref{viscosity_pic}.  
Since the definition is based on locally moving derivatives from the viscosity solution 
to a smooth test function, the concept of viscosity solutions relies upon a ``differentiation-by-parts" approach.  This is in contrast to weak solution theory where derivatives 
are moved to a test function using an ``integration-by-parts" approach, which, in general, 
is not a local operation.  

\begin{figure}[htb] 
\begin{center}
\includegraphics[width=0.43\textwidth]{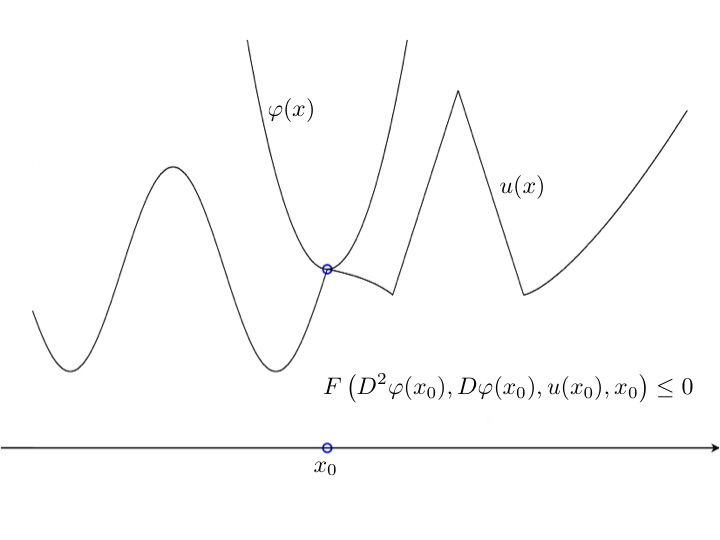}
\qquad 
\includegraphics[width=0.43\textwidth]{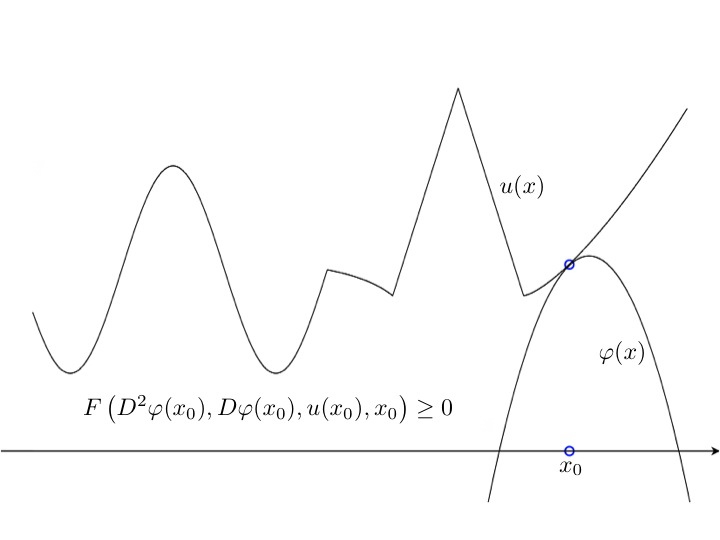}
\caption{
A geometric interpretation of viscosity solutions for second order problems. 
The figure on the left corresponds to a viscosity subsolution being ``touched" from above 
and the figure on the right corresponds to a viscosity supersolution being ``touched" from below. 
}
\label{viscosity_pic}
\end{center}
\end{figure}

\subsection{Ellipticity and the comparison principle} \label{elliptic_section}

An advantage of the viscosity solution concept for fully nonlinear problems is that, 
due to its entirely local structure, it is flexible enough to seek solutions in the space of bounded  functions. To ensure the existence and uniqueness of viscosity solutions,  additional structure 
conditions on the differential operator $G$ are often necessary. The most common such conditions 
are an ellipticity requirement and a comparison principle. The following definitions of ellipticity 
and the comparison principle are standard (cf. \cite{viscosity_guide}).

\begin{definition}\label{elliptic}
(i) Equation \eqref{PDE_wbc} is said to be {\em uniformly elliptic} if, for all 
$(\mathbf{p},v, \bx)\in \mathbb{R}^d \times \mathbb{R} \times \oOme$, there holds
\begin{align*}
\lambda \text{tr } (A-B) \leq 
G(B, \mathbf{p}, v, \bx) - G(A, \mathbf{p}, w, \bx) 
\leq \Lambda \text{tr } (A-B) 
\end{align*}
whenever $A \geq B$, 
where $A, B \in \mathcal{S}^{d \times d}$ 
and $A\geq B$ means that $A-B$ is a nonnegative definite matrix.

(ii) Equation \eqref{PDE_wbc} is said to be {\em proper elliptic} if, for all 
$(\mathbf{p},\bx)\in \mathbb{R}^d \times \oOme$, there holds
\begin{align*}
G(A, \mathbf{p}, v, \bx) \leq G(B, \mathbf{p}, w, \bx) 
\end{align*}
whenever $v \leq w$ and $A \geq B$, 
where $v,w \in \mathbb{R}$ and $A, B \in \mathcal{S}^{d \times d}$.  

(iii) Equation \eqref{PDE_wbc} is said to be {\em degenerate elliptic} if, for all 
$(\mathbf{p},v, \bx)\in \mathbb{R}^d \times \mathbb{R} \times \oOme$, there holds
\begin{align*}
G(A, \mathbf{p}, v, \bx) \leq G(B, \mathbf{p}, v, \bx)  
\end{align*}
whenever $A \geq B$, 
where $A, B \in \mathcal{S}^{d \times d}$. 
\end{definition}

\begin{definition}\label{comparison}
Problem \eqref{PDE_wbc} is said to satisfy a {\em comparison 
principle} if the following statement holds. For any upper semicontinuous 
function $u$ and lower semicontinuous function $v$ on $\overline{\Omega}$,  
if $u$ is a viscosity subsolution and $v$ is a viscosity supersolution 
of \eqref{PDE_wbc}, then $u\leq v$ on $\overline{\Omega}$.
\end{definition}
 
The comparison principle is sometimes called a {\em strong uniqueness property} 
due to the fact it immediately yields the uniqueness of viscosity solutions 
for problem \eqref{PDE_wbc} (cf. \cite{Barles_Souganidis91}).  
When $G$ is differentiable, the degenerate ellipticity condition 
can also be defined by requiring that the matrix $\frac{\partial G}{\partial D^2 u}$ 
is negative semi-definite  while the proper ellipticity condition additionally requires 
the number $\frac{\partial G}{\partial u}$ is nonnegative (cf. \cite[p. 441]{Gilbarg_Trudinger01}).
We then clearly see that the ellipticity concept for nonlinear problems  
generalizes the notion for linear elliptic equations.  

\begin{remark}
Since we have assumed that $G$ satisfies the comparison principle, 
it follows that that problem \eqref{PDE_wbc} has a unique viscosity solution $u \in C(\overline{\Omega})$.  
\end{remark}

\subsection{Hamilton-Jacobi-Bellman equations} \label{hjb-sec}

As mentioned in Section \ref{intro},  a specific class of fully nonlinear PDEs that fits the structure 
requirements in this paper are (stationary) Hamilton-Jacobi-Bellman (HJB) equations described by \eqref{HJB}.
  
We shall assume that $A^\theta$ is uniformly (in $\theta$) negative definite in the sense 
that there exists constants $0 < \lambda \leq \Lambda$ such that 
\begin{equation}\label{ellipticity}
	- \Lambda | \xi |^2 \leq A^{\theta} (\bx) \, \xi \cdot \xi 	\leq - \lambda | \xi |^2 \qquad
\forall \xi \in \mathbb{R}^d, \bx \in \Omega, \theta \in \Theta.
\end{equation}   
We also assume there exists a constant $K>0$ such that 
\[
	\| A^\theta\|_{L^\infty(\Omega)} , \, \| b^\theta\|_{L^\infty(\Omega)}, \,
	\|c^\theta\|_{L^\infty(\Omega)}  \leq K
\]
and $c^\theta (\bx) \geq k_0 > 0$ for some constant $k_0$ 
for all $\bx \in \Omega$ and $\theta \in \Theta$.    
Thus, the equation is proper elliptic and satisfies the comparison principle.  
We note that for stochastic optimal control applications, we have $c^\theta \equiv 0$
(cf. \cite{Fleming_Soner06}).  
This case will also be considered.  
However, such problems may not satisfy a comparison principle as 
discussed in \cite{NNZ17} and \cite{Safonov99}.  

Under the above assumptions on the coefficients and the control set, problem 
\eqref{FD_problem}--\eqref{HJB} has a unique solution 
and $H$ is globally Lipschitz continuous. 
Thus, $H$ is differentiable almost everywhere and the weak derivatives are bounded by 
the constant $K$.  
Since $\Theta$ is closed and bounded, there actually exists a function $\theta^*(x) \in \Theta$ 
such that 
\[
	H[u] = \min_{\theta \in \Theta} \bigl( L_\theta u - f_\theta \bigr) 
	= A^{\theta^*} : D^2 u + b^{\theta^*} \cdot \nabla u + c^{\theta^*} \, u 
	= 0 \qquad \text{in } \Omega . 
\]
Thus, $\frac{\partial H}{\partial D^2 u} = A^{\theta^*}$, $\frac{\partial H}{\partial \nabla u} = b^{\theta^*}$, 
and $\frac{\partial H}{\partial u} = c^{\theta^*}$, and we have 
$H$ is differentiable everywhere with bounded first order derivatives.  
 
The analysis assumes a global Lipschitz condition for ease of presentation; however, 
locally Lipschitz should be sufficient.    
The global Lipschitz assumption excludes Monge-Amp\`ere-type equations
which are only conditionally degenerate elliptic \cite{Caffarelli_Cabre95} and locally Lipschitz. 
In light of the recent work \cite{Feng_Jensen}, we know that the Monge-Amp\`ere equation 
has an equivalent HJB reformulation.  
Much of the work of this paper can be extended to the Monge-Amp\`ere equation via its HJB reformulation.

\subsection{Difference operators} \label{diff_quotients_sec}

We introduce several difference operators for approximating first and second order 
partial derivatives. The multiple difference operators will be used to help resolve 
the low regularity of viscosity solutions and will play a key role in motivating and defining 
our new narrow-stencil FD methods.  

Assume $\Omega$ is a $d$-rectangle, i.e., 
$\Omega = \left( a_1 , b_1 \right) \times \left( a_2 , b_2 \right) \times \cdots \times 
		\left( a_d , b_d \right)$.    
We shall only consider grids that are uniform in each coordinate $x_i$, $i = 1, 2, \ldots, d$.  
Let $J_i (\geq 2)$ be an integer and $h_i = \frac{b_i-a_i}{J_i-1}$ for $i = 1, 2, \ldots, d$. 
Define $\mathbf{h} = \left( h_1, h_2, \ldots, h_d \right) \in \mathbb{R}^d$, 
$h = \max_{i=1,2,\ldots,d} h_i$,  $J = \prod_{i=1}^d J_i$, and   
$\NJ = \{ \alpha = (\alpha_1, \alpha_2, \ldots, \alpha_d) 
\mid 1 \leq \alpha_i \leq J_i, i = 1, 2, \ldots, d \}$.  Then, $\left| \NJ \right| = J$. 
We partition $\Omega$ into $\prod_{i=1}^d \left(J_i-1 \right)$ sub-$d$-rectangles with grid points
$\mathbf{x}_{\alpha} = \bigl (a_1+ (\alpha_1-1)h_1 , a_2 + (\alpha_2-1) h_2 , \ldots , 
		a_d + (\alpha_d-1) h_d \bigr)$
for each multi-index $\alpha \in \NJ$.
We call $\cT_{\mathbf{h}}=\{\mathbf{x}_{\alpha} \}_{\alpha \in \NJ}$ 
a mesh (set of nodes) for $\overline{\Omega}$. 
We also introduce an extended mesh 
$\cT_{\mathbf{h}}'$ which extends  
$\cT_{\mathbf{h}}$ by a collection of ghost grid points that are at most one layer exterior to 
$\overline{\Omega}$ in each coordinate direction. 
In particular, we choose ghost grid points $\bx$ such that 
$\bx = \by \pm 2 h_i \mathbf{e}_i$ for some $\by \in \cT_{\bh} \cap \Omega$, $i \in \{1,2,\ldots,d\}$.  
We set $J_i' = J_i+2$ and $\NJ'$ is defined by replacing $J_i$ by $J_i'$ in 
the definition of $\NJ$ and then removing the extra multi-indices that would correspond to ghost grid points 
that are not in the set $\cT_{\bh}'$ to ensure $| \NJ' | = | \cT_{\bh}' |$.  


\subsubsection{First order difference operators} \label{discrete_gradients_sec}
 
We define the standard forward and backward difference operators as well as the central 
difference operator for approximating first order derivatives including the gradient operator. 
The forward and backward difference operators will serve as the building blocks for constructing
all of the first and second order difference operators in this paper.
 
Let $\left\{ \mathbf{e}_i \right\}_{i=1}^d$ denote the canonical basis vectors for $\mathbb{R}^d$. 
Define the (first order) forward and backward difference operators by
\begin{equation} \label{fd_x}
	\del_{x_i,h_i}^+ v(\mathbf{x})\equiv \frac{v(\mathbf{x} + h_i \mathbf{e}_i) - v(\mathbf{x})}{h_i},\qquad
	\del_{x_i,h_i}^- v(\mathbf{x})\equiv \frac{v(\mathbf{x})- v(\mathbf{x}-h_i \mathbf{e}_i)}{h_i}
\end{equation}
for a function $v$ defined on $\mathbb{R}^d$ and 
\[
	\del_{x_i,h_i}^+ V_\alpha \equiv \frac{V_{\alpha + \mathbf{e}_i} - V_\alpha}{h_i} , \qquad
	\del_{x_i,h_i}^- V_\alpha \equiv \frac{V_\alpha- V_{\alpha - \mathbf{e}_i}}{h_i}
\]
for a grid function $V$ defined on the grid $\cT_h$.  Note that ``ghost-values" may need to be 
introduced in order for the above difference operators to be well-defined on the boundary of $\Omega$.
We also define the following central difference operator: 
\begin{equation} \label{fd_xc}
	\delta_{x_i, h_i}  \equiv \frac{1}{2} \left( \delta_{x_i, h_i}^+ + \delta_{x_i, h_i}^- \right) 
\qquad\mbox{for } i=1,2,\cdots, d.
\end{equation} 
Lastly, we define the ``sided" and central gradient operators $\nabla_{\mathbf{h}}^+$, 
$\nabla_{\mathbf{h}}^-$, and $\nabla_{\mathbf{h}}$ by 
\begin{align} \label{discrete_grad_def}
	\nabla_{\mathbf{h}}^\pm \equiv \bigl[ \delta^\pm_{x_1, h_1} , \delta_{x_2, h_2}^\pm , \cdots , 
		\delta_{x_d, h_d}^\pm \bigr]^T, \quad   
	\nabla_{\mathbf{h}} \equiv \bigl[ \delta_{x_1, h_1} , \delta_{x_2, h_2} , \cdots , 
		\delta_{x_d, h_d} \bigr]^T. 
\end{align}


\subsubsection{Second order difference operators} \label{FD_hessian_sec}
We now introduce a number of difference operators that approximate second 
order derivatives including the Hessian operator. Using the forward and backward 
difference operators introduced in the previous subsection, we have the following
four possible approximations of the second order differential operator
$\partial^2_{x_ix_j}$:
 
\[
D_{\mathbf{h},ij}^{\mu\nu}\equiv \delta_{x_j, h_j}^\nu \delta_{x_i, h_i}^\mu \qquad
\mbox{for } \mu,\nu\in \{+,-\},
\]
which in turn leads to the definition of the following four approximations
of the Hessian operator $D^2\equiv [\partial^2_{x_ix_j}]$: 
\begin{equation}\label{discrete_Hessian}
D_{\mathbf{h}}^{\mu\nu} \equiv \bigl[D_{\mathbf{h},ij}^{\mu\nu}\bigr]_{i,j=1}^d  \qquad
\mbox{for } \mu,\nu\in \{+,-\}.
\end{equation}

To construct our numerical methods in the next section, we also need to introduce the
following three sets of averaged second order difference operators:
\begin{subequations}\label{fdxy}
\begin{align}
	\widehat{\delta}_{x_i, x_j; h_i, h_j}^2 
	& \equiv \frac12 \bigl( D_{\mathbf{h},ij}^{+-} + D_{\mathbf{h},ij}^{-+} \bigr)
          =\frac12 \bigl(\delta_{x_j, h_j}^- \delta_{x_i, h_i}^+  
		+ \delta_{x_j, h_j}^+ \delta_{x_i, h_i}^- \bigr) , \\ 
	\widetilde{\delta}_{x_i, x_j; h_i, h_j}^2 
	& \equiv \frac12 \bigl( D_{\mathbf{h},ij}^{--} + D_{\mathbf{h},ij}^{++} \bigr)
           = \frac12 \bigl(\delta_{x_j, h_j}^+ \delta_{x_i, h_i}^+  
		+ \delta_{x_j, h_j}^-\delta_{x_i, h_i}^-  \bigr) , \\ 
	\overline{\delta}_{x_i, x_j; h_i, h_j}^2 
	& \equiv \frac12 \bigl( \widehat{\delta}_{x_i, x_j; h_i, h_j}^2 
                + \widetilde{\delta}_{x_i, x_j; h_i, h_j}^2  \bigr) 
                = \frac14 \bigl( D_{\mathbf{h},ij}^{++} + D_{\mathbf{h},ij}^{+-} + D_{\mathbf{h},ij}^{-+} 
                		+ D_{\mathbf{h},ij}^{--} \bigr) \\ 
         &=\frac14 \left(\delta_{x_j, h_j}^+ \delta_{x_i, h_i}^+  + \delta_{x_j, h_j}^- \delta_{x_i, h_i}^+  
		+ \delta_{x_j, h_j}^+ \delta_{x_i, h_i}^-  + \delta_{x_j, h_j}^- \delta_{x_i, h_i}^- \right)  
        \nonumber
\end{align}
\end{subequations}
for all $i,j = 1,2,\ldots, d$. 
For notation brievity, we also set
\[
\delta_{x_i,h_i}^2\equiv\widehat{\delta}_{x_i,h_i}^2\equiv\widehat{\delta}_{x_i, x_i; h_i, h_i}^2,\qquad
\widetilde{\delta}_{x_i,h_i}^2\equiv\widetilde{\delta}_{x_i, x_i; h_i, h_i}^2,\qquad
\overline{\delta}_{x_i,h_i}^2\equiv\overline{\delta}_{x_i, x_i; h_i, h_i}^2 . 
\]
Using the above difference operators, we define the following three ``centered" approximations 
of the Hessian operator $D^2\equiv [\partial^2_{x_ix_j}]$:
\begin{equation} \label{fd_hessc}
\widehat{D}_{\mathbf{h}}^2=\bigl[\widehat{\delta}_{x_i, x_j; h_i, h_j}^2 \bigr]_{i,j=1}^d,\quad
\widetilde{D}_{\mathbf{h}}^2=\bigl[\widetilde{\delta}_{x_i, x_j; h_i, h_j}^2 \bigr]_{i,j=1}^d,\quad
\overline{D}_{\mathbf{h}}^2=\bigl[\overline{\delta}_{x_i, x_j; h_i, h_j}^2 \bigr]_{i,j=1}^d.
\end{equation}
We will also consider the (standard) central approximation of the Hessian operator $D_{\bh}^2$ 
defined by 
\begin{equation} \label{standard_FD_Hessian}
	\left[ D_{\mathbf{h}}^2 \right]_{ij} \equiv \begin{cases}
	\delta_{x_i, h_i}^2 & \text{if } i = j , \\ 
	\overline{\delta}_{x_i, x_j; h_i, h_j}^2 , \qquad & \text{if } i \neq j . 
	\end{cases}
\end{equation}

We now record some facts that can be easily verified for the various second order operators.  
First, 
\[
2\widehat{D}_{\mathbf{h}}^2=D_\mathbf{h}^{+-} + D_\mathbf{h}^{-+}, \quad
2\widetilde{D}_{\mathbf{h}}^2=D_\mathbf{h}^{--} + D_\mathbf{h}^{++}, \quad
2\overline{D}_{\mathbf{h}}^2 = \widehat{D}_{\mathbf{h}}^2 + \widetilde{D}_{\mathbf{h}}^2.
\]
Second, a simple computation reveals 
\begin{subequations} \label{discrete_hess_explicit}
	{\small
\begin{align}
	\widehat{\delta}_{x_i, x_j; h_i, h_j}^2 v(\mathbf{x}) 
		& = - \frac{v(\mathbf{x}+h_i\mathbf{e}_i-h_j\mathbf{e}_j) - v(\mathbf{x}+h_i\mathbf{e}_i ) 
			- \big[ v(\mathbf{x}-h_j\mathbf{e}_j ) 	- v(\mathbf{x}) \big]}{2 h_i h_j} \\ 
		\nonumber & \quad 
			- \frac{v(\mathbf{x}-h_i\mathbf{e}_i+h_j\mathbf{e}_j) - v(\mathbf{x}-h_i\mathbf{e}_i ) 
			- \big[ v(\mathbf{x}+h_j\mathbf{e}_j) - v(\mathbf{x}) 
			\big]}{2 h_i h_j} , \\ 
	\widetilde{\delta}_{x_i, x_j; h_i, h_j}^2 v(\mathbf{x}) 
		& = \frac{v(\mathbf{x}+h_i\mathbf{e}_i+h_j\mathbf{e}_j) - v(\mathbf{x}+h_i\mathbf{e}_i )
			- \big[ v(\mathbf{x}+h_j\mathbf{e}_j ) - v(\mathbf{x}) \big]}{2 h_i h_j} \\ 
		\nonumber & \quad + \frac{ v(\mathbf{x}-h_i\mathbf{e}_i-h_j\mathbf{e}_j) 
			- v(\mathbf{x}-h_i\mathbf{e}_i ) 
			- \big[ v(\mathbf{x}-h_j\mathbf{e}_j ) - v(\mathbf{x}) \big] }{2 h_i h_j} ,\\
        \overline{\delta}_{x_i, x_j; h_i, h_j}^2 v(\mathbf{x}) 
        & = \frac{v(\mathbf{x}+h_i \mathbf{e}_i+h_j \mathbf{e}_j) 
                        + v(\mathbf{x}-h_i\mathbf{e}_i-h_j\mathbf{e}_j)}{4 h_i h_j} \\
                \nonumber & \quad
                        - \frac{v(\mathbf{x}-h_i \mathbf{e}_i+h_j \mathbf{e}_j)  
                        + v(\mathbf{x}+h_i\mathbf{e}_i-h_j\mathbf{e}_j)}{4 h_i h_j} . 
\end{align}
}
\end{subequations}
Lastly, there holds the relationships 
\begin{subequations}\label{hess_diagonal}
\begin{align}
	\delta_{x_i, h_i}^2 & = \delta_{x_i, h_i}^\pm \delta_{x_i, h_i}^\mp , \\ 
	\delta_{x_i, 2 h_i}^2 & = \overline{\delta}_{x_i, h_i}^2 = \delta_{x_i, h_i} \delta_{x_i, h_i} . 
\end{align}
\end{subequations}

We have $\delta_{x_i,h_i}^2\equiv \widehat{\delta}_{x_i,h_i}^2$ is the standard (second order) three-point 
central difference operator for approximating second order derivatives in one-dimension, 
$\overline{\delta}_{x_i, h_i}^2 = \delta_{x_i,2h_i}^2$ is a lower-resolution (second-order) three-point 
central difference operator for approximating second order derivatives in one-dimension, 
$\widetilde{\delta}_{x_i, h_i}^2$ 
is a (second order) five-point central difference operator 
for approximating second order derivatives in one-dimension, 
$\overline{\delta}_{x_i, x_j; h_i, h_j}^2$ is the standard (second order) central difference operator 
for approximating second order mixed derivatives, 
and $\widehat{\delta}_{x_i, x_j; h_i, h_j}^2$ and $\widetilde{\delta}_{x_i, x_j; h_i, h_j}^2$ 
are two alternative (second order) central difference operators 
for approximating second order mixed derivatives.  
Observe that in two dimensions, 
$\widehat{D}_{\mathbf{h}}^2$ corresponds to a 7-point stencil, 
$\overline{D}_{\mathbf{h}}^2$ and $D_{\bh}^2$ correspond to 9-point stencils, 
and $\widetilde{D}_{\mathbf{h}}^2$ corresponds to an 11-point stencil, 
as seen in Figure~\ref{central_hessians_pic}.  
The operators $\widehat{D}_{\mathbf{h}}^2$ and $D_{\bh}^2$ 
use nearest neighbors in the stencil (including diagonal directions) 
while $\overline{D}_{\mathbf{h}}^2$ and $\widetilde{D}_{\mathbf{h}}^2$ 
use neighbors two steps away in each Cartesian direction.  

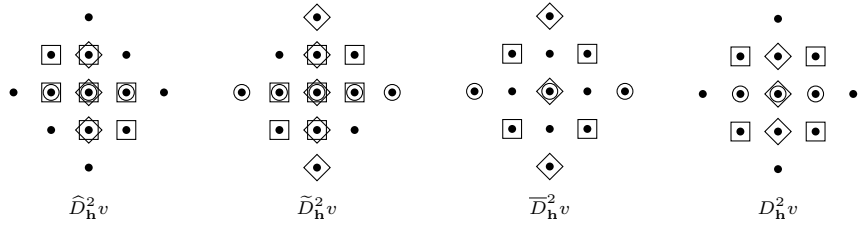
\begin{figure} 
\begin{center}
\begin{tikzpicture}[scale = 0.5]
\fill (-1,0) circle (3pt);
\fill (0,0) circle (3pt);
\fill (1,0) circle (3pt);
\fill (0,-1) circle (3pt);
\fill (0,1) circle (3pt);
\fill (-1,1) circle (3pt);
\fill (-1,-1) circle (3pt);
\fill (1,-1) circle (3pt);
\fill (1,1) circle (3pt);
\fill (0,2) circle (3pt);
\fill (0,-2) circle (3pt);
\fill (2,0) circle (3pt);
\fill (-2,0) circle (3pt);
\draw (-1.25, 1.25) -- (-0.75, 1.25) -- (-0.75, 0.75) -- (-1.25, 0.75) -- cycle;
\draw (1.25, -1.25) -- (0.75, -1.25) -- (0.75, -0.75) -- (1.25, -0.75) -- cycle; 
\draw (-0.25, 0.25) -- (0.25, 0.25) -- (0.25, -0.25) -- (-0.25, -0.25) -- cycle;
\draw (-1.25, 0.25) -- (-0.75, 0.25) -- (-0.75, -0.25) -- (-1.25, -0.25) -- cycle;
\draw (1.25, 0.25) -- (0.75, 0.25) -- (0.75, -0.25) -- (1.25, -0.25) -- cycle;
\draw (0.25, -1.25) -- (0.25, -0.75) -- (-0.25, -0.75) -- (-0.25, -1.25) -- cycle;
\draw (0.25, 1.25) -- (0.25, 0.75) -- (-0.25, 0.75) -- (-0.25, 1.25) -- cycle;
\draw (0,0) circle (6pt);
\draw (1,0) circle (6pt);
\draw (-1,0) circle (6pt);
\draw (0,-1.35) -- (0.35,-1) -- (0,-0.65) -- (-0.35,-1) -- cycle;
\draw (0,-0.35) -- (0.35,0) -- (0,0.35) -- (-0.35,0) -- cycle;
\draw (0,1.35) -- (0.35,1) -- (0,0.65) -- (-0.35,1) -- cycle;
\node[below] at (0,-2.5) {\scriptsize $\widehat{D}_{\mathbf{h}}^2 v$};
\end{tikzpicture}
\hspace{0.25in}
\begin{tikzpicture}[scale = 0.5]
\fill (-1,0) circle (3pt);
\fill (0,0) circle (3pt);
\fill (1,0) circle (3pt);
\fill (0,-1) circle (3pt);
\fill (0,1) circle (3pt);
\fill (-1,1) circle (3pt);
\fill (-1,-1) circle (3pt);
\fill (1,-1) circle (3pt);
\fill (1,1) circle (3pt);
\fill (0,2) circle (3pt);
\fill (0,-2) circle (3pt);
\fill (2,0) circle (3pt);
\fill (-2,0) circle (3pt);
\draw (0.75, 0.75) -- (1.25, 0.75) -- (1.25, 1.25) -- (0.75, 1.25) -- cycle;
\draw (-0.75, -0.75) -- (-1.25, -0.75) -- (-1.25, -1.25) -- (-0.75, -1.25) -- cycle;
\draw (-0.25, 0.25) -- (0.25, 0.25) -- (0.25, -0.25) -- (-0.25, -0.25) -- cycle;
\draw (-1.25, 0.25) -- (-0.75, 0.25) -- (-0.75, -0.25) -- (-1.25, -0.25) -- cycle;
\draw (1.25, 0.25) -- (0.75, 0.25) -- (0.75, -0.25) -- (1.25, -0.25) -- cycle;
\draw (0.25, -1.25) -- (0.25, -0.75) -- (-0.25, -0.75) -- (-0.25, -1.25) -- cycle;
\draw (0.25, 1.25) -- (0.25, 0.75) -- (-0.25, 0.75) -- (-0.25, 1.25) -- cycle;
\draw (0,0) circle (6pt); 
\draw (1,0) circle (6pt);
\draw (-1,0) circle (6pt);
\draw (2,0) circle (6pt);
\draw (-2,0) circle (6pt);
\draw (0,-2.35) -- (0.35,-2) -- (0,-1.65) -- (-0.35,-2) -- cycle;
\draw (0,-0.35) -- (0.35,0) -- (0,0.35) -- (-0.35,0) -- cycle;
\draw (0,2.35) -- (0.35,2) -- (0,1.65) -- (-0.35,2) -- cycle;
\draw (0,-1.35) -- (0.35,-1) -- (0,-0.65) -- (-0.35,-1) -- cycle;
\draw (0,1.35) -- (0.35,1) -- (0,0.65) -- (-0.35,1) -- cycle;
\node[below] at (0,-2.5) {\scriptsize $\widetilde{D}_{\mathbf{h}}^2 v$};
\end{tikzpicture}
\hspace{0.25in}
\begin{tikzpicture}[scale = 0.5]
\fill (-1,0) circle (3pt);
\fill (0,0) circle (3pt);
\fill (1,0) circle (3pt);
\fill (0,-1) circle (3pt);
\fill (0,1) circle (3pt);
\fill (-1,1) circle (3pt);
\fill (-1,-1) circle (3pt);
\fill (1,-1) circle (3pt);
\fill (1,1) circle (3pt);
\fill (0,2) circle (3pt);
\fill (0,-2) circle (3pt);
\fill (2,0) circle (3pt);
\fill (-2,0) circle (3pt);
\draw (-1.25, 1.25) -- (-0.75, 1.25) -- (-0.75, 0.75) -- (-1.25, 0.75) -- cycle;
\draw (1.25, -1.25) -- (0.75, -1.25) -- (0.75, -0.75) -- (1.25, -0.75) -- cycle;
\draw (0.75, 0.75) -- (1.25, 0.75) -- (1.25, 1.25) -- (0.75, 1.25) -- cycle; 
\draw (-0.75, -0.75) -- (-1.25, -0.75) -- (-1.25, -1.25) -- (-0.75, -1.25) -- cycle; 
\draw (0,0) circle (6pt);
\draw (2,0) circle (6pt);
\draw (-2,0) circle (6pt);
\draw (0,-2.35) -- (0.35,-2) -- (0,-1.65) -- (-0.35,-2) -- cycle;
\draw (0,-0.35) -- (0.35,0) -- (0,0.35) -- (-0.35,0) -- cycle;
\draw (0,2.35) -- (0.35,2) -- (0,1.65) -- (-0.35,2) -- cycle;
\node[below] at (0,-2.5) {\scriptsize $\overline{D}_{\mathbf{h}}^2 v$};
\end{tikzpicture}
\hspace{0.25in}
\begin{tikzpicture}[scale = 0.5]
\fill (-1,0) circle (3pt);
\fill (0,0) circle (3pt);
\fill (1,0) circle (3pt);
\fill (0,-1) circle (3pt);
\fill (0,1) circle (3pt);
\fill (-1,1) circle (3pt);
\fill (-1,-1) circle (3pt);
\fill (1,-1) circle (3pt);
\fill (1,1) circle (3pt);
\fill (0,2) circle (3pt);
\fill (0,-2) circle (3pt);
\fill (2,0) circle (3pt);
\fill (-2,0) circle (3pt);
\draw (-1.25, 1.25) -- (-0.75, 1.25) -- (-0.75, 0.75) -- (-1.25, 0.75) -- cycle;
\draw (1.25, -1.25) -- (0.75, -1.25) -- (0.75, -0.75) -- (1.25, -0.75) -- cycle;
\draw (0.75, 0.75) -- (1.25, 0.75) -- (1.25, 1.25) -- (0.75, 1.25) -- cycle; 
\draw (-0.75, -0.75) -- (-1.25, -0.75) -- (-1.25, -1.25) -- (-0.75, -1.25) -- cycle; 
\draw (0,0) circle (6pt);
\draw (1,0) circle (6pt);
\draw (-1,0) circle (6pt);
\draw (0,-1.35) -- (0.35,-1) -- (0,-0.65) -- (-0.35,-1) -- cycle;
\draw (0,-0.35) -- (0.35,0) -- (0,0.35) -- (-0.35,0) -- cycle;
\draw (0,1.35) -- (0.35,1) -- (0,0.65) -- (-0.35,1) -- cycle;
\node[below] at (0,-2.5) {\scriptsize $D_{\mathbf{h}}^2 v$};
\end{tikzpicture}
\end{center}
\caption{
Local stencils for the discrete central Hessian operators  
$\widehat{D}_{\mathbf{h}}^2 v$, $\widetilde{D}_{\mathbf{h}}^2 v$, 
$\overline{D}_{\mathbf{h}}^2 v$, 
and $D_{\bh}^2 v$ in two dimensions.  
Circled nodes correspond to a discrete approximation of $v_{xx}$.  
Boxed nodes correspond to a discrete approximation of $v_{xy}$.
Nodes with a diamond correspond to a discrete approximation of $v_{yy}$.  
}
\label{central_hessians_pic}
\end{figure}

We lastly introduce the (diagonal) variables 
$\xi_i^j = \xi_j^i$ and $\eta_i^j = \eta_j^i$ so that we can define the central difference operators 
$\delta_{\xi_i^j, \mathbf{h}}^2$ and $\delta_{\eta_i^j, \mathbf{h}}^2$ by 
\begin{subequations} \label{diag_laplacians}
\begin{align}
	\delta_{\xi_i^j, \mathbf{h}}^2 v(\mathbf{x}) 
	& \equiv \frac{v(\mathbf{x}-h_i \mathbf{e}_i + h_j \mathbf{e}_j) - 2 v(\mathbf{x}) 
		+ v(\mathbf{x}+h_i \mathbf{e}_i - h_j \mathbf{e}_j)}{h_i^2 + h_j^2} , \\ 
	\delta_{\eta_i^j, \mathbf{h}}^2 v(\mathbf{x}) 
	& \equiv \frac{v(\mathbf{x}+h_i \mathbf{e}_i + h_j \mathbf{e}_j) - 2 v(\mathbf{x}) 
		+ v(\mathbf{x}-h_i \mathbf{e}_i - h_j \mathbf{e}_j)}{h_i^2 + h_j^2}
\end{align}
\end{subequations}
for all $i,j = 1,2,\ldots,d$ with $i \neq j$. Then, a direct computation yields
\begin{subequations} \label{hessian_diag}
\begin{align}
	\label{hessian_diaga} 
	\bigl[ \overline{D}_{\mathbf{h}}^2 v \bigr]_{ij} 
	& = \bigl[ D_{\mathbf{h}}^2 v \bigr]_{ij} 
	   = \frac{h_i^2+h_j^2}{4 h_i h_j} \left( \delta_{\eta_i^j,\mathbf{h}}^2 v - 
		\delta_{\xi_i^j,\mathbf{h}}^2 v \right) , \\ 
	\bigl[ \widehat{D}_{\mathbf{h}}^2 v \bigr]_{ij} 
	& = \frac{h_i}{2 h_j} \delta_{x_i,h_i}^2 v + \frac{h_j}{2 h_i} \delta_{x_j,h_j}^2 v 
		- \frac{h_i^2 + h_j^2}{2 h_i h_j} \delta_{\xi_i^j,\mathbf{h}}^2 v , \\ 
	\bigl[ \widetilde{D}_{\mathbf{h}}^2 v \bigr]_{ij} 
	& = - \frac{h_i}{2 h_j} \delta_{x_i,h_i}^2 v - \frac{h_j}{2 h_i} \delta_{x_j,h_j}^2 v 
		+ \frac{h_i^2 + h_j^2}{2 h_i h_j} \delta_{\eta_i^j,\mathbf{h}}^2 v ,  
\end{align}
\end{subequations}
relationships that are illustrated in Figure~\ref{hessian_diag_pic}.   
Lastly, there holds  
\begin{align} \label{moment_ii_2h}
	\bigl[ \widetilde{D}_{\mathbf{h}}^2 V_\alpha \bigr]_{ii} - \bigl[ \widehat{D}_{\mathbf{h}}^2 V_\alpha \bigr]_{ii} 
	& = \widetilde{\delta}_{x_i, h_i}^2 V_\alpha - \delta_{x_i, h_i}^2 V_\alpha \\ 
	\nonumber & = 2 \left( \delta_{x_i, 2 h_i}^2 V_\alpha - \delta_{x_i, h_i}^2 V_\alpha \right) \\ 
	\nonumber & = \frac12 \left( \delta_{x_i, h_i}^2 V_{\alpha - \mathbf{e}_i} - 2 \delta_{x_i, h_i}^2 V_\alpha 
		+ \delta_{x_i, h_i}^2 V_{\alpha+\mathbf{e}_i} \right)
\end{align}
for all $i=1,2,\ldots,d$.  

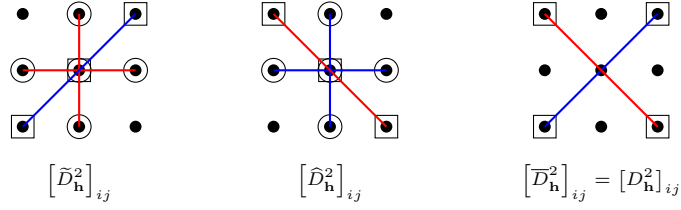
\begin{figure}
\begin{center}
\begin{tikzpicture}[scale = 0.75]
\fill (-1,0) circle (3pt);
\fill (0,0) circle (3pt);
\fill (1,0) circle (3pt);
\fill (0,-1) circle (3pt);
\fill (0,1) circle (3pt);
\fill (-1,1) circle (3pt);
\fill (-1,-1) circle (3pt);
\fill (1,-1) circle (3pt);
\fill (1,1) circle (3pt);
\draw (-0.2, 0.2) -- (0.2, 0.2) -- (0.2, -0.2) -- (-0.2, -0.2) -- cycle;
\draw (0.8, 0.8) -- (1.2, 0.8) -- (1.2, 1.2) -- (0.8, 1.2) -- cycle; 
\draw (-0.8, -0.8) -- (-1.2, -0.8) -- (-1.2, -1.2) -- (-0.8, -1.2) -- cycle; 
\draw (0,0) circle (6pt);
\draw (1,0) circle (6pt);
\draw (-1,0) circle (6pt);
\draw (0,1) circle (6pt);
\draw (0,-1) circle (6pt);
\draw[thick, red] (-1,0) -- (1,0);
\draw[thick, red] (0,-1) -- (0,1); 
\draw[thick, blue] (-1,-1) -- (1,1);
\node[below] at (0,-1.5) {\scriptsize
$\left[ \widetilde{D}_{\mathbf{h}}^2 \right]_{ij}$};
\end{tikzpicture}
\hspace{0.5in}
\begin{tikzpicture}[scale = 0.75]
\fill (-1,0) circle (3pt);
\fill (0,0) circle (3pt);
\fill (1,0) circle (3pt);
\fill (0,-1) circle (3pt);
\fill (0,1) circle (3pt);
\fill (-1,1) circle (3pt);
\fill (-1,-1) circle (3pt);
\fill (1,-1) circle (3pt);
\fill (1,1) circle (3pt);
\draw (-1.2, 1.2) -- (-0.8, 1.2) -- (-0.8, 0.8) -- (-1.2, 0.8) -- cycle;
\draw (-0.2, 0.2) -- (0.2, 0.2) -- (0.2, -0.2) -- (-0.2, -0.2) -- cycle;
\draw (1.2, -1.2) -- (0.8, -1.2) -- (0.8, -0.8) -- (1.2, -0.8) -- cycle;
\draw (0,0) circle (6pt);
\draw (1,0) circle (6pt);
\draw (-1,0) circle (6pt);
\draw (0,1) circle (6pt);
\draw (0,-1) circle (6pt);
\draw[thick, blue] (-1,0) -- (1,0);
\draw[thick, blue] (0,-1) -- (0,1); 
\draw[thick, red] (-1,1) -- (1,-1);
\node[below] at (0,-1.5) {\scriptsize
$\left[ \widehat{D}_{\mathbf{h}}^2 \right]_{ij}$};
\end{tikzpicture}
\hspace{0.5in}
\begin{tikzpicture}[scale = 0.75]
\fill (-1,0) circle (3pt);
\fill (0,0) circle (3pt);
\fill (1,0) circle (3pt);
\fill (0,-1) circle (3pt);
\fill (0,1) circle (3pt);
\fill (-1,1) circle (3pt);
\fill (-1,-1) circle (3pt);
\fill (1,-1) circle (3pt);
\fill (1,1) circle (3pt);
\draw (-1.2, 1.2) -- (-0.8, 1.2) -- (-0.8, 0.8) -- (-1.2, 0.8) -- cycle;
\draw (1.2, -1.2) -- (0.8, -1.2) -- (0.8, -0.8) -- (1.2, -0.8) -- cycle;
\draw (0.8, 0.8) -- (1.2, 0.8) -- (1.2, 1.2) -- (0.8, 1.2) -- cycle; 
\draw (-0.8, -0.8) -- (-1.2, -0.8) -- (-1.2, -1.2) -- (-0.8, -1.2) -- cycle; 
\draw[thick, blue] (-1,-1) -- (1,1);
\draw[thick, red] (-1,1) -- (1,-1);
\node[below] at (0,-1.5) {\scriptsize
$\left[ \overline{D}_{\mathbf{h}}^2 \right]_{ij} = \left[ D_{\mathbf{h}}^2 \right]_{ij}$};
\end{tikzpicture}
\end{center}
\caption{
Illustration of the relationship between the Cartesian directions and diagonal directions 
when defining various mixed second order partial derivative approximations.  
Blue corresponds to {\em added} three-point (non-mixed) second derivative approximations 
and red corresponds to {\em subtracted} three-point (non-mixed) second derivative approximations. 
}
\label{hessian_diag_pic}
\end{figure}

\section{A narrow-stencil finite difference method} \label{FD_method_sec}

We now propose a new finite difference method for approximating viscosity solutions 
of the fully nonlinear second order problem \eqref{FD_problem}.
In the construction of such a narrow-stencil scheme, our main idea is to design a numerical operator
that utilizes the various numerical Hessians defined in the previous section so that
a relaxed (or generalized) monotonicity condition is satisfied.  
The crux of the formulation is to introduce the concept of numerical moments  
which can be considered analogous to the concept of numerical viscosities used 
in the Crandall and Lions' finite difference framework for (first order) Hamilton-Jacobi equations. 
Below, we first define our narrow-stencil finite difference method, and we then describe the idea and details 
of numerical moments. 
We also prove a generalized monotonicity 
result for our proposed numerical operator.

\subsection{Formulation of the narrow-stencil finite difference method}

In Section \ref{FD_hessian_sec} we defined four basic sided discrete Hessian operators 
$D_\mathbf{h}^{\mu\nu}$ for $\mu,\nu\in \{+,-\}$.  
These four operators are the building blocks for 
constructing our numerical operator $\widehat{F}$ 
that approximates the differential operator $F$. 
However, inspired by our earlier 1-D work \cite{Feng_Kao_Lewis13}, 
we construct the numerical operator $\widehat{F}$ so that it only depends explicitly upon 
the central difference operators $\widehat{D}_\mathbf{h}^2$ and $\widetilde{D}_\mathbf{h}^2$ 
with regard to the Hessian argument in $F$. Recall that $\cT_{\mathbf{h}}'$ denotes the
extended mesh of $\overline{\Omega}$ that includes ghost mesh points and $\NJ'$ is the corresponding 
extended index set of $\NJ$
(see Section \ref{diff_quotients_sec}). 

Our specific narrow-stencil finite difference method is defined by seeking 
a grid function $U_\alpha: \NJ'\to \mathbb{R}$ such that for all $\alpha \in \NJ$
\begin{subequations}\label{FD_method}
\begin{alignat}{2} 
\hF[U_\alpha, \bx_\alpha] & =0  &&\qquad\mbox{for } \mathbf{x}_\alpha\in\mathcal{T}_{\mathbf{h}} \cap \Omega, \label{FD_method:1}\\
U_\alpha &= g(\mathbf{x}_\alpha)  &&\qquad\mbox{for } \mathbf{x}_\alpha\in\mathcal{T}_{\mathbf{h}}\cap  \partial\Omega , \label{FD_method:2}
\end{alignat}
\end{subequations}
where
\begin{align}\label{hatF}
\hF[U_\alpha, \bx_\alpha] &\equiv \widetilde{F}\bigl(D_{\mathbf{h}}^{--}U_\alpha, 
D_{\mathbf{h}}^{-+} U_\alpha, D_{\mathbf{h}}^{+-}U_\alpha, 
D_{\mathbf{h}}^{++} U_\alpha, 
\nabla_{\mathbf{h}}^+ U_\alpha,\nabla_{\mathbf{h}}^- U_\alpha,U_\alpha,\bx_\alpha \bigr) \\
\nonumber 
&\equiv \hF(\widehat{D}_{\mathbf{h}}^2 U_\alpha, \widetilde{D}_{\mathbf{h}}^2 U_\alpha, 
	\nabla_{\mathbf{h}}^+ U_\alpha,\nabla_{\mathbf{h}}^- U_\alpha,U_\alpha,\bx_\alpha) \\
      \nonumber  &\equiv F \left( \overline{D}_{\mathbf{h}}^2 U_\alpha , \nabla_{\mathbf{h}} U_\alpha , 
			U_\alpha , \bx_\alpha \right) 
	+ A(U_\alpha, \bx_\alpha) : 
		\bigl( \widetilde{D}_{\mathbf{h}}^2 U_\alpha - \widehat{D}_{\mathbf{h}}^2 U_\alpha \bigr)  \\ 
	\nonumber &\qquad 
	 	- \vec{\beta}(U_\alpha , \bx_\alpha) \cdot 
	 \bigl( \nabla_{\mathbf{h}}^+ U_\alpha - \nabla_{\mathbf{h}}^- U_\alpha \bigr), 
\end{align}
for $A : \mathbb{R}^J \times \cT_{\bh} \to \mathbb{R}^{d \times d}$ a matrix-valued function and  
$\vec{\beta} : \mathbb{R}^J \times \cT_{\bh} \to \mathbb{R}^d$ a vector-valued function. 
Recall that $2\overline{D}_{\mathbf{h}}^2 U_\alpha
=\widetilde{D}_{\mathbf{h}}^2 U_\alpha + \widehat{D}_{\mathbf{h}}^2 U_\alpha$ and
$2\nabla_{\mathbf{h}} U_\alpha = \nabla_{\mathbf{h}}^+ U_\alpha+\nabla_{\mathbf{h}}^- U_\alpha$.
Hence, $\overline{D}_{\mathbf{h}}^2 U_\alpha$ and $\nabla_{\mathbf{h}} U_\alpha$ are not regarded
as independent variables in the function $\hF$. 
We also note that the explicit dependence of $\hF$ on 
the $U_\alpha$ and $\bx_\alpha$ arguments is inherited from $F$.  
Suitable choices for $A$ and $\vec{\beta}$ will be specified in Section~\ref{monotoncity_sec} 
to guarantee well-posedness and convergence of the scheme. 

In order for the above scheme to be well defined, the issue of how to provide the ghost values 
must be addressed. 
This is because when $\mathbf{x}_\alpha$ is next to the 
boundary $\partial\Omega$, equation \eqref{FD_method:1} uses ghost grid points which are
outside of the domain $\Omega$ when evaluating $\widetilde{D}_{\bh}^2 U_\alpha$.  
Inspired by our 1-D work \cite{Feng_Kao_Lewis13}, we propose to use 
the following auxiliary equations, which can be viewed as a discretization of an additional boundary condition,  to define the required ghost values in terms of the boundary and interior grid values:
\begin{equation} \label{FD_auxiliaryBC}
	\Delta_{\mathbf{h}} U_\alpha \equiv \sum_{i=1}^d \delta_{x_i, h_i}^2 U_\alpha = 0 
\qquad \mbox{for } \mathbf{x}_\alpha\in \mathcal{S}_{\mathbf{h}} \subset \mathcal{T}_{\mathbf{h}} \cap \partial\Omega , 
\end{equation}
where 
\begin{align}\label{Sh_grid}
	\mathcal{S}_{\bh} \equiv \big\{ \bx_\alpha \in \cT_{\bh} \cap \partial \Omega \mid & \,  
		\bx_{\alpha} + h_i \mathbf{e}_i \in \cT_{\bh} \cap \Omega \text{ or } 
		\bx_{\alpha} - h_i \mathbf{e}_i \in \cT_{\bh} \cap \Omega \\ 
		\nonumber & 
		\text{ for some } i \in \{1,2,\ldots,d\} \big\} . 
\end{align}
We refer the reader to \cite{Feng_Kao_Lewis13,Lewis_dissertation,Feng_Glowinski_Neilan10} 
for detailed explanations of why such auxiliary equations are appropriate from the PDE point of view.  
See Figure~\ref{ghost_data_fig} for an illustration of how to implement the auxiliary boundary condition 
\eqref{FD_auxiliaryBC}.

We remark that the definition of scheme \eqref{FD_method}--\eqref{FD_auxiliaryBC}
does not depend on the particular structure of the differential operator $F$.  Hence, it is defined for 
all fully nonlinear second order PDEs including HJB equations and Monge-Amp\`ere-type equations. 
We also note that the numerical operator $\hF$ defined by \eqref{hatF} 
satisfies a {\it consistency} property which says that
$\hF [ U_\alpha , \bx_\alpha] = F(P,\mathbf{q} , U_\alpha , \bx_\alpha )$ provided  
$\widetilde{D}_{\bh}^2 U_\alpha = \widehat{D}_{\bh}^2 U_\alpha = P$ 
and $\nabla_{\bh}^+ U_\alpha = \nabla_{\bh}^- U_\alpha = \mathbf{q}$.  

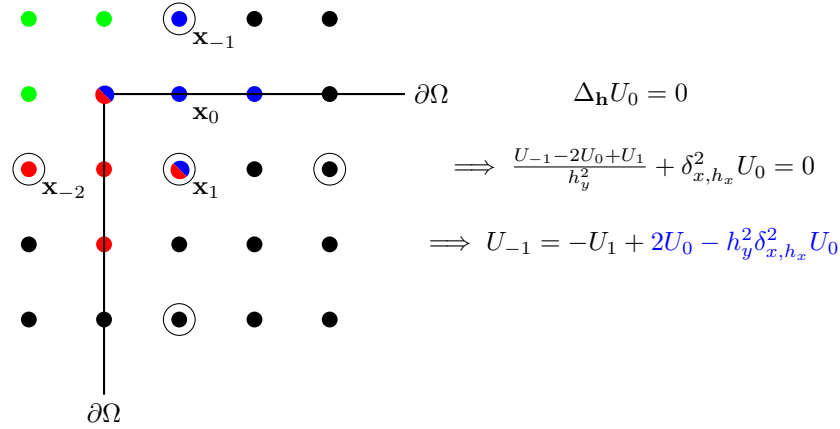
\begin{figure}[htb]
\begin{center}
\begin{tikzpicture}
\node[below] at (0,-4) {$\partial \Omega$};
\node[right] at (4,0) {$\partial \Omega$};
\fill[color=blue] (1,0) circle (3pt);
\fill[color=blue] (2,0) circle (3pt);
\fill (3,0) circle (3pt);
\fill[color=red] (0,-1) circle (3pt);
\fill[color=red] (0,-2) circle (3pt);
\fill (0,-3) circle (3pt);
\draw[thick] (0,-4) -- (0,0) -- (4,0);
\node[below right] at (1.05,-1.05) {$\mathbf{x}_1$};
\node[below right] at (1.05,-0.05) {$\mathbf{x}_0$};
\node[below right] at (1.05,0.95) {$\mathbf{x}_{-1}$};
\node[below right] at (-0.95,-1.05) {$\mathbf{x}_{-2}$};
\fill (1,-2) circle (3pt);
\fill (1,-3) circle (3pt);
\fill (2,-1) circle (3pt);
\fill (2,-2) circle (3pt);
\fill (2,-3) circle (3pt);
\fill (3,-1) circle (3pt);
\fill (3,-2) circle (3pt);
\fill (3,-3) circle (3pt);
\fill[color=blue] (1,1) circle (3pt);
\fill (2,1) circle (3pt);
\fill (3,1) circle (3pt);
\fill[color=red] (-1,-1) circle (3pt);
\fill (-1,-2) circle (3pt);
\fill (-1,-3) circle (3pt);
\fill[color=green] (-1,0) circle (3pt);
\fill[color=green] (-1,1) circle (3pt);
\fill[color=green] (0,1) circle (3pt);
\draw (1,-1) circle (6pt);
\draw (1,-3) circle (6pt);
\draw (1,1) circle (6pt);
\draw (-1,-1) circle (6pt);
\draw (3,-1) circle (6pt); 
\fill[color=blue] (-0.5,0.5) -- (0.1,-0.1) arc(-45:135:0.125) -- cycle;
\fill[color=red] (0.5,-0.5) -- (-0.08,0.08) arc(135:315:0.125) -- cycle;
\fill[color=blue] (0.5,-0.5) -- (1.1,-1.1) arc(-45:135:0.125) -- cycle;
\fill[color=red] (1.5,-1.5) -- (0.92,-0.92) arc(135:315:0.125) -- cycle;

\node at (7,0) {$\Delta_{\mathbf{h}} U_0 = 0$};
\node at (7,-1) {$\implies \frac{U_{-1} - 2 U_0 + U_1}{h_y^2} + \delta_{x,h_x}^2 U_0 = 0$};
\node at (7,-2) {$\implies U_{-1} = - U_1 + {\color{blue} 2 U_0 - h_y^2 \delta_{x,h_x}^2 U_0}$};
\end{tikzpicture}
\end{center}
\caption{
An example in two dimensions of how to enforce the auxiliary boundary condition 
when building the matrix representation of $\delta_{x_i, 2 h_i}^2$.  
The circled nodes correspond to evaluating $\Delta_{2 \mathbf{h}} U_1$.  
The blue and red ghost nodes are defined using the auxiliary boundary condition 
$\Delta_{\mathbf{h}} U_\alpha = 0$ along $\partial \Omega$.  
The blue and red nodes along $\partial \Omega$ are treated as knowns.   
Thus, we can replace the ghost nodes with unknown nodes corresponding to 
grid points in the interior of the domain as well as known values coming from the 
Dirichlet boundary condition.  
Note that the green nodes are ghost nodes that are not needed since they do not appear 
in the calculation of $\hF[U]$.  
Accordingly, the corner mesh node is not included in the construction of the set $\mathcal{S}_{\bh}$ 
over which the auxiliary boundary condition is enforced.  
}
\label{ghost_data_fig}
\end{figure}

\subsection{Numerical moments}
 
Since the last two terms in \eqref{hatF} where the numerical operator $\hF$ is defined
are special, we now take a closer look at them and explain the insights behind their introduction. 
First, a direct calculation immediately yields 
\begin{equation} \label{viscosity_h}
\delta_{x_i, h_i}^+ U_\alpha -  \delta_{x_i, h_i}^- U_\alpha  
=  h_i \, \frac{U_{\alpha-h_i\mathbf{e}_i} - 2 U_\alpha + U_{\alpha+h_i\mathbf{e}_i}}{h_i^2}
=  h_i \, \delta_{x_i, h_i}^2 U_\alpha , 
\end{equation} 
which is a central difference approximation of $u_{x_i x_i}(\mathbf{x}_\alpha)$ scaled by $h_i$.
This is the very reason that this term is called a {\em numerical viscosity} in the literature, and 
it was used in \cite{Crandall_Lions84} to construct a Lax-Friedrichs 
monotone scheme for Hamilton-Jacobi equations.  

Second, also by a direct calculation, we get for $i,j \in \{1,2,\ldots,d \}$
{\small
\begin{align} \label{moment_factor}
\bigl[ \widetilde{D}_{\mathbf{h}}^2 V_\alpha \bigr]_{ij} 
- \bigl[ \widehat{D}_{\mathbf{h}}^2 V_\alpha \bigr]_{ij}  
&=\frac12\bigl( \delta_{x_i,h_i}^+ \delta_{x_j,h_j}^+ - \delta_{x_i,h_i}^+ \delta_{x_j,h_j}^- 
- \delta_{x_i,h_i}^- \delta_{x_j,h_j}^+ + \delta_{x_i,h_i}^- \delta_{x_j,h_j}^- \bigr) V_\alpha \\ 
\nonumber &= \frac{h_i h_j}2 \delta_{x_i, h_i}^2 \delta_{x_j, h_j}^2 V_\alpha 
= \frac{h_i h_j}2 \delta_{x_j, h_j}^2 \delta_{x_i, h_i}^2 V_\alpha , 
\end{align}
}
which is an $O( h_i^2 + h_j^2)$ approximation of $v_{x_i x_i x_j x_j}(\mathbf{x}_\alpha)$ 
scaled by $h_i h_j / 2$. Thus, 
\[
2 \mathbf{1}_{d \times d} : 
\left( \widetilde{D}_{\mathbf{h}}^2 - \widehat{D}_{\mathbf{h}}^2 \right) V_\alpha
\approx h^2 \Delta^2 v(\mathbf{x}_\alpha), 
\]
where $\mathbf{1}_{d \times d}$ denotes the $d \times d$ matrix with all entries equal to $1$.

Due to the above observation and motivated by its connection to the vanishing moment method of 
\cite{Feng_Neilan09a}, we introduce the following definition for the above expression.

\begin{definition} \label{moment_def} 
Let $A : \mathbb{R}^J \times \cT_{\mathbf{h}} \to \mathbb{R}^{d \times d}$ 
and $V$ be a given grid function.  
The discrete operator $M : \mathbb{R}^J \to \mathbb{R}$ defined by 
\[
M[V,\bx_\alpha] \equiv  A \bigl( V_\alpha , \mathbf{x}_\alpha \bigr) 
: \bigl( \widetilde{D}_{\mathbf{h}}^2 V_\alpha - \widehat{D}_{\mathbf{h}}^2 V_\alpha \bigr)
\] 
for all $\bx_\alpha \in \cT_{\bh} \cap \Omega$ 
is called a {\em numerical moment operator}.  
\end{definition}

Let $M_{ij}$ denote the $(i,j)$ component of the matrix difference 
$\bigl[ \widetilde{D}_{\mathbf{h}}^2 V_\alpha \bigr] - \bigl[ \widehat{D}_{\mathbf{h}}^2 V_\alpha \bigr]$ 
and let $A_i$ and $A_j$ denote the matrix representations of the difference quotients   
$- \delta_{x_i, h_i}^2 V_\alpha$ and $-\delta_{x_j, h_j}^2 V_\alpha$ with Dirichlet boundary conditions, respectively. 
Then, 
\[
	M_{ij} = \frac12 h_i h_j A_i A_j = \frac12 h_i h_j A_j A_i . 
\]
Since $A_i$ and $A_j$ are symmetric positive definite, it follows that $M_{ij}$ is symmetric positive definite.  
Thus, the numerical moment is a strictly positive operator.  
Similarly, by \eqref{viscosity_h}, we have the numerical viscosity is also a strictly positive operator.  

\subsection{Generalized monotonicity properties of the numerical operator $\hF$} \label{monotoncity_sec}

In the definition of $\hF$, no restriction or guideline is given for the choices of 
the matrix-valued function $A$ and the vector-valued function $\vec{\beta}$. 
Suggested by the admissibility and stability proofs to be given in subsequent sections,  
we consider a special family of the pairs $A$ and $\vec{\beta}$, that is, 
$A = \gamma \mathbf{1}_{d \times d}$ for some constant $\gamma > 0$ 
and $\vec{\beta} = \beta \vec{1}$ for some constant $\beta \geq 0$, 
where $\mathbf{1}_{d \times d}$ and $\vec{1}$ denote the matrix and the vector 
with all entries equal to one, respectively. 
Thus, the numerical operator is a linear perturbation of the nonlinear operator $F$. 
We note that the particular choice for $\gamma$ and $\beta$ depends on the Lipschitz constants of
the differential operator $F$ which we now explain. For the ease of presentation we 
only present the details for the case that $F$ is differentiable with respect to its first three arguments  
and make a comment  for the Lipschitz continuous case at the end. 

Since $F$ is proper elliptic and differentiable, there exist constants $K_{ii} \geq k_{ii} > 0$, 
$k_0 \geq 0$, 
and $K_{ij}, K_i, K^0 \geq 0$ such that 
\begin{alignat*}{2}
	& -K_{ii} \leq \frac{\partial F}{\partial u_{x_i x_i}} \leq - k_{ii}, \qquad 
		&& k_0 \leq \frac{\partial F}{\partial u} \leq K^0 , \\ 
	& \left| \frac{\partial F}{\partial u_{x_i x_j}} \right| \leq K_{ij} , \qquad 
		&& \left| \frac{\partial F}{\partial u_{x_i}} \right| \leq K_{i} 
\end{alignat*}
for all $i,j$.  
Let $k_{**} = \min \{ k_{ii} \}$, $K^{**} = \max \left\{ K_{ij} \right\}$, 
and $K^* = \max \left\{ K_i \right\}$.  

By the definition of $\overline{D}_{\mathbf{h}}^2 U_\alpha$ and $\nabla_{\mathbf{h}} U_\alpha$, 
we get 
\begin{subequations} \label{partials}
\begin{align}
	\frac{\partial \hF}{\partial \widetilde{D}_{\bh}^2 U_\alpha} 
	& = \frac12 \frac{\partial F}{D^2 u} + \gamma \mathbf{1}_{d \times d}
		\geq \left(\gamma - \frac12 K^{**} \right) \mathbf{1}_{d \times d} , \\ 
	\frac{\partial \hF}{\partial \widehat{D}_{\bh}^2 U_\alpha} 
	& = \frac12 \frac{\partial F}{D^2 u} - \gamma \mathbf{1}_{d \times d}
		\leq \left(\frac12 K^{**} - \gamma \right) \mathbf{1}_{d \times d} , \\ 
	\frac{\partial \hF}{\partial \nabla_{\bh}^+ U_\alpha} 
	& = \frac12 \frac{\partial F}{\partial \nabla u} - \beta \vec{1}
		 \leq \left( \frac12 K^* - \beta \right) \vec{1} , \\ 
	\frac{\partial \hF}{\partial \nabla_{\bh}^- U_\alpha} 
	& = \frac12 \frac{\partial F}{\partial \nabla u} + \beta \vec{1}
		\geq \left( \beta - \frac12 K^* \right) \vec{1} , \\ 
	\frac{\partial \hF}{\partial U_\alpha} 
	& = \frac{\partial F}{\partial u} 
		\geq k_0  , 
\end{align} 
\end{subequations}
where the inequalities hold component-wise.  
Clearly, for sufficiently large $\gamma$ and $\beta$, every partial derivative  
in \eqref{partials} has a fixed sign as reflected in the following lemma.  

\begin{lemma} \label{g-mon-lemma}
Assume that $F$ is proper elliptic and has bounded derivatives with respect to 
its first three arguments for all $x\in \Omega$.  Then for sufficiently 
large $\gamma > 0$ and $\beta \geq 0$, the numerical operator
$\hF$ is nonincreasing with respect to each component of $\widehat{D}_{\bh}^2 U_\alpha$ 
and $\nabla_{\bh}^+ U_\alpha$ and nondecreasing with respect to each component of 
$\widetilde{D}_{\bh}^2 U_\alpha$, $\nabla_{\bh}^- U_\alpha$, 
and $U_\alpha$. 
\end{lemma}

\begin{proof}
The assertion holds by \eqref{partials} if $\gamma$ and $\beta$ are chosen such that 
$\gamma \geq \frac{K^{**}}{2}$ and $\beta \geq \frac{K^*}{2}$.  The proof is complete. 
\end{proof}

\begin{remark}
(a) It is easy to see that the conclusion of Lemma \ref{g-mon-lemma} still holds if
	$F$ is proper elliptic and Lipschitz continuous  with respect to 
	its first three arguments for all $x\in \Omega$. 
	This is because the inequalities 
	in \eqref{partials} still hold if all of the partial derivatives are replaced 
	by their respective difference quotients, which is sufficient to ensure the 
	monotonicity stated in the lemma.  

(b) It should be noted that the monotonicity proved above in Lemma \ref{g-mon-lemma} is nonstandard.  First, it
    holds only with respect to the various zero-, first-, and second-order difference operators. 
    Second, the monotonicity with respect to each matrix or vector argument holds component-wise;
    in other words, in a component-wise ordering for these matrix or vector arguments. 
    
(c) We shall refer to the monotonicity stated in Lemma \ref{g-mon-lemma} for $\hF$ 
     as {\em generalized monotonicity} conditions throughout the rest of the paper.  
     We also note that, for $h$ sufficiently small, we can choose $\beta = 0$ 
     when $k_{**} > 0$. i.e., $F$ is not degenerate elliptic.  

\end{remark}
  
\section{$\ell^2$-norm stability and well-posedness}\label{stability_admissibility_sec}

The goal of this section is to show that the proposed narrow stencil scheme 
as defined in Section~\ref{monotoncity_sec} with the choices of the parameters 
\begin{equation} \label{Abeta}
	A=\gamma \mathbf{1}_{d \times d}, \qquad \vec{\beta} = \beta \vec{1} 
\end{equation} 
has a unique solution for sufficiently large $\gamma, \beta \geq 0$ 
that is uniformly bounded in a weighted $\ell^2$ norm.
For transparency, we consider the second-order PDE problem 
\begin{subequations} \label{F_D2}
\begin{alignat}{2}
	F[u](\mathbf{x}) \equiv F \left( D^2 u, u , \mathbf{x} \right) 
	& = 0 &&\qquad\forall \mathbf{x} \in \Omega \subset \mathbb{R}^d, \\
	u(\mathbf{x}) & = g(\mathbf{x}) && \qquad\forall \mathbf{x} \in \partial \Omega.
\end{alignat} 
\end{subequations}
We assume that $F$ is degenerate elliptic (see Definition \ref{elliptic}), 
differentiable, and Lipschitz with respect to the first two arguments. 
Since $F$ is independent of 
$\nabla u$ in \eqref{F_D2}, we choose $\beta = 0$. 
We note that the results for this 
special case can be extended to the more general PDE \eqref{FD_problem}.

The idea for proving the well-posedness and stability of the method is to equivalently reformulate  the proposed scheme as a fixed point problem for a nonlinear mapping and to prove the mapping is 
contractive in the $\ell^2$-norm. The well-posedness of the scheme
 then follows from 
the Contractive Mapping Theorem and the weighted $\ell^2$-norm stability will be  
obtained as a consequence of the contractive property.
To this end, let $S(\cT_{\mathbf{h}}^\prime)$
denote the space of all grid functions on $\cT_{\mathbf{h}}^\prime$, and introduce the 
mapping $\cM_\rho : S(\cT_{\mathbf{h}}^\prime) \to S(\cT_{\mathbf{h}}^\prime)$ defined by 
\begin{equation}\label{M_rho}
	\widehat{U}
	\equiv \cM_\rho U, 
\end{equation}
where the grid function $\widehat{U}\in S(\cT_{\mathbf{h}}')$ is defined by 
\begin{subequations} \label{M_rho_matrix}
\begin{alignat}{2} \label{M_rho_interior} 
	 \widehat{U}_\alpha &= U_\alpha - \rho \hF \left[ U_\alpha , \bx_\alpha \right] , \qquad 
		&& \text{if } \bx_\alpha \in \cT_{\bh} \cap \Omega , \\ 
	 \widehat{U}_\alpha &= g(\bx_\alpha) , && \text{if } \bx_\alpha \in \cT_{\bh} \cap \partial \Omega , \\ 
	 \Delta_{\bh} \widehat{U}_\alpha &= 0 , && \text{if } \bx_\alpha \in \mathcal{S}_{\bh} \subset \cT_{\bh} \cap \partial \Omega 
		\label{M_rho_matrix_bc2}
\end{alignat}
\end{subequations} 
for $\rho > 0$ an underdetermined constant.  
Clearly, the iteration defined in \eqref{M_rho_matrix}
is the standard forward Euler method with pseudo time-step $\rho$. 
In the following, we will show that the mapping $\cM_\rho$ is a contraction 
with respect to the $\ell^2$ norm when 
$\gamma > 0$ is sufficiently large 
and $\rho > 0$ is sufficiently small. 
We note that this is in contrast to the corresponding analysis for monotone methods 
that shows the same fixed point iteration is a contraction with respect to the $\ell^\infty$ norm.  

A couple of remarks are needed here. 
First, the auxiliary boundary condition \eqref{M_rho_matrix_bc2} is used to extend 
$\widehat{U}$ to a layer of ghost points that are needed in the evaluation of $\hF$ in \eqref{M_rho_interior}.  
Note that the boundary conditions in \eqref{M_rho_matrix} are consistent with those in 
\eqref{FD_method}--\eqref{FD_auxiliaryBC}.  
Second, any fixed point of 
$\cM_\rho$ is a solution to the proposed finite difference scheme 
\eqref{FD_method}--\eqref{FD_auxiliaryBC} and vice versa.  

To show that the mapping $\cM_\rho$ has a unique fixed point in $S(\cT_{\mathbf{h}}^\prime)$, 
we first establish a lemma that specifies conditions under which $\cM_\rho$ is a contraction in $\ell^2$.  
The proof relies upon a couple of auxiliary linear algebra results that can be found in Section~\ref{appendix_sec}.  

\begin{lemma} \label{lemma_contraction} 
Suppose the operator $F$ in \eqref{F_D2} is proper elliptic, differentiable, and 
Lipschitz continuous with respect to its first two arguments.  
Choose $U,V\in S(\cT_{\mathbf{h}}')$, 
and let $\widehat{U} = \cM_\rho U$ and $\widehat{V} = \cM_\rho V$ 
for $\cM_\rho$ defined by \eqref{M_rho} and \eqref{M_rho_matrix} with 
$\gamma > K^{**}/2$ and $\beta = 0$. 
Then there holds 
\[
	\| \widehat{U} - \widehat{V} \|_{\ell^2(\cT_{\mathbf{h}})} 
	\leq (1 - \rho k_{**} c d - \rho k_0 d - \rho \gamma c h_*^2 ) \| U - V \|_{\ell^2(\cT_{\mathbf{h}})} 
\] 
for all $\rho > 0$ sufficiently small and some constant $c$ independent of $h$ and $h_* = \min_i \{ h_i \}$.  
\end{lemma}

\begin{proof}
Let $W\equiv V-U$ and $\widehat{W}\equiv \widehat{V}-\widehat{U}$. 
By the Mean Value Theorem and \eqref{partials}, there holds  
{\small 
\begin{align} \label{V_U} 
	&\widehat{W}_\alpha 
	= W_\alpha - \rho \left( \hF[V_\alpha, \bx_\alpha] - \hF[U_\alpha, \bx_\alpha] \right) \\ 
	\nonumber &\, = W_\alpha 
		- \rho \left( \frac12 \frac{\partial F}{\partial D^2 u} + \gamma \mathbf{1}_{d \times d} \right) 
			: \widetilde{D}_{\bh}^2 W_\alpha 
		- \rho \left( \frac12 \frac{\partial F}{\partial D^2 u} - \gamma \mathbf{1}_{d \times d} \right) 
			: \widehat{D}_{\bh}^2 W_\alpha 
		- \rho \frac{\partial F}{\partial u} W_\alpha \\ 
	\nonumber &\, = \left( 1 - \rho \frac{\partial F}{\partial u} \right) W_\alpha 
		- \rho \sum_{i=1}^d \frac{\partial F}{\partial u_{x_i x_i}} \delta_{x_i, 2 h_i}^2 W_\alpha \\ 
		\nonumber & \quad 
		- \rho \gamma \sum_{i=1}^d \left( \left[ \widetilde{D}_{\bh}^2 W_\alpha \right]_{ii} 
				- \left[ \widehat{D}_{\bh}^2 W_\alpha \right]_{ii} \right) \\ 
			\nonumber & \quad 
		- \rho \sum_{i=1}^d \sum_{j=1 \atop j \neq i}^d \left( \frac12 \frac{\partial F}{\partial u_{x_i x_j}} + \gamma \right) 
			\left[ \widetilde{D}_{\bh}^2 W_\alpha \right]_{ij} 
		- \rho \sum_{i=1}^d \sum_{j=1 \atop j \neq i}^d \left( \frac12 \frac{\partial F}{\partial u_{x_i x_j}} - \gamma \right) 
			\left[ \widehat{D}_{\bh}^2 W_\alpha \right]_{ij} \\ 
	\nonumber &\, \equiv G_\alpha W_\alpha 
\end{align} 
}
where $G_\alpha$ is a linear operator which depends on $V_\alpha$ and $U_\alpha$.  

Observe that the boundary conditions can naturally be eliminated since 
$W_\alpha = \widehat{W}_\alpha = 0$ for all $\mathbf{x}_\alpha \in \cT_{\bh} \cap \partial \Omega$
and the auxiliary boundary condition implies that, for each 
$\mathbf{x}_\alpha \in \mathcal{S}_{\bh} \subset \cT_{\bh} \cap \partial \Omega$, 
there exists an index $i$ such that 
$W_{\alpha + \mathbf{e}_i} = - W_{\alpha - \mathbf{e}_i}$ and 
$\widehat{W}_{\alpha + \mathbf{e}_i} = - \widehat{W}_{\alpha - \mathbf{e}_i}$. 
Thus, the problem can naturally be reformulated on the interior grid $\cT_{\bh} \cap \Omega$.  

Next, we are going to write the above mapping (on grid functions) as an 
equivalent matrix transformation (on vectors). To this end, let  $J_0 = | \cT_{\bh} \cap \Omega |$ and 
$\widehat{\mathbf{W}} , \mathbf{W} \in \mathbb{R}^{J_0}$ denote the vectorization of the 
grid functions $\widehat{W}$ and $W$ restricted to $\cT_{\bh} \cap \Omega$, respectively.   
We first introduce a series of symmetric positive definite matrices that correspond to the various central difference 
operators appearing in \eqref{V_U}.  
Let $A_i$ denote the matrix representation of $-\delta_{x_i, 2 h_i}^2 W_\alpha$ and $\lambda_i$ be  its smallest eigenvalue. 
Define $\lambda_0 = \min_i \{ \lambda_i \}$.  
Then, 
\begin{align*}
	A_i  \geq \lambda_{i} I \geq \lambda_0 I . 
\end{align*}
Let $A_{ii}$ denote the matrix representation of 
$[ \widetilde{D}_{\bh}^2 W_\alpha ]_{ii} - [ \widehat{D}_{\bh}^2 W_\alpha ]_{ii}$. 
By \eqref{moment_factor} and the choice of the auxiliary boundary condition restricted to the boundary nodes in 
$\mathcal{S}_{\bh}$, we have $A_{ii}$ is symmetric positive definite 
with minimal eigenvalue bounded below by $c h_i^2$ for some constant $c$ independent of $h$ 
using the fact that the minimal eigenvalue of the matrix representation of $-\delta_{x_i, h_i}^2$ 
with Dirichlet boundary conditions 
is a positive constant independent of $h$.  
Thus, 
\[
	A_{ii} \geq c h_*^2 I . 
\]
Let $\widetilde{A}_{ij}$ denote the matrix representation of the operator 
$[ \widetilde{D}_{\bh}^2 W_\alpha ]_{ij}$ 
and $\widehat{A}_{ij}$ denote the matrix representation of the operator 
$[ - \widehat{D}_{\bh}^2 W_\alpha ]_{ij}$. 
Then, by \eqref{hessian_diag}, 
\[
	\widetilde{A}_{ij} = \frac{1}{2 h_i h_j} \left( B_i + B_j - B_{\eta, ij} \right) , \qquad 
	\widehat{A}_{ij} = - \frac{1}{2 h_i h_j} \left( - B_i - B_j + B_{\xi, ij} \right)
\]
for $B_i$ being the matrix representation of 
$2 W_\alpha - W_{\alpha + \mathbf{e}_i} - W_{\alpha - \mathbf{e}_i}$, 
$B_j$ the matrix representation of 
$2 W_\alpha - W_{\alpha + \mathbf{e}_j} - W_{\alpha - \mathbf{e}_j}$, 
$B_{\eta, ij}$ the matrix representation of 
$2 W_\alpha - W_{\alpha - \mathbf{e}_i - \mathbf{e}_j} - W_{\alpha + \mathbf{e}_i + \mathbf{e}_j}$, 
and $B_{\xi, ij}$ the matrix representation of 
$2 W_\alpha - W_{\alpha - \mathbf{e}_i + \mathbf{e}_j} - W_{\alpha + \mathbf{e}_i - \mathbf{e}_j}$.  
Then, $B_i$, $B_j$, $B_{\eta,ij}$, and $B_{\xi,ij}$ are all symmetric positive definite 
and 
\[
	B_i + B_j = X^{-1} ( B_{\eta,ij} + B_{\xi,ij} ) X 
\]
for a nonsingular matrix $X$ corresponding to the coordinate transformation 
that maps $x_i \to \eta_i^j$ and $x_j \to \xi_i^j$.  
Thus, the eigenvalues of $B_i + B_j$ and $B_{\eta,ij} + B_{\xi,ij}$ are the same.  
Since both matrices are normal, we have there exists an orthogonal matrix $Q$ such that 
\[
	B_i + B_j = Q^T ( B_{\eta,ij} + B_{\xi,ij} ) Q > Q^T B Q 
\]
for $B = B_{\eta,ij}, B_{\xi,ij}$, 
and it follows by Lemma~\ref{SPD_ordering_lemma} that 
\[
	B_i + B_j > B_{\eta,ij} , \qquad B_i + B_j > B_{\xi,ij} . 
\]
Therefore, $\widetilde{A}_{ij}$ and $\widehat{A}_{ij}$ are symmetric positive definite.  

We next introduce several diagonal matrices that correspond to the various coefficients 
for the central difference operators appearing in \eqref{V_U}.  
Let $F_0$ denote the diagonal matrix 
corresponding to the nodal values of $\frac{\partial F}{\partial u}$,  
$F_{i}$ denote the matrix corresponding to 
$-\frac{\partial F}{\partial u_{x_i x_i}}$, 
$\widetilde{F}_{ij}$ denote the matrix corresponding to  
$\frac12 \frac{\partial F}{\partial u_{x_i x_j}} + \gamma$,  
and $\widehat{F}_{ij}$ denote the matrix corresponding to  
$\gamma - \frac12 \frac{\partial F}{\partial u_{x_i x_j}}$. 
Then, 
\begin{align*}
	F_0  \geq k_0 I , \qquad
	F_i  \geq k_{**} I , \qquad 
	\widetilde{F}_{ij} \geq \left( \gamma - K^{**}/2 \right) I , \qquad 
	\widehat{F}_{ij} \geq \left( \gamma - K^{**}/2 \right) I . 
\end{align*}

Let $N = 1 + 2d + 2d(d-1)$. 
Putting everything together, we have \eqref{V_U} can be equivalently written as
\begin{align} \label{V_U_matrix}
\widehat{\mathbf{W}}  =G \mathbf{W},
\end{align}
where 
\begin{align*}
G& = I - \rho F_0 - \rho  \sum_{i=1}^d F_i A_i
- \rho \gamma \sum_{i=1}^d A_{ii}  
- \rho \sum_{i=1}^d \sum_{j=1 \atop j \neq i}^d \left( \widetilde{F}_{ij} \widetilde{A}_{ij} + 	\widehat{F}_{ij} \widehat{A}_{ij}  \right) \\ 
\nonumber 
& = \frac12 I+ \left( \frac{1}{2N} I - \rho F_0 \right) 
+ \sum_{i=1}^d \left( \frac{1}{4N} I - \rho \frac{k_{**}}{2} A_i \right)  \\ 
& \qquad \nonumber 
+ \sum_{i=1}^d \left[ \frac{1}{4N} I - \rho \left(F_i - \frac{k_{**}}{2} I \right) A_i \right]
+ \sum_{i=1}^d \left( \frac{1}{2N} I - \rho \gamma A_{ii} \right) \\ 
& \qquad \nonumber 
+ \sum_{i=1}^d \sum_{j=1 \atop j \neq i}^d \left( \frac{1}{2N} I - \rho \widetilde{F}_{ij} \widetilde{A}_{ij} \right) 
+ \sum_{i=1}^d \sum_{j=1 \atop j \neq i}^d \left( \frac{1}{2N} I - \rho \widehat{F}_{ij} \widehat{A}_{ij} \right).
\end{align*}
Notice that the $F$ matrices are symmetric positive definite since they are all diagonal with strictly positive entries 
and the $A$ matrices are all symmetric positive definite.  
Thus, by Lemma~\ref{lemma_symmetrization}, for $\rho > 0$ sufficiently small, there holds 
\begin{align} \label{V_U_matrix2}
\| G \|_2 
& \leq \frac12 + \left\| \frac{1}{2N} I - \rho F_0 \right\|_2 
+ \sum_{i=1}^d \left\| \frac{1}{4N} I - \rho \frac{k_{**}}{2} A_i \right\|_2 \\ 
& \qquad \nonumber 
+ \sum_{i=1}^d \left\| \frac{1}{4N} I - \rho \left(F_i - \frac{k_{**}}{2} I \right) A_i \right\|_2
+ \sum_{i=1}^d \left\| \frac{1}{2N} I - \rho \gamma A_{ii} \right\|_2 \\ 
& \qquad \nonumber 
+ \sum_{i=1}^d \sum_{j=1 \atop j \neq i}^d \left\| \frac{1}{2N} I - \rho \widetilde{F}_{ij} \widetilde{A}_{ij} \right\|_2  
+ \sum_{i=1}^d \sum_{j=1 \atop j \neq i}^d \left\| \frac{1}{2N} I - \rho \widehat{F}_{ij} \widehat{A}_{ij} \right\|_2 \\ 
& \nonumber \leq \frac12 - \rho k_0 - \rho \frac{d k_{**} \lambda_0}{2} - \rho d \gamma c h_*^2 + N \frac{1}{2N} \\ 
& \nonumber = 1 - \rho k_0 - \rho \frac{d k_{**} \lambda_0}{2} - \rho d \gamma c h_*^2,
\end{align}
which, when combined with the inequality $\|\widehat{\bf{W}}\|_2\leq \|G\|_2 \| \mathbf{W} \|_2$, 
yields the desired inequality. The proof is complete.  
\hfill 
\end{proof}


As a corollary to Lemma~\ref{lemma_contraction}, we immediately 
have the following well-posedness result for our finite difference scheme.  

\begin{theorem} \label{admissible_prop}
Suppose the operator $F$ in \eqref{F_D2} 
is proper elliptic and Lipschitz continuous with respect to its first two arguments. 
The scheme defined by 
\eqref{FD_method}--\eqref{FD_auxiliaryBC} and \eqref{Abeta} for problem \eqref{F_D2} 
with $\gamma > K^{**}/2$ and $\beta = 0$ 
has a unique solution.  
\end{theorem}

\begin{proof} 
By Lemma~\ref{lemma_contraction}, we have there exists a value $\rho > 0$ such that 
the operator $\mathcal{M}_\rho$ is a contraction in its matrix form.  
By the Contractive Mapping Theorem we conclude that $\cM_\rho$ has a unique fixed point $U\in S(\cT_{\bh}^\prime)$.  
Thus, there is a unique solution to the FD scheme 
\eqref{FD_method}--\eqref{FD_auxiliaryBC} and \eqref{Abeta}
in the space $S(\cT_{\bh}^\prime)$ by the equivalence of the fixed-point problem. The proof is complete. 
\hfill
\end{proof}

\begin{remark} 
We emphasize that the scheme has a unique solution even for degenerate
 problems with $k_{**} = 0$ and $k_0 = 0$.  
Furthermore, if $F$ is independent of $u_{x_i x_j}$ for all $i \neq j$, 
then we can choose $\gamma = 0$. 

\end{remark}


We now derive an $\ell^2$-norm stability estimate for the proposed narrow-stencil finite difference scheme 
when the PDE operator $F$ is proper elliptic.  

\begin{theorem} \label{prop_stability} 
Suppose the operator $F$ in \eqref{F_D2} 
is proper elliptic and Lipschitz continuous with respect to its first two arguments 
and $k_0 > 0$ or $k_{**} > 0$,  
where $k_0$ denotes the lower ellipticity constant with respect to the zeroth order variable 
and $k_{**}$ denotes the lower ellipticity constant with respect to the second order variable.  
The finite difference scheme defined by 
\eqref{FD_method}--\eqref{FD_auxiliaryBC} and \eqref{Abeta} for problem \eqref{F_D2} 
with $\hF$ defined by \eqref{hatF} 
with $\gamma > K^{**}/2$ and $\beta = 0$ 
is $\ell^2$-norm stable when $F(0,0,\cdot) \in C^0(\overline{\Omega})$ and 
$g \in C^0(\partial \Omega)$ 
in the sense that the unique solution $U$ of the scheme satisfies  
\[
	\Bigl( \prod_{i=1,2,\ldots,d} h_i^{\frac12} \Bigr) \left\| U \right\|_{\ell^2(\cT_{\mathbf{h}} \cap \Omega)} 
	+ \left\| U \right\|_{\ell^\infty(\cT_{\bh} \cap \partial \Omega)}
	\leq C , 
\]
where $C$ is a positive constant independent of $\mathbf{h}$ 
that depends on $\Omega$, the lower (proper) ellipticity constants $k_0$ and $k_{**}$, 
$\| F(0,0,\cdot) \|_{C^0(\overline{\Omega})}$,   
and $\| g \|_{C^0(\partial \Omega)}$. 
\end{theorem}

\begin{proof}
Let $\cM_\rho$ be the mapping defined by \eqref{M_rho} 
and $\rho > 0$ be sufficiently small to ensure Lemma~\ref{lemma_contraction} holds.   
Define $\theta \equiv 1 - \rho k_{**} c - \rho k_0 \in (0,1)$.  
Then, by Lemma~\ref{lemma_contraction} and Theorem~\ref{admissible_prop}, 
for $U$ the solution to the proposed FD scheme and $V \in S(\cT_{\mathbf{h}}^\prime)$, 
there holds 
\begin{align*}
	\| U \|_{\ell^2(\cT_{\mathbf{h}} \cap \Omega)}  
	& = \| \widehat{U} \|_{\ell^2(\cT_{\mathbf{h}} \cap \Omega)} \\ 
	& \leq \| \widehat{U} - \widehat{V} \|_{\ell^2(\cT_{\mathbf{h}} \cap \Omega)} 
		+ \| \widehat{V} \|_{\ell^2(\cT_{\mathbf{h}} \cap \Omega)} \\ 
	& \leq \theta \| U - V \|_{\ell^2(\cT_{\mathbf{h}} \cap \Omega)} 
		+ \| \widehat{V} \|_{\ell^2(\cT_{\mathbf{h}} \cap \Omega)} \\ 
	& \leq \theta \| U \|_{\ell^2(\cT_{\mathbf{h}} \cap \Omega)} 
		+ \theta \| V \|_{\ell^2(\cT_{\mathbf{h}} \cap \Omega)} 
		+ \| \widehat{V} \|_{\ell^2(\cT_{\mathbf{h}} \cap \Omega)} , 
\end{align*}
where the restriction to $\cT_{\mathbf{h}} \cap \Omega$ follows from the fact that 
we can assume $U_\alpha = V_\alpha = 0$ on $\cT_{\mathbf{h}} \cap \partial \Omega$ 
without a loss of generality in the proof of Lemma~\ref{lemma_contraction}.  Thus, 
\begin{equation} 
	\left\| U \right\|_{\ell^2(\cT_{\mathbf{h}} \cap \Omega)} 
	\leq \frac{\theta}{1-\theta} \Bigl( \| V \|_{\ell^2(\cT_{\mathbf{h}} \cap \Omega)} 
	+ \frac{1}{\theta} \| \widehat{V} \|_{\ell^2(\cT_{\mathbf{h}} \cap \Omega)} \Bigr) . 
\end{equation}

Choose $V = 0 \in S(\cT_{\mathbf{h}}')$.  
Then, $\| 0 \|_{\ell^2(\cT_{\mathbf{h}} \cap \Omega)} = 0$, and we have 
\begin{equation} \label{ubound0}
	\left\| U \right\|_{\ell^2(\cT_{\mathbf{h}} \cap \Omega)} \leq 
	\frac{1}{1 - \theta} \| \widehat{0} \|_{\ell^2(\cT_{\mathbf{h}} \cap \Omega)} . 
\end{equation} 
Thus, the result holds if we can uniformly bound $\frac{1}{1-\theta} \| \widehat{0} \|_{\ell^2(\cT_{\mathbf{h}} \cap \Omega)}$.  

By the definition of $\cM_\rho$ in \eqref{M_rho_matrix}, 
we have $\widehat{0}$ is the solution to 
\begin{subequations} 
\begin{alignat}{2} 
	\widehat{0}_\alpha
	& =  - \rho F (0,0,\mathbf{x}_\alpha)  \qquad  
		&& \text{for all } \mathbf{x}_\alpha \in \Omega , \\ 
	\widehat{0}_\alpha 
	& = g(\bx_\alpha) 
		&& \text{for all } \mathbf{x}_\alpha \in \partial \Omega . 
\end{alignat} 
\end{subequations} 
Thus, 
\[
	\| \widehat{0} \|_{\ell^2(\cT_{\mathbf{h}} \cap \Omega)} 
	\leq 
	\rho \| F(0,0,\cdot) \|_{\ell^2(\cT_{\mathbf{h}} \cap \Omega)} . 
\] 
Plugging this into \eqref{ubound0}, it follows that 
\[
	\left\| U \right\|_{\ell^2(\cT_{\mathbf{h}} \cap \Omega)} \leq 
	\frac{1}{k_{**} c + k_0} \| F (0,0,\cdot) \|_{\ell^2(\cT_{\mathbf{h}} \cap \Omega)} . 
\]
Thus, 
\begin{align*}
	\Bigl( \prod_{i=1,2,\ldots,d} h_i^{\frac12} \Bigr) \bigl\| U \bigr\|_{\ell^2(\cT_{\mathbf{h}} \cap \Omega)} 
	& \leq \Bigl( \prod_{i=1,2,\ldots,d} h_i^{\frac12} \Bigr) \frac{1}{k_{**} c + k_0} \| F (0,0,\cdot) \|_{\ell^2(\cT_{\mathbf{h}} \cap \Omega)} \\ 
	& \leq \frac{| \Omega |^d}{k_{**} c + k_0} \| F(0,0,\cdot) \|_{C^0(\overline{\Omega})} , 
\end{align*}
and the result follows 
from the fact that $U_\alpha = g(\bx_\alpha)$ for all $\bx_\alpha \in \cT_{\bh} \cap \partial \Omega$ 
with $\| g \|_{C^0(\partial \Omega)} < \infty$.  
The proof is complete. 
\hfill
\end{proof}

\begin{remark}
Note that the weighting $\bigl( \prod_{i=1,2,\ldots,d} h_i^{\frac12} \bigr)$ comes from the fact that 
we are bounding $U$ in $\ell^2$.  
The weighting is consistent with using the $L^2$-norm for $F(0,0,\cdot)$ in the limit as $h \to 0$.  
\end{remark}


\section{$\ell^2$-norm stability for discrete second order derivatives and $\ell^\infty$-norm stability} \label{H2-stability_sec}
 
In this section we derive an $\ell^\infty$-norm stability estimate for solutions of the proposed finite difference method  
using a novel discrete Sobolev embedding estimate that requires first showing $\ell^2$-norm stability 
for discrete second order derivatives of the solution to the finite difference method.  

\subsection{$\ell^2$-norm stability for discrete second order derivatives}\label{sec-5.1}
In this subsection we derive uniform $\ell^2$-norm estimates for the discrete 
second order derivatives
$\delta_{x_i, 2h_i}^2 U$ by using another fixed-point mapping 
based on a Sobolev iteration technique.
Again for the ease of presentation we consider the parameters 
\[
	A=\gamma \mathbf{1}_{d \times d}, \qquad \vec{\beta} = \beta \vec{1} 
\] 
for $\gamma, \beta \geq 0$ sufficiently large and the PDE problem \eqref{F_D2}. 

Consider the mapping 
$\cM_{1,\rho} : S(\cT_{\mathbf{h}}^\prime) \to S(\cT_{\mathbf{h}}^\prime)$ defined by 
\begin{equation}\label{M21_rho}
	\widehat{U}
	\equiv \cM_{1,\rho} U, 
\end{equation}
where the grid function $\widehat{U}\in S(\cT_{\mathbf{h}}^\prime)$ is defined by 
\begin{subequations} \label{M21_rho_matrix}
\begin{alignat}{2} \label{M21_rho_interior} 
	 - \delta_{x_1, 2h_1}^2 \widehat{U}_\alpha &= - \delta_{x_1, 2h_1}^2 U_\alpha 
		- \rho \hF \left[ U_\alpha , \bx_\alpha \right] , \qquad 
		&& \text{if } \bx_\alpha \in \cT_{\bh} \cap \Omega , \\ 
	 \widehat{U}_\alpha &= g(\bx_\alpha) , && \text{if } \bx_\alpha \in \cT_{\bh} \cap \partial \Omega , 
		\label{M21_rho_matrix_bc1} \\ 
	 \Delta_{\bh} \widehat{U}_\alpha &= 0 , && \text{if } \bx_\alpha \in \mathcal{S}_{\bh} \subset \cT_{\bh} \cap \partial \Omega 
		\label{M21_rho_matrix_bc2}
\end{alignat}
\end{subequations} 
for $\rho > 0$ constant.  Note that any fixed point of 
$\cM_{1,\rho}$ is a solution to the proposed finite difference scheme 
\eqref{FD_method}--\eqref{FD_auxiliaryBC} and vice versa.

\begin{lemma} \label{lemma_contraction2} 
Suppose the operator $F$ in \eqref{F_D2} is proper elliptic, differentiable, and 
Lipschitz continuous with respect to its first two arguments.  
Choose $U,V\in S(\cT_{\mathbf{h}}')$, 
and let $\widehat{U} = \cM_{1,\rho} U$ and $\widehat{V} = \cM_{1,\rho} V$ 
for $\cM_{1,\rho}$ defined by \eqref{M21_rho} and \eqref{M21_rho_matrix} with 
$\hF$ defined by \eqref{hatF}, $\gamma > K^{**}/2$, and $\beta = 0$. 
Then there holds 
\[
	\Big\| \delta_{x_1, 2h_1}^2 \big( \widehat{U} - \widehat{V} \big) \Big\|_{\ell^2(\cT_{\mathbf{h}} \cap \Omega)} 
	\leq (1 - \rho k_{**}) 
		\Big\| \delta_{x_1, 2h_1}^2 \left( U - V \right) \Big\|_{\ell^2(\cT_{\mathbf{h}} \cap \Omega)} 
\] 
for all $\rho > 0$ with $\rho$ sufficiently small.  
\end{lemma}

\begin{proof}
Let $W\equiv V-U$ and $\widehat{W}\equiv \widehat{V}-\widehat{U}$.  
Then, $W_\alpha = \widehat{W}_\alpha = 0$ for all $\mathbf{x}_\alpha \in \partial \Omega$, 
and by the Mean Value Theorem, there holds 
{\small
\begin{align} \label{V_U2} 
	&- \delta_{x_1, 2h_1}^2 \widehat{W}_\alpha 
	= - \delta_{x_1, 2h_1}^2 W_\alpha - \rho \left( \hF[V_\alpha, \bx_\alpha] - \hF[U_\alpha, \bx_\alpha] \right) \\ 
	\nonumber & \, = -(1 - \rho k_{**})\delta_{x_1, 2h_1}^2 W_\alpha 
		- \rho \Bigl( \frac{\partial F}{\partial u_{x_1 x_1}} + k_{**} \Bigr) \delta_{x_1, 2h_1}^2 W_\alpha \\ 
		\nonumber & \quad 
		- \rho \sum_{i=2}^d \frac{\partial F}{\partial u_{x_i x_i}} \delta_{x_i, 2h_i}^2 W_\alpha  
		- \rho \frac{\partial F}{\partial u} W_\alpha 
		- \rho \gamma \sum_{i=1}^d \Bigl( \Bigl[ \widetilde{D}_{\bh}^2 W_\alpha \Bigr]_{ii} 
				- \Bigl[ \widehat{D}_{\bh}^2 W_\alpha \bigr]_{ii} \Bigr) \\ 
			\nonumber & \quad 
		- \rho \sum_{i=1}^d \sum_{j=1 \atop j \neq i}^d \Bigl( \frac12 \frac{\partial F}{\partial u_{x_i x_j}} + \gamma \Bigr) 
			\Bigl[ \widetilde{D}_{\bh}^2 W_\alpha \Bigr]_{ij} 
		- \rho \sum_{i=1}^d \sum_{j=1 \atop j \neq i}^d \Bigl( \gamma - \frac12 \frac{\partial F}{\partial u_{x_i x_j}} \Bigr) 
			\Bigl[ -\widehat{D}_{\bh}^2 W_\alpha \Bigr]_{ij} \\ 
	\nonumber &\, \equiv 
	-(1-\rho k_{**}) \delta_{x_1, 2h_1}^2 W_\alpha - \rho G'_\alpha W_\alpha 
\end{align} 
}
for a linear operator $G'_\alpha$ which depends on $U$, $V$. 

Let $\widehat{\mathbf{W}} , \mathbf{W} \in \mathbb{R}^{J_0}$ denote the vectorization of the grid functions 
$\widehat{W}$ and $W$ restricted to $\cT_{\bh} \cap \Omega$, respectively.  
Let $D_{2h}^2$ denote the matrix representation of the operator $-\delta_{x_1, 2h_1}^2 W_\alpha$
complemented by the boundary condition $\eqref{M21_rho_matrix_bc1}$. 
Then, using the notation in the proof of Lemma~\ref{lemma_contraction} to rewrite \eqref{V_U2}, we have  
\begin{align} \label{GD2h} 
	D^2_{2h} \widehat{\mathbf{W}} 	= G \bf{W},
\end{align}
where
\begin{align*}
	G & = \left(1-\rho k_{**} \right) D^2_{2h}  
		- \rho \left( F_1 - k_{**} I \right) D^2_{2h} \\ 
		\nonumber & \quad 
		- \rho \sum_{i=2}^d F_i A_i   
		- \rho F_0  
		- \rho \gamma \sum_{i=1}^d A_{ii}  
		- \rho \sum_{i=1}^d \sum_{j=1 \atop j \neq i}^d 
			\left( \widetilde{F}_{ij} \widetilde{A}_{ij} + \widehat{F}_{ij} \widehat{A}_{ij}  \right) \\ 
	& = \left( \frac12 - \rho k_{**} \right) D^2_{2h} 
		+ \left( \frac{1}{2N} I - \rho \left( F_1 - k_{**} I \right) \right) D^2_{2h} \\ 
		\nonumber & \quad 
		+ \sum_{i=2}^d \left( \frac{1}{2N} D^2_{2h} - \rho F_i A_i \right) 
		+ \left( \frac{1}{2N} D^2_{2h} - \rho F_0 \right) 
		+ \sum_{i=1}^d \left( \frac{1}{2N} D^2_{2h} - \rho \gamma A_{ii} \right) \\ 
		\nonumber & \quad  
	+ \sum_{i=1}^d \sum_{j=1 \atop j \neq i}^d \left( \frac{1}{2N} D^2_{2h} - \rho \widetilde{F}_{ij} \widetilde{A}_{ij} \right) 
	+ \sum_{i=1}^d \sum_{j=1 \atop j \neq i}^d \left( \frac{1}{2N} D^2_{2h} - \rho \widehat{F}_{ij} \widehat{A}_{ij} \right) 
\end{align*}
for $N = 1 + 2d + 2d(d-1)$. 
Thus, by Corollary~\ref{corollary4.2} with $B = D^2_{2h}$, there holds 
\begin{align*}
	\| G \mathbf{W} \|_2 
	& \leq \left( \frac12 - \rho k_{**} \right) \| D^2_{2h} \mathbf{W} \|_2 
		+ N \frac{1}{2N} \| D^2_{2h} \mathbf{W} \|_2 
	=  \left( 1 - \rho k_{**} \right) \| D^2_{2h} \mathbf{W} \|_2 
\end{align*}
for $\rho > 0$ sufficiently small. 
Taking the $\ell^2$ norm of both sides of \eqref{GD2h}, it follows that  
\[
	\| D^2_{2h} \widehat{\mathbf{W}} \|_2 \leq \left( 1 - \rho k_{**} \right) \| D^2_{2h} \mathbf{W} \|_2 . 
\]
The proof is complete.  
\hfill 
\end{proof}

\begin{lemma} \label{lemma_stable2}
Suppose the operator $F$ in \eqref{F_D2} 
is proper elliptic and Lipschitz continuous with respect to its first two arguments 
and $k_{**} > 0$,  
where $k_{**}$ denotes the lower ellipticity constant with respect to the second order variable.  
Let $U$ be the unique solution to the finite difference scheme defined by 
\eqref{FD_method}--\eqref{FD_auxiliaryBC} and \eqref{Abeta} for problem \eqref{F_D2} 
with $\gamma > K^{**}/2$ and $\beta = 0$. 
Then, 
\[
	\Bigl( \prod_{i=1,2,\ldots,d} h_i^{\frac12} \Bigr) \left\| \delta_{x_1, 2h_1}^2 U \right\|_{\ell^2(\cT_{\mathbf{h}} \cap \Omega)} 
	\leq C , 
\]
where $C$ is a positive constant independent of $\mathbf{h}$ 
that depends on $\Omega$, the lower ellipticity constant $k_{**}$, and 
$\| F(0,0,\cdot) \|_{C^0(\overline{\Omega})}$. 
\end{lemma}

\begin{proof}
Let $\cM^2_{1,\rho}$ be the mapping defined by \eqref{M21_rho} 
and $\rho > 0$ sufficiently small to ensure $\cM^2_{1,\rho}$ is a contraction.    
Let  $\widehat{0}$ denote the image of the input $0$ under the mapping, that is, 
$\widehat{0}$ satisfies 
\begin{alignat}{2}\label{zero}
	- \delta_{x_1, 2h_1}^2 \widehat{0}_\alpha
	& =  - \rho F (0,0,\mathbf{x}_\alpha)  \qquad  
		&& \text{for all } \mathbf{x}_\alpha \in \cT_{\bh} \cap \Omega  . 
\end{alignat} 
Then, by Lemma~\ref{lemma_contraction2}, 
for $U$ the solution to the proposed FD scheme, there holds 
\begin{align*} 
	\left\| \delta_{x_1, 2h_1}^2 U \right\|_{\ell^2(\cT_{\mathbf{h}} \cap \Omega)}
	 &\leq \left\| \delta_{x_1, 2h_1}^2 (U-  \widehat{0}) \right\|_{\ell^2(\cT_{\mathbf{h}} \cap \Omega)} + \| \delta_{x_1, 2h_1}^2 \widehat{0} \|_{\ell^2(\cT_{\mathbf{h}} \cap \Omega)} \\
	 &\leq (1-\rho k_{**})  \left\| \delta_{x_1, 2h_1}^2 (U-0)\right\|_{\ell^2(\cT_{\mathbf{h}} \cap \Omega)} 
	 + \| \delta_{x_1, 2h_1}^2 \widehat{0} \|_{\ell^2(\cT_{\mathbf{h}} \cap \Omega)} .
\end{align*} 
Hence,
\begin{equation} \label{ubound0-2}
	\left\| \delta_{x_1, 2h_1}^2 U \right\|_{\ell^2(\cT_{\mathbf{h}} \cap \Omega)} \leq 
	\frac{1}{\rho k_{**}} \| \delta_{x_1, 2h_1}^2 \widehat{0} \|_{\ell^2(\cT_{\mathbf{h}} \cap \Omega)} . 
\end{equation} 

It follows from \eqref{zero} that 
\[
	\| \delta_{x_1, 2h_1}^2 \widehat{0} \|_{\ell^2(\cT_{\mathbf{h}} \cap \Omega)} 
	\leq 
	\rho \| F(0,0,\cdot) \|_{\ell^2(\cT_{\mathbf{h}} \cap \Omega)} . 
\] 
Plugging this into \eqref{ubound0-2} yields 
\begin{align*}
	\Bigl( \prod_{i=1,2,\ldots,d} h_i^{\frac12} \Bigr) \left\| \delta_{x_1, 2h_1}^2 U \right\|_{\ell^2(\cT_{\mathbf{h}} \cap \Omega)} 
	& \leq \frac{ | \Omega |^d}{k_{**}} \| F(0,0,\cdot) \|_{C^0(\overline{\Omega})} .
\end{align*}
The proof is complete. 
\hfill 
\end{proof}

Clearly the analysis in this section can be extended to uniformly bound $\delta_{x_i, 2h_i}^2 U$ 
for all $i=1,2,\ldots,d$ as reflected in the following theorem. 

\begin{theorem} \label{thm_stable2}
Under the assumptions of Lemma \ref{lemma_stable2}, there holds
\[
	\Bigl( \prod_{i=1,2,\ldots,d} h_i^{\frac12} \Bigr) \left\| \delta_{x_\ell, 2h_\ell}^2 U \right\|_{\ell^2(\cT_{\mathbf{h}} \cap \Omega)} 
	\leq C 
\]
for all $\ell=1,2,\ldots,d$, 
where $C$ is a positive constant independent of $\mathbf{h}$ 
that depends on $\Omega$, the lower ellipticity constant $k_{**}$, and 
$\| F(0,0,\cdot) \|_{C^0(\overline{\Omega})}$. 
\end{theorem}

\subsection{$\ell^\infty$-norm stability} \label{c0-stability_sec}
 
We now derive an $\ell^\infty$-norm stability estimate for  the numerical 
solution $U$ using the $\ell^\infty$-norm estimate for $U$ restricted to the boundary in Theorem~\ref{prop_stability} 
and the (high-order) $\ell^2$-norm stability estimate obtained in Theorem~\ref{thm_stable2}. 
Such an $\ell^\infty$-norm stability estimate is vital for our convergence analysis to be given in Section~\ref{conv_proof_sec}. 
This result also has an independent interest because it can be regarded as a discrete Sobolev embedding result which 
holds for all grid functions that satisfy the two stability estimates of Theorems \ref{prop_stability} and \ref{thm_stable2} 
as well as the Dirichlet boundary condition.

We first define a piecewise constant extension $u_{\mathbf{h}}$ for a given grid function 
$U \in S(\cT_{\mathbf{h}}')$ 
to be used in the proof of the stability theorem.  
Let $\alpha \in \mathbb{N}_J'$, and define $B_\alpha$ by 
\[
	B_\alpha \equiv \prod_{i=1,2,\ldots,d} 
		\big( \mathbf{x}_\alpha-\frac{h_i}2  { \mathbf{e}_i} , \mathbf{x}_\alpha+ \frac{h_i}2 { \mathbf{e}_i}  \big] , 
\] 
where $\{\mathbf{e}_i\}_{i=1}^d$ denotes the 
canonical basis for $\mathbb{R}^d$.  
We define the piecewise constant extension $u_{\mathbf{h}}$ of a grid function $U$ by 
\begin{equation} \label{FD_extension_nd}
	u_{\mathbf{h}} (\mathbf{x}) \equiv 
	U_\alpha , \qquad \mathbf{x} \in B_\alpha   
\end{equation} 
for all $\mathbf{x} \in \Omega' \equiv \cup_{\alpha \in \mathbb{N}_J'} B_\alpha \supset \overline{\Omega}$.  

To motivate the proof, assume $u_{\bh_k} \to v$ pointwise.  By the uniform $L^2$ stability of $u_{\bh_k}$ 
guaranteed by Theorem~\ref{prop_stability}, we have $v \in L^2(\Omega)$.  
Suppose $v \in H^2(\Omega)$.  Then $v \in C^0(\overline{\Omega})$ for $d=1,2,3$ by the Sobolev embedding theorem.  
In order to prove the result using only the discrete high-order stability of $u_{\bh_k}$ guaranteed by 
Theorem~\ref{thm_stable2}, we assume the pointwise limit function $v$ is discontinuous and arrive at a contradiction.  
Such a contradiction arises because if $v$ is discontinuous at $\bx_0$, 
we can choose a sequence $\bx_k \to \bx_0$ such that $u_{\bh_k} (\bx_k)$ has a similar ``jump" with 
\[
	C_2 \geq \| \delta_{x_\ell, 2 h_\ell}^2 u_{\bh_k} \|_{L^2(\Omega)} 
	\geq h_\ell^{d/2} \left| \delta_{x_\ell, 2 h_\ell}^2 u_{\bh_k} (\bx_k) \right| = \mathcal{O}(h_\ell^{d/2-2}) \to \infty 
\]
for $d = 1,2,3$ whenever $h_\ell = \min h_j$ for all $k$.  
Thus, $v \in C^0(\overline{\Omega})$, and it follows that $v$ is uniformly bounded.  
Consequently, $u_{\bh_k}$ cannot be unbounded as $k \to \infty$.  
The key to the proof is constructing the correct pointwise limit function $v$.  

\begin{theorem} \label{thm_stability} 
Under the assumptions of Lemma \ref{lemma_stable2}, the numerical 
solution $U$ is stable in the $\ell^\infty$-norm for $d\leq 3$; that is, $U$ satisfies
\[
	\left\| U \right\|_{\ell^\infty(\cT_{\mathbf{h}})} 
	\leq C 
\]
for $d\leq 3$, where $C$ is a positive constant independent of $\mathbf{h}$.  
\end{theorem}

\begin{proof}  
To show that $U$ is uniformly bounded independent of $\bh$, we will use a proof by contradiction.  
Suppose there exists a sequence $\bh_{k}$ such that $h_{k} \to 0$ and 
$\left\| U^{(k)} \right\|_{\ell^\infty(\cT_{\mathbf{h}_{k}})} \to \infty$.  
Let $u_{\bh_k}$ denote the piecewise constant 
extension of the grid function $U^{(k)}$ defined by \eqref{FD_extension_nd}.  
Then, there exists a point $\bx_0 \in \overline{\Omega}$ and a sequence $\bx_k \in \cT_{\bh_k}$ 
such that $\bx_k \to \bx_0$ and $\left| u_{\bh_k}(\bx_k) \right| \to \infty$.  
Furthermore, since $g \in C^0(\partial \Omega)$ and $u_{\bh_k}(\bx) = g(\bx)$ for all $\bx \in \cT_{\bh_k} \setminus \Omega$, 
we choose a subsequence such that $\bx_k$ does not correspond to a boundary node for all $k$.  
In the following assume $u_{\bh_k}(\bx_k) \to + \infty$.  
The proof can easily be modified for $u_{\bh_k}(\bx_k) \to - \infty$.  

Define the sequence $i_k \subset \{ 1,2,\ldots,d \}$ such that 
$h_{i_k}^{(k)} = \min_j h_j^{(k)}$ for all $k$.  
Choose subsequences such that $i_k = i^*$ for some $i^* \in \{1,2,\ldots,d\}$ for all $k$.  
Let $\widetilde{g} \in C(\mathbb{R}^d)$ such that $\widetilde{g} \big|_{\partial \Omega} = g$, 
$\widetilde{U}^{(k)}$ denote the extension of $U^{(k)}$ to a grid over $\mathbb{R}^d$ that contains $\cT_{\bh_k}$ 
with $\widetilde{U}^{(k)}_\alpha = \widetilde{g}(\bx_\alpha)$ for all $\bx_\alpha \notin \cT_{\bh_k} \cap \Omega$, 
and $\widetilde{u}_{\bh_k}$ denote its corresponding piecewise constant extension.  
Define the function $v : \overline{\Omega} \to \overline{\mathbb{R}}$ by 
\begin{equation} \label{v_USC}
	v(\bx) = \limsup_{k \to \infty \atop \sigma \to 0} \widetilde{u}_{\bh_k} \left( \bx + (\bx_k - \bx_0) + \sigma \mathbf{e}_{i^*}  \right) , 
\end{equation}
where we have restricted the paths to only vary along the $x_{i^*}$ direction.  
Then, $v$ is upper semi-continuous along the $x_{i^*}$ direction.  
Since $\bx_k \in \Omega$ for all $k$, there holds 
$v(\bx_0) = \infty$ since $\widetilde{u}_{\bh_k}(\bx_k) = u_{\bh_k}(\bx_k) \to \infty$.  

Suppose there exists $\bx^* \in \overline{\Omega}$ such that $v$ is discontinuous at $\bx^*$ along the $x_{i^*}$ direction, 
i.e., 
\[
	v(\bx^*) \neq \lim_{\sigma \to 0^+} v(\bx^* + \sigma \mathbf{e}_{i^*}) 
	\qquad \text{or} \qquad 
	v(\bx^*) \neq \lim_{\sigma \to 0^-} v(\bx^* + \sigma \mathbf{e}_{i^*}) . 
\]  
We consider two cases based on whether $\bx^* \in \Omega$ or $\bx^* \in \partial \Omega$.  

\underline{Case 1}:  $\bx^* \in \Omega$.  
Then, by the definition of $v$ in \eqref{v_USC}, 
there exists a sequence $\sigma_k \to 0$ and a constant $c > 0$ such that 
\begin{subequations}
\begin{align}
	& u_{\bh_k} \left( \bx + (\bx_k - \bx_0) + \sigma_k \mathbf{e}_{i^*}  \right) \to v (\bx^*) , \\ 
	& \left( 2 h_{i^*}^{(k)} \right)^2 \delta_{x_{i^*}, 2 h_{i^*}^{(k)}}^2 u_{\bh_k} \left( \bx + (\bx_k - \bx_0) + \sigma_k \mathbf{e}_{i^*}  \right) \to -c 
\end{align}
\end{subequations}
since 
$$\limsup_{k \to \infty} \Big[ 
u_{\bh_k} \bigl( \bx + (\bx_k - \bx_0) + \sigma_k \mathbf{e}_{i^*}  \bigr) 
- u_{\bh_k} \bigl( \bx + (\bx_k - \bx_0) + \sigma_k \mathbf{e}_{i^*} \pm 2 h_{i^*}^{(k)} \mathbf{e}_{i^*} \bigr) \Big] \geq 0$$ 
and at least one choice must be positive if $v$ is discontinuous at $\bx^*$ along the $x_{i^*}$ direction.  
Define $\bz_k = \bx + (\bx_k - \bx_0) + \sigma_k \mathbf{e}_{i^*} \in \Omega$ for $k$ sufficiently large.  
Applying Theorem~\ref{thm_stable2}, there exists a constant $C_2$ independent of $\bh_k$ such that 
\begin{align} \label{H2_contradiction}
	C_2 
	& \geq \Bigl( \prod_{j=1,2,\ldots,d} \sqrt{h_j^{(k)}} \Bigr) 
		\left\| \delta_{x_{i^*}, 2 h_{i^*}^{(k)}}^2 U \right\|_{\ell^2(\cT_{\mathbf{h}_k} \cap \Omega)} \\ 
	\nonumber & \geq \Bigl( h_{i^*}^{(k)} \Bigr)^{d/2}  
		\left\| \delta_{x_{i^*}, 2 h_{i^*}^{(k)}}^2 U \right\|_{\ell^2(\cT_{\mathbf{h}_k} \cap \Omega)} \\ 
	\nonumber & \geq \Bigl( h_{i^*}^{(k)} \Bigr)^{d/2}  
		\left| \delta_{x_{i^*}, 2 h_{i^*}^{(k)}}^2 u_{\bh_k}(\bz_k) \right| \\ 
	\nonumber & \geq  \frac{c}8 \Bigl( h_{i^*}^{(k)} \Bigr)^{d/2-2} 
\end{align}
for all $k$ sufficiently large, 
a contradiction when $d \leq 3$ since $\Bigl( h_{i^*}^{(k)} \Bigr)^{d/2-2} \to + \infty$.  

\underline{Case 2}:  $\bx^* \in \partial \Omega$.  
Suppose there exists a sequence $\bz_k \to \bx^*$ such that $\widetilde{u}_{\bh_k}(\bz_k) \to v(\bx^*)$ 
with $\bz_k \in \Omega$ for all $k$.  
Then, the same argument as Case 1 applies leading to a contradiction.  
Thus, we have the only sequences $\bz_k \to \bx^*$ with $\widetilde{u}_{\bh_k}(\bz_k) \to v(\bx^*)$ 
must satisfy $\bz_k \notin \Omega$ for all $k$, 
from which we have $v(\bx^*) = g(\bx^*)$.  

Assume $\bz_k + \sigma \mathbf{e}_{i^*} \notin \Omega$ for all $\sigma \in \mathbb{R}$ and $k > 0$. 
Then, $\widetilde{u}_{\bh_k}(\bz_k + \sigma \mathbf{e}_{i^*}) = \widetilde{g}(\bz_k + \sigma \mathbf{e}_{i^*}) 
\to \widetilde{g}(\bx^* + \sigma \mathbf{e}_{i^*})$ as $k \to \infty$ 
and $ \widetilde{g}(\bx^* + \sigma \mathbf{e}_{i^*}) \to g(\bx^*) = v(\bx^*)$ as $\sigma \to 0$, 
contradicting the assumption that $v$ is discontinuous at $\bx^*$ along the $x_{i^*}$ direction.  
Thus, by the continuity of $\widetilde{g}$, we can assume $\bz_k \in \partial \Omega$ and either 
$\bz_k + 2 h_{i^*}^{(k)} \mathbf{e}_{i^*} \in \Omega$ or $\bz_k - 2 h_{i^*}^{(k)} \mathbf{e}_{i^*} \in \Omega$ 
for $k$ sufficiently large. 

Without a loss of generality, assume $\bz_k + 2 h_{i^*}^{(k)} \mathbf{e}_{i^*} \in \Omega$.  
Then, by the definition of $v$ in \eqref{v_USC} and the fact that no sequence can be used 
that approaches $\bx^*$ from the interior of $\Omega$, 
there holds 
\begin{equation}
	\limsup_{k \to \infty}  u_{\bh_k}(\bz_k + 2 h_{i^*}^{(k)} \mathbf{e}_{i^*} ) < v(\bx^*) . 
\end{equation}
Define $c > 0$ by 
\[
	c = \limsup_{k \to \infty} 
	 \left[ u_{\bh_k}(\bz_k) - u_{\bh_k}(\bz_k + 2 h_{i^*}^{(k)} \mathbf{e}_{i^*} ) \right] .  
\]
Suppose 
\[
	c \neq \limsup_{k \to \infty} 
	 \left[ u_{\bh_k}(\bz_k + 4 h_{i^*}^{(k)} \mathbf{e}_{i^*} ) ) - u_{\bh_k}(\bz_k + 2 h_{i^*}^{(k)} \mathbf{e}_{i^*} ) \right] .  
\]
Then, we have 
\[
	 \left( 2 h_{i^*}^{(k)} \right)^2 \left| \delta_{x_{i^*}, 2 h_{i^*}^{(k)}}^2 u_{\bh_k} \left( \bz_k + 2 h_{i^*}^{(k)} \mathbf{e}_{i^*}   \right) \right| \neq 0 , 
\]
and we can again form a contradiction to the $\ell^2$ stability of $\delta_{x_{i^*}, 2 h_{i^*}^{(k)}}^2 u_{\bh_k}$ for $d \leq 3$ 
using the estimate \eqref{H2_contradiction} with $\bz_k$ replaced by $\bz_k + 2 h_{i^*}^{(k)} \mathbf{e}_{i^*}$.    
Thus, we must have 
\[
	c = \limsup_{k \to \infty} 
	 \left[ u_{\bh_k}(\bz_k + 4 h_{i^*}^{(k)} \mathbf{e}_{i^*} ) ) - u_{\bh_k}(\bz_k + 2 h_{i^*}^{(k)} \mathbf{e}_{i^*} ) \right] .  
\]
Continuing in this fashion, we can show 
\begin{equation}
	c = \limsup_{k \to \infty} 
	 \left[ u_{\bh_k}(\bz_k + (2N_k+2) h_{i^*}^{(k)} \mathbf{e}_{i^*} ) ) - u_{\bh_k}(\bz_k + 2N_k h_{i^*}^{(k)} \mathbf{e}_{i^*} ) \right] 
\end{equation}
for all $N_k$ such that $\bz_k + (2N_k+1) h_{i^*}^{(k)} \mathbf{e}_{i^*} \in \overline{\Omega}$ 
or $\bz_k + (2N_k+2) h_{i^*}^{(k)} \mathbf{e}_{i^*} \in \overline{\Omega}$. 
However, as $k \to \infty$, this would imply $\widetilde{u}_{\bh_k}(\bs_k) \to - \infty$ for some sequence $\bs_k \notin \Omega$ 
with $\bs_k \to \bs \in \partial \Omega$
on the opposite side of the domain from $\bx^*$.  
Note that $\bs_k \in \partial \Omega$ or $\bs_k$ is a ghost point adjacent to $\partial \Omega$.  
Thus, $\widetilde{u}_{\bh_k}(\bs_k) = \widetilde{g}(s_\alpha) \to - \infty$ is a contradiction,   
and it follows that $\bx^* \not\in \partial \Omega$.  

Combining both cases, we have $v$ defined by \eqref{v_USC} must be continuous over $\overline{\Omega}$ 
along the $x_{i^*}$ direction.  
Let the line segment $S \subset \overline{\Omega}$ be defined by 
\[
	S = \left\{ \bx \in \overline{\Omega} \mid \bx = \bx_0 + \sigma \mathbf{e}_{i^*} \text{ for some } \sigma \in \mathbb{R} \right\} . 
\]   
Then $v$ is continuous over $S$, and we have $v$ is uniformly bounded over $S$.  
This is a contradiction to the fact that $v(\bx_0) = \infty$ and $\bx_0 \in S$.   

If $u_{\bh_k}(\bx_k) \to - \infty$, then we can construct $v$ to be lower semi-continuous 
along the $x_{i^*}$ direction by using the $\liminf$ 
in \eqref{v_USC} and arrive at an analogous contradiction in that $v(\bx_0) \neq - \infty$.  
Therefore, we must have $\left\| U \right\|_{\ell^\infty(\cT_{\bh})}$ is uniformly bounded independent of $h$ 
when $d \leq 3$.  
The proof is complete. 
\hfill 
\end{proof}

\section{Convergence of the narrow-stencil finite difference scheme} \label{conv_proof_sec}
The goal of this section is to establish the convergence of the solution to the proposed  
scheme \eqref{FD_method}--\eqref{FD_auxiliaryBC} and \eqref{Abeta} 
to the viscosity solution of \eqref{FD_problem}. Since the scheme is not 
monotone in the sense of \cite{Barles_Souganidis91}, the convergence 
framework therein is not applicable to our scheme. Instead, we
shall give a direct convergence proof which can be regarded as the high
dimensional extension of the 1-D proof given in our early work \cite{Feng_Kao_Lewis13}.  
The new proof is more involved due to the additional difficulty caused by the mixed second 
order derivatives in the Hessian $D^2 u$.

Before stating our convergence theorem and presenting its proof, we first give a wholistic 
account of the proof. 
Let $\overline{u}$ (resp. $\underline{u}$) denote the upper (resp. lower) limit of the sequence $\{ u_{\bh_{k}} \}$ 
(see \eqref{upper_lower_limits} for the precise definition) 
which exists because $u_{\bh_{k}}$ is uniformly bounded in the $L^\infty$-norm.  
Our task is to show that $\overline{u}$ (resp. $\underline{u}$) is a viscosity subsolution (resp. 
supersolution) of problem \eqref{FD_problem}.  
To this end, let $\varphi\in C^2$ be a test function and  $\overline{u}-\varphi$ 
take a local maximum at $\bx_0\in \Ome$ 
(the case $\bx_0\in \p\Ome$ needs to be considered separately to check that the boundary condition is satisfied in the 
viscosity sense).  
We must show that $F(D^2\varphi(\bx_0), \nabla\varphi(\bx_0), \overline{u}(\bx_0), \bx_0)\leq 0$  
using the fact that $\widehat{F}[u_{\bh_{k}},\bx_0]=0$.  
As expected, the difficulty of the proof is caused 
by the loss of the monotonicity (in the Barles-Souganidis sense \cite{Barles_Souganidis91}) of our numerical 
scheme. 
Moreover, the complexity of the proof for the above desired inequality 
depends on the regularity of $\overline{u}$.  
We will consider three cases: (i) $\overline{u}\in C^2$; (ii) $\overline{u}_{x_\ell} \in \mbox{Lip} \setminus C^1$; 
(iii) $\overline{u}_{x_\ell} \not\in \mbox{Lip}$ for some $\ell=1,2,\cdots,d.$.  

Case (i) is easy due to the consistency of the scheme and the ellipticity of $F$.  
The other two cases are more involved.  
The second case (ii) is subtle in that $\nabla \overline{u}$ is Lipschitz but 
$\overline{u} \notin C^2$ 
implying $D^2 \overline{u}$ exists almost everywhere and is bounded 
but there is no guarantee that $D^2 \ou(\bx_0) \leq D^2 \varphi (\bx_0)$.  
The key is to choose a maximizing sequence similar to the Barles-Souganidis proof 
but show that the narrow-stencil scheme still provides sufficient directional resolution without needing a wide-stencil.  
This requires using a strategic interpolation function and 
choosing the the correct path when sending $\bh_{k} \to \mathbf{0}^+$ and $\bx_k \to \bx_0$ 
to ensure that the local grid approximately aligns with the eigenvectors of $H - D^2 \varphi (\bx_0)$ for an appropriate matrix $H$ 
such that $D^2 \ou(\bx_k) \to H$.  

The rest of the proof focuses on case (iii) where 
$\bigl[ \widetilde{D}_{\mathbf{h}_{k}}^2 \overline{u}(\bx_0) ]_{\ell\ell}$  is unbounded for
some $1\leq \ell\leq d$. 
Let $\{u_{\bh_k} \}$ be a corresponding subsequence such that $u_{\bh_{k}} \to \overline{u}$ at $\bx_0$ 
and 
$\bigl[ \widetilde{D}_{\mathbf{h}_k}^2 u_{\bh_k}(\bx_0) ]_{\ell \ell}$ becomes unbounded as $k \to \infty$. 
We then choose subsequences such that $h_\ell^{(k)} << h_j^{(k)}$ for all $j \neq \ell$ ensuring 
$\bigl[ \widetilde{D}_{\mathbf{h}_k}^2 u_{\bh_k}(\bx_0) ]_{\ell \ell} << \bigl[ \widetilde{D}_{\mathbf{h}_k}^2 u_{\bh_k}(\bx_0) ]_{j j}$.  
Using the ellipticity of $F$, we can extract the positive term 
$- k_{**} \bigl[ \widetilde{D}_{\mathbf{h}_k}^2 u_{\bh_k}(\bx_0) ]_{\ell \ell} \to + \infty$.  
Since the term scales as $\left(h_\ell^{(k)} \right)^{-2}$, we expect this term to dominate when $h_\ell^{(k)} << h_j^{(k)}$.  
As such, we use this term to absorb all of the remaining unsigned terms that result when passing from 
$\widehat{F}[u_{\bh_{k}},\bx_0]$ to $F[\varphi](\bx_0)$.  
The numerical moment ensures the existence of an index $\ell$ such that 
$\bigl[ \widetilde{D}_{\mathbf{h}_k}^2 u_{\bh_k}(\bx_0) ]_{\ell \ell}$ 
becomes unbounded as $k \to \infty$ if any term in the discretization diverges.  

In order to exploit the blow-up inherent to case (iii), we consider three possibilities: (a) 
$\overline{u}$ is $C^1$; (b)
$\overline{u}$ is Lipschitz, i.e., $\nabla \overline{u}$ exists almost everywhere and is bounded; (c)
$\nabla \overline{u}$ does not exist due to unboundedness.     
As motivation, 
if $\overline{u} \in C^1$, we would have $\nabla \overline{u}(\bx_0) = \nabla \varphi (\bx_0)$ 
and $\nabla \ou$ is bounded on a neighborhood of $\bx_0$, 
a property that can be exploited to bound the mixed second order derivatives in 
the Hessian approximation $D_{\bh_k}^2 u_{\bh_k}$.   
When $\nabla \overline{u}$ is bounded, we can use a similar argument as $(a)$ 
and the continuity of $\ou$ to ensure the other terms in the discretization do not behave too badly.  
When $\nabla \overline{u}(\bx_0)$ is unbounded, 
our idea is to use the fact that 
$\bigl[ \widetilde{D}_{\mathbf{h}_{k}}^2 u_{\bh_k} (\bx_0) \bigr]_{\ell\ell}$ diverges at a high enough rate 
to directly absorb all of the other terms.  
Thus, the convergence proof is based on using the special structure of the scheme 
and choosing appropriate sequences that exploit the regularity of the underlying 
viscosity subsolution $\overline{u}$. 
In addition, the ellipticity of $F$ and the Lipschitz continuity of $F$ 
play an important role for us to move ``derivatives" onto $\varphi$.  

Figure~\ref{max_moment_fig} illustrates an aspect of the wholistic approach
assuming the function can be touched from above.  
For lower-regularity functions, we expect the diagonal components in the numerical moment 
to be positive giving additional freedom to absorb the contributions of the non-monotone components 
in the scheme. This idea is illustrated in Figure~\ref{max_moment_fig} where we see that for smooth functions 
we expect the numerical moment terms 
$\widetilde{\delta}_{x_i, h_i^{(k)}}^2 u_{\bh_k}(\mathbf{x}_0) - \widehat{\delta}_{x_i, h_i^{(k)}}^2 u_{\bh_k}(\mathbf{x}_0)$ 
to be negative (and going to zero by the consistency of the scheme) 
and for non-smooth functions we expect them to be positive and potentially diverging.

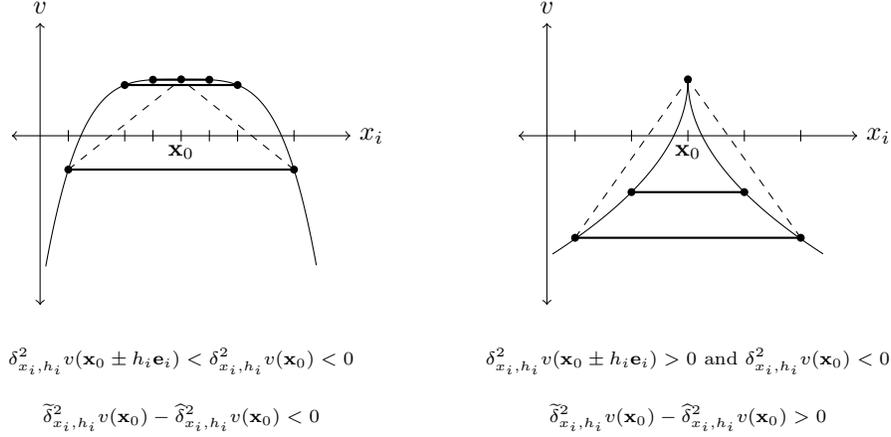
\begin{figure}[htb] \label{moment_sign_fig}
\begin{center}
\begin{tikzpicture}[scale=0.75]
\draw[<->] (-3,0) -- (3,0) node[right] {$x_i$};
\draw[<->] (-2.5,-3) -- (-2.5,2) node[above] {$v$};
\draw[samples=1000,domain=-2.4:2.4, smooth, variable=\x] plot ({\x}, {-0.1*\x*\x*\x*\x+1});
\fill (0,1) circle (2pt); 
\draw (0,-0.1) -- (0,0.1);
\node[below] at (0,0) {\small $\mathbf{x}_0$};
\fill (-0.5, 0.99375) circle (2pt); 
\draw (-0.5,-0.1) -- (-0.5,0.1);
\fill (-1, 0.9) circle (2pt);
\draw (-1,-0.1) -- (-1,0.1);
\fill (-2, -0.6) circle (2pt);
\draw (-2, -0.1) -- (-2, 0.1);
\fill (0.5, 0.99375) circle (2pt); 
\draw (0.5,-0.1) -- (0.5,0.1);
\fill (1, 0.9) circle (2pt);
\draw (1,-0.1) -- (1,0.1);
\fill (2, -0.6) circle (2pt);
\draw (2, -0.1) -- (2, 0.1);
\draw[dashed] (-2, -0.6) -- (0,1);
\draw[dashed] (2, -0.6) -- (0,1);
\draw[thick] (-0.5, 0.99375) -- (0.5, 0.99375);
\draw[thick] (-1,0.9) -- (1,0.9); 
\draw[thick] (-2, -0.6) -- (2,-0.6);
\node at (0,-4) {\scriptsize $\delta_{x_i, h_i}^2 v(\mathbf{x}_0 \pm h_i \mathbf{e}_i) 
< \delta_{x_i, h_i}^2 v(\mathbf{x}_0) < 0$};
\node at (0,-5) {\scriptsize 
$\widetilde{\delta}_{x_i, h_i}^2 v(\mathbf{x}_0) - \widehat{\delta}_{x_i, h_i}^2 v(\mathbf{x}_0) < 0$};
\end{tikzpicture}
\hspace{0.35in}
\begin{tikzpicture}[scale=0.75]
\draw[<->] (-3,0) -- (3,0) node[right] {$x_i$};
\draw[<->] (-2.5,-3) -- (-2.5,2) node[above] {$v$};
\draw[samples=1000,domain=-2.4:2.4, smooth, variable=\x] plot ({\x}, {-2*sqrt(sqrt(\x*\x))+1});
\fill (0,1) circle (2pt); 
\draw (0,-0.1) -- (0,0.1);
\node[below] at (0,0) {\small $\mathbf{x}_0$};
\fill (-1, -1) circle (2pt);
\draw (-1,-0.1) -- (-1,0.1);
\fill (-2, -1.81) circle (2pt);
\draw (-2, -0.1) -- (-2, 0.1);
\fill (1, -1) circle (2pt);
\draw (1,-0.1) -- (1,0.1);
\fill (2, -1.81) circle (2pt);
\draw (2, -0.1) -- (2, 0.1);
\draw[dashed] (-2, -1.81) -- (0,1);
\draw[dashed] (2, -1.81) -- (0,1);
\draw[thick] (-1,-1) -- (1,-1); 
\draw[thick] (-2, -1.81) -- (2,-1.81);
\node at (0,-4) {\scriptsize $\delta_{x_i, h_i}^2 v(\mathbf{x}_0 \pm h_i \mathbf{e}_i) > 0$ and 
$\delta_{x_i, h_i}^2 v(\mathbf{x}_0) < 0$};
\node at (0,-5) {\scriptsize 
$\widetilde{\delta}_{x_i, h_i}^2 v(\mathbf{x}_0) - \widehat{\delta}_{x_i, h_i}^2 v(\mathbf{x}_0) > 0$};
\end{tikzpicture}
\end{center}
\caption{
The regularity of $v$ determines the sign of the nonmixed components of the numerical moment.  
If $v$ has lower regularity and can be touched from above at $\bx_0$, then we expect the 
terms to be increasingly positive.  
}
\label{max_moment_fig}
\end{figure}

\begin{theorem} \label{FD_convergence_nd}
Suppose that $g$ is continuous on $\partial \Omega$. Assume  problem \eqref{FD_problem} 
satisfies the comparison principle of Definition~\ref{comparison},  has a unique continuous 
viscosity solution $u$,  
the operator $F$ is proper elliptic with a lower ellipticity constant $k_{**} > 0$, 
and $F$ is Lipschitz continuous with respect to its first three arguments.  
Let $U\in S(\cT_{\mathbf{h}}')$ be the solution 
to the finite difference scheme defined by  \eqref{FD_method}--\eqref{FD_auxiliaryBC} and \eqref{Abeta} 
with  $\gamma \geq \frac{K^{**}}{2}$ and $\beta \geq \frac{K^*}{2}$ for 
$K^{**} = \max \left\{ K_{ij} \right\}$ and $K^* = \max \left\{ K_i \right\}$, where $K_{ij}$ and $K_{i}$ 
denote the Lipschitz constants of $F$ with respect to the components of the 
$D^2 u$ and $\nabla u$ arguments, respectively.  
Let $u_{\mathbf{h}}$ be the piecewise constant extension of $U$ defined  by \eqref{FD_extension_nd}.  
Suppose the scheme is admissible and $\ell^\infty$-norm stable.  
Then $u_{\mathbf{h}}$ converges to $u$ locally uniformly
as $\mathbf{h} \to  \mathbf{0}^+$.  
\end{theorem}

\begin{proof}
Since the proof is long and technical, we divide it into six steps.   

\medskip 
{\em Step 1}: 
Because the finite difference scheme is $\ell^\infty$-norm stable,
there exists a constant $C > 0$ such that 
\begin{equation}\label{e3.7}
\left\| u_{\mathbf{h}} \right\|_{L^\infty(\Omega)} \leq C
\end{equation}
independent of $\bh$.   
Define the upper and lower semicontinuous functions $\ou$ and $\uu$ by 
\begin{equation}\label{upper_lower_limits}
	\ou(\mathbf{x})\equiv\limsup_{ \bh \to \mathbf{0}^+ \atop \xi \to \bx} u_{\bh}(\xi) , 
	\qquad 
	 \uu(\mathbf{x})\equiv\liminf_{ \bh \to \mathbf{0}^+ \atop \xi \to \bx}  u_{\bh}(\xi) , 
\end{equation} 
where the limits are understood as multi-limits 
(we refer the reader to \cite{Habil05} for an introduction to multi-limits and multi-index notation).
The remainder of the proof is to show that $\ou$ and $\uu$ are, 
respectively, a viscosity subsolution and a viscosity supersolution of \eqref{FD_problem}.  
Hence, they must be the same and coincide with the viscosity solution of \eqref{FD_problem} by the comparison principle. 
We will only show that $\ou$ is a viscosity subsolution since the
proof that $\uu$ is a viscosity supersolution is analogous.

\medskip
{\em Step 2}: 
To show $\ou$ is a viscosity subsolution of \eqref{FD_problem},  let $\varphi\in C^2(\oOme)$ 
such that $\ou-\varphi$ takes a strict local maximum at $\mathbf{x}_0\in \overline{\Omega}$.
We first assume that $\varphi\in \mathcal{P}_2$, the set of all quadratic polynomials.
In {\em Step 4} we shall consider the general test function  $\varphi\in C^2(\oOme)$. 
Without a loss of generality, we assume $\ou(\mathbf{x}_0)=\varphi(\mathbf{x}_0)$ (after a translation 
in the dependent variable). Then there exists a ball, 
$B_{r_0}(\mathbf{x}_0) \subset \mathbb{R}^d$, 
centered at $\mathbf{x}_0$ with radius $r_0>0$ (in the $\ell^\infty$ metric) 
such that
\begin{equation}\label{e3.9}
\ou(\mathbf{x})-\varphi(\mathbf{x}) < \ou(\mathbf{x}_0)-\varphi(\mathbf{x}_0)= 0  
\qquad\forall \mathbf{x}\in \left( B_{r_0}(\mathbf{x}_0) \cap \overline{\Omega} \right) \setminus \{ \bx_0 \}.
\end{equation}

\medskip
{\em Step 3}:  We now show that 
if $\mathbf{x}_0 \in \Omega$, then 
\begin{equation}  \label{step3a} 
F \bigl( D^2 \varphi(\mathbf{x}_0) , \nabla \varphi(\mathbf{x}_0) , \varphi(\mathbf{x}_0) , 
		\mathbf{x}_0 \bigr) \leq 0 , 
\end{equation} 
and, if $\mathbf{x}_0 \in \partial \Omega$, then either 
\begin{subequations}\label{step3ab}
\begin{align}
	F \bigl( D^2 \varphi(\mathbf{x}_0) , \nabla \varphi(\mathbf{x}_0) , \varphi(\mathbf{x}_0) , 
		\mathbf{x}_0 \bigr) &\leq 0 \label{step3b1} \\
\noalign{or}
	B \left[ \varphi \right](\mathbf{x}_0) &\leq 0 . \label{step3b2}
\end{align}
\end{subequations}
Such a conclusion would validate that $\ou$ is a viscosity subsolution of \eqref{FD_problem} 
and satisfies the boundary condition in the viscosity sense with respect to the quadratic 
test function $\varphi$.  

\medskip 
{\em Step 3a}:  We first consider the scenario when $\mathbf{x}_0 \in \Omega$ and 
prove that \eqref{step3a} holds. 
To this end, we will consider cases based on the regularity of $\ou$ at $\bx_0$.  

By the definition of $\ou$ and \eqref{e3.9}, there exists (maximizing) 
sequences $\{\bh_{\bk}\}$, $\{\bx_{\bk}\}$, and $\{\bz_{\bk}\}$ 
such that 
\begin{subequations}\label{max_seq}
\begin{align}
	& \bh_{\bk} \to \mathbf{0}^+ , \\ 
	& \bx_{\bk} \to \bx_0 \text{ with } \bx_{\bk} \in \cT_{\bh_{\bk}} , \\ 
	& u_{\bh_{\bk}}(\bx_{\bk}) \to \ou(\bx_0) , \\ 
	& \bz_{\bk} \to \bx_0  \text{ with } \left| x_i^{(k_i)} - z_i^{(k_i)} \right| \leq \frac12 h_i^{(k_i)} 
		\text{ and } u_{\bh_{\bk}}(\bz_{\bk}) = u_{\bh_{\bk}}(\bx_{\bk}), \\
	& u_{\bh_{\bk}}(\bz) - \varphi(\bz) \text{ is locally maximized at } 
		\bz = \bz_{\bk} \text{ for all } \min \bk \text{ sufficiently large} . 
\end{align}
\end{subequations}
Note that the need for the two possibly different points $\bx_{\bk}$ and $\bz_{\bk}$ is due to the fact that 
$u_{\bh_{\bk}}$ is piecewise constant and $\varphi$ is not.  
Let $\cN_{\bk} \subset \cT_{\bh_{\bk}}'$ denote the local neighborhood of $\bx_{\bk}$ defined by 
\begin{align} \label{Nkdef}
	\cN_{\bk} \equiv \Big\{ \bx \in \cT_{\bh_{\bk}} \mid 
	& \bx = \bx_{\bk} + \sum_{i=1}^d \alpha_i h_i^{(k_i)} \mathbf{e}_i 
	\text{ for some } \alpha_i \in \{1, 0, -1 \} \Big\} , 
\end{align}
and let $\overline{\cN}_{\bk}$ denote the local stencil of the proposed finite difference scheme centered at $\bx_{\bk}$.  
Thus, $\cN_{\bk}$ denotes all nearest neighbors while $\overline{\cN}_{\bk}$ does not contain all 
nearest neighbors when $d \geq 3$ and contains additional nodes two steps away in the Cartesian 
directions.  
We denote the corresponding neighborhoods centered at $\bz_{\bk}$ by 
$\cN_{\bk}'$ and $\overline{\cN}_{\bk}'$, where 
$\cN_{\bk}'$ is defined by 
\[
	\cN_{\bk}' \equiv \left\{ \bz \mid \bz = \bx + (\bz_{\bk} - \bx_{\bk}) \text{ for some } \bx \in \cN_{\bk} \right\} 
\]
and $\overline{\cN}_{\bk}'$ consists of the local stencil of the proposed finite difference scheme centered at $\bz_{\bk}$.
Observe that $\bz_{\bk} \in \cN_{\bk}' \cap \overline{\cN}_{\bk}'$, and, for $\min \bk$ sufficiently large, we have 
\begin{equation} \label{uk_local_max}
	u_{\bh_{\bk}}(\bz_{\bk}) - \varphi(\bz_{\bk}) 
	> u_{\bh_{\bk}}(\bz_{\bk}') - \varphi ( \bz_{\bk}') 
	\text{ for all } \bz_{\bk}' \in \left( \cN_{\bk}' \cup \overline{\cN}_{\bk}' \right) \setminus \{ \bz_{\bk} \} . 
\end{equation}
Thus, 
\begin{subequations} \label{ops_bounds}
\begin{align}
	& \delta_{x_i, h_i^{(k_i)}}^- u_{\bh_{\bk}}(\bz_{\bk}) 
		\geq \delta_{x_i, h_i^{(k_i)}}^- \varphi(\bz_{\bk}) \to \varphi_{x_i}(\bx_0) , \label{e6.9a} \\ 
	& \delta_{x_i, h_i^{(k_i)}}^+ u_{\bh_{\bk}}(\bz_{\bk}) 
		\leq \delta_{x_i, h_i^{(k_i)}}^+ \varphi(\bz_{\bk}) \to \varphi_{x_i}(\bx_0) , \label{e6.9b} \\ 
	& \delta_{x_i, h_i^{(k_i)}}^2 u_{\bh_{\bk}}(\bz_{\bk}) 
		\leq \delta_{x_i, h_i^{(k_i)}}^2 \varphi(\bz_{\bk}) = \varphi_{x_i x_i}(\bx_0) , \\ 
	& \delta_{x_i, 2 h_i^{(k_i)}}^2 u_{\bh_{\bk}}(\bz_{\bk}) 
		\leq \delta_{x_i, 2 h_i^{(k_i)}}^2 \varphi(\bz_{\bk}) = \varphi_{x_i x_i}(\bx_0) , \\ 
	& \delta_{\xi_i^j, \bh_{\bk}}^2 u_{\bh_{\bk}}(\bz_{\bk}) 
		\leq \delta_{\xi_i^j, \bh_{\bk}}^2 \varphi(\bz_{\bk}) , \\ 
	& \delta_{\eta_i^j, \bh_{\bk}}^2 u_{\bh_{\bk}}(\bz_{\bk}) 
		\leq \delta_{\eta_i^j, \bh_{\bk}}^2 \varphi(\bz_{\bk})
\end{align}
\end{subequations}
for all $\min \bk$ sufficiently large.  

Using the fact that $\left| \cN_{\bk}' \cup \overline{\cN}_{\bk}' \right| = 3^d + 2d$, 
we can uniquely define the local interpolation function $u^{(\bk)}$ of $u_{\bh_{\bk}}$ by 
\begin{subequations} 
\begin{align}
	& u^{(\bk)} \in S , \\ 
	& u^{(\bk)} (\bz) = u_{\bh_{\bk}} (\bz) \text{ for all } \bz \in \cN_{\bk}' \cup \overline{\cN}_{\bk}' , 
\end{align}
\end{subequations}
where 
\begin{align*}
	S &\equiv Q_2 \cup \text{span }  \{ x_j^3, x_j^4;\ j=1,2, \cdots, d \}, \\  
	Q_2 &\equiv \text{span } \Bigl\{ p(x)=\sum_{\ell_j=0,1,2\atop 1\leq j\leq d}  a_{\ell_1\ell_2\cdots\ell_d}  x_1^{\ell_1} 	x_2^{\ell_2} \cdots x_d^{\ell_d} \Bigr\}.
\end{align*}
Thus, $S$ corresponds to adding monomials of degree 3 and 4 to the standard space $Q_2$ formed by the 
tensor product of polynomials with degree 2 or less.  For $d=2$, we have 
\begin{align*}
	u^{(\bk)}(\bx) 
	& = a_{0,0} + a_{1,0} x_1 + a_{2,0} x_1^2 + a_{0,1} x_2 + a_{0,2} x_2^2 \\ 
	& \qquad + a_{1,1} x_1 x_2 + a_{2,1} x_1^2 x_2 + a_{1,2} x_1 x_2^2 + a_{2,2} x_1^2 x_2^2 \\ 
	& \qquad + b_0 x_1^3 + b_1 x_1^4 + c_0 x_2^3 + c_1 x_2^4
\end{align*} 
for some unknown constants 
$a_{i,j}, b_\ell, c_\ell$ with $i, j \in \{0,1,2\}$, $\ell \in \{0,1\}$. 
Define $H^{(\bk)} \in \mathbb{R}^{d \times d}$ by $H^{(\bk)} = D^2 u^{(\bk)}(\bz_{\bk})$.  
We now consider cases based on whether $H^{(\bk)}$ is bounded or not.  
When the sequence is bounded, we will use the definition of the local neighborhood 
$\cN_{\bk}'$ to provide directional resolution by choosing the path $\bh_{\bk}$ appropriately. 
When the sequence is not bounded, we will exploit the lower regularity of the underlying limiting function $\ou$ 
combined with the ellipticity of $F$ and/or the numerical moment 
to help move derivatives onto $\varphi$.  


\smallskip 
\underline{\em Case {\rm (i)}}: $\{ H^{(\bk)} \}$ has a uniformly bounded subsequence. 

In this case, there exists a symmetric matrix $H \in \mathbb{R}^{d \times d}$ 
and a subsequence (not relabeled) such that $H^{(\bk)} \to H$. 
Using the facts that the following (central) discrete operators are all second-order accurate 
and $u^{(\bk)}$ is smooth for all $\bk$ with $D^2 u^{(\bk)} \to H$, 
we also have 
$D_{\bh_{\bk}}^2 u_{\bh_{\bk}}(\bz_{\bk}) \to H$,   
$\overline{D}_{\bh_{\bk}}^2 u_{\bh_{\bk}}(\bz_{\bk}) \to H$,   
$\widetilde{D}_{\bh_{\bk}}^2 u_{\bh_{\bk}}(\bz_{\bk}) \to H$,   
and $\widehat{D}_{\bh_{\bk}}^2 u_{\bh_{\bk}}(\bz_{\bk}) \to H$. 

We show $D^2 \varphi - H$ is symmetric positive semidefinite.  
By the symmetry of $D^2 \varphi - H$, there exists a unitary matrix $Q$ and a diagonal matrix $\Lambda$ 
such that $D^2 \varphi - H = Q \Lambda Q^T$.  
Let $\mathbf{q}_i$ denote the $i$th column of $Q$.  
Then, $\Lambda_{ii} = \mathbf{q}_i^T (D^2 \varphi - H) \mathbf{q}_i$.  
By the definition of $\cN_{\bk}'$, there exists a multi-index $\bk_0$ such that 
$\mathbf{d} \equiv \bz - \bz_{\bk_0}$ is essentially parallel to $\mathbf{q}_i$
for some $\bz \in \cN_{\bk_0}'$. 
Define 
\[
	\delta_{\mathbf{d}}^2 v(\bx) \equiv \frac{v(\bx-\mathbf{d}) - 2v(\bx) + v(\bx+\mathbf{d})}{\mathbf{d} \cdot \mathbf{d}}, 
\]
which can be proved to be a second order approximation of the directional 
derivative $\partial^2_{\mathbf{d}} v(\bx)$. 
Choose $\epsilon > 0$. Then, there exists a multi-index $\bk_0$ such that 
\begin{subequations}
\begin{align}
	& \Bigl| \mathbf{q}_i \cdot \left( D^2 \varphi(\bz_{\bk_0}) - H \right) \mathbf{q}_i 
		- \frac{1}{\mathbf{d} \cdot \mathbf{d}} \mathbf{d} \cdot \left( D^2 \varphi(\bz_{\bk_0}) - H \right) \mathbf{d} \Bigr| 
		< \epsilon/3 , \\ 
	& \Bigl| \frac{1}{\mathbf{d} \cdot \mathbf{d}} \mathbf{d} \cdot \left( H - H^{(\bk_0)} \right) \mathbf{d} \Bigr| < \epsilon / 3 , \\ 
	& \Bigl| \partial^2_{\mathbf{d}} (\varphi - u^{(\bk_0)})(\bz_{\bk_0}) 
		- \delta_{\mathbf{d}}^2 \left( \varphi - u^{(\bk_0)} \right)(\bz_{\bk_0}) \Bigr| < \epsilon/3 . 
\end{align}
\end{subequations}
Thus, there holds 
\begin{align*}
	\Lambda_{ii} 
	& = \mathbf{q}_i 
		\cdot \left( D^2 \varphi(\bz_{\bk_0}) - H \right) \mathbf{q}_i 
	\geq \frac{1}{\mathbf{d} \cdot \mathbf{d}} \mathbf{d} 
		\cdot \left( D^2 \varphi(\bz_{\bk_0}) - H \right) \mathbf{d} - \epsilon/3 \\ 
	& \geq \frac{1}{\mathbf{d} \cdot \mathbf{d}} \mathbf{d} 
		\cdot \left( D^2 \varphi - H^{(\bk_0)} \right) (\bz_{\bk_0}) \mathbf{d} - 2 \epsilon / 3 \\ 
	& = \frac{1}{\mathbf{d} \cdot \mathbf{d}} \mathbf{d} 
		\cdot \left( D^2 \varphi - D^2 u^{(\bk_0)} \right)(\bz_{\bk_0}) \mathbf{d} - 2 \epsilon / 3 
	= \partial^2_{\mathbf{d}} (\varphi - u^{(\bk_0)})(\bz_{\bk_0}) - 2 \epsilon / 3 \\ 
	& \geq \delta_{\mathbf{d}}^2 \left( \varphi - u^{(\bk_0)} \right)(\bz_{\bk_0}) - \epsilon \\ 
	& > - \epsilon
\end{align*}
by the definition of the interpolation function $u^{(\bk_0)}$ and \eqref{uk_local_max}. 
Therefore, $\Lambda_{ii} \geq 0$.  
Since $i \in \{1,2,\ldots,d\}$ was arbitrary, 
it follows that $D^2 \varphi \geq H$.  

There holds  
$\nabla_{\bh_{\bk}}^\pm u_{\bh_{\bk}}(\bz_{\bk}) \to \nabla \varphi(\bx_0)$ 
by \eqref{e6.9a}, \eqref{e6.9b}, and 
the fact that $D_{\bh_{\bk}}^2 u_{\bh_{\bk}}(\bz_{\bk})$ is convergent.  
Thus,  
\begin{align*}
	0 & = \lim_{\min \bk \to \infty} 
		\hF \left[ u_{\bh_{\bk}} , \bx_0 \right] \\ 
	& = F \left( H , \nabla \varphi(\bx_0) , \varphi(\bx_0) , \bx_0 \right) 
	\geq F \left( D^2 \varphi(\bx_0) , \nabla \varphi(\bx_0) , \varphi(\bx_0) , \bx_0 \right) 
\end{align*}
by the consistency of the scheme and the ellipticity of $F$.

\smallskip 
\underline{\em Case {\rm (ii)}}: $\{ H^{(\bk)} \}$ does not have a uniformly bounded subsequence.

Choose a grid function $\widetilde{u}_{\bh_{\bk}} \leq \varphi$, 
and define the local interpolation function $\widetilde{u}^{(\bk)}$ by 
\begin{subequations} 
\begin{align}
	& \widetilde{u}^{(\bk)} \in S , \\ 
	& \widetilde{u}^{(\bk)} (\bz) = u_{\bh_{\bk}} (\bz) \text{ for all } \bz \in \overline{\cN}_{\bk}' , \\ 
	& \widetilde{u}^{(\bk)} (\bz) = \widetilde{u}_{\bh_{\bk}} (\bz) \text{ for all } \bz \in \cN_{\bk}' \setminus \overline{\cN}_{\bk}' . 
\end{align}
\end{subequations}
Define $\widetilde{H}^{(\bk)} \in \mathbb{R}^{d \times d}$ by $\widetilde{H}^{(\bk)} = D^2 \widetilde{u}^{(\bk)}(\bz_{\bk})$.  

Suppose there exists a subsequence such that $\widetilde{H}^{(\bk)}$ is uniformly bounded.  
Then, we can apply the same arguments as in {\em Case {\rm (i)}} to the grid functions 
$u_{\bh_{\bk}}'$ defined by 
\[
	u_{\bh_{\bk}}'(\bz) \equiv \begin{cases}
	\widetilde{u}_{\bh_{\bk}} (\bz) & \text{ if } \bz \in \cN_{\bk}' \setminus \overline{\cN}_{\bk}'  , \\ 
	u_{\bh_{\bk}}(\bz) & \text{ otherwise} 
	\end{cases}
\]
for all $\bz \in \cN_{\bk}' \cup \overline{\cN}_{\bk}'$ 
to show that $F \left( D^2 \varphi(\bx_0) , \nabla \varphi(\bx_0) , \varphi(\bx_0) , \bx_0 \right) \leq 0$. 

For the remainder of Case (ii), 
we assume that for all choices of $\widetilde{u}_{\bh_{\bk}}$ the sequence $\widetilde{H}^{(\bk)}$ 
does not have a bounded subsequence.  
Then, there exists a pair of indices $(i,j)$ such that the sequence $[ \widetilde{D}_{\bh_{\bk}}^2 u_{\bh_{\bk}}(\bz_{\bk}) ]_{ij}$ 
or $[ \widehat{D}_{\bh_{\bk}}^2 u_{\bh_{\bk}}(\bz_{\bk}) ]_{ij}$
does not have a bounded subsequence.  
By the definition of the scheme and \eqref{hessian_diag}, there holds 
\begin{align}\label{eq6.12}
	0 & = \hF [ u_{\bh_{\bk}} , \bx_{0} ] \\ 
& = F \left( \overline{D}_{\bh_{\bk}}^2 u_{\bh_{\bk}} (\bz_{\bk}), \nabla_{\bh_{\bk}} u_{\bh_{\bk}} (\bz_{\bk}) , 
	u_{\bh_{\bk}}(\bx_{\bk}) , \bx_{\bk} \right)  \nonumber\\ 
	& \quad + 2 \gamma \sum_{i=1}^d \left( \delta_{x_i, 2 h_i^{(k_i)}}^2 u_{\bh_{\bk}} (\bz_{\bk}) 
		- \delta_{x_i, h_i^{(k_i)}}^2 u_{\bh_{\bk}} (\bz_{\bk}) \right) \nonumber \\ 
	& \quad - \gamma \sum_{i=1}^d \sum_{j=1 \atop j \neq i}^d 
		\left( \frac{h_i^{(k_i)}}{h_j^{(k_j)}} \delta_{x_i, h_i^{(k_i)}}^2 u_{\bh_{\bk}} (\bz_{\bk}) 
		+ \frac{h_j^{(k_j)}}{h_i^{(k_i)}} \delta_{x_j, h_j^{(k_j)}}^2 u_{\bh_{\bk}} (\bz_{\bk}) \right) \nonumber\\ 
	& \quad + \gamma \sum_{i=1}^d \sum_{j=1 \atop j \neq i}^d 
		\frac{\big( h_i^{(k_i)} \big)^2 + \big( h_j^{(k_j)} \big)^2}{2 h_i^{(k_i)} h_j^{(k_j)}}
		\left( \delta_{\xi_i^j, \bh_{\bk}}^2 u_{\bh_{\bk}}(\bz_{\bk}) + \delta_{\eta_i^j, \bh_{\bk}}^2 u_{\bh_{\bk}}(\bz_{\bk}) \right) \nonumber \\ 
	& \quad - \beta \vec{1} \cdot \left[ \nabla_{\bh_{\bk}}^+ - \nabla_{\bh_{\bk}}^- \right]_i u_{\bh_{\bk}}(\bx_{0}) . 
	\nonumber
\end{align}
By the choice of $\gamma$ and $\beta$ and \eqref{ops_bounds}, we have there exists an index $\ell$ such that 
$\delta_{x_\ell, h_\ell^{(k_\ell)}}^2 u_{\bh_{\bk}}(\bz_{\bk}) \to - \infty$ 
if one of the following three limits holds: 
$\delta_{\xi_i^j, \bh_{\bk}}^2 u_{\bh_{\bk}}(\bz_{\bk}) \to - \infty$,  
$\delta_{\eta_i^j, \bh_{\bk}}^2 u_{\bh_{\bk}}(\bz_{\bk}) \to - \infty$, 
or $\delta_{x_\ell, 2 h_\ell^{(k_\ell)}}^2 u_{\bh_{\bk}}(\bz_{\bk}) \to - \infty$. 
Therefore, there exists an index $\ell$ such that the sequence 
$\delta_{x_{\ell}, h_{\ell}^{(k_{\ell})}}^2 u_{\bh_{\bk}}(\bz_{\bk})$ 
does not have a bounded subsequence as $\min \bk \to \infty$.

Our aim now is to move the approximate derivatives onto $\varphi$ in \eqref{eq6.12}.
To proceed, we first choose sequences $\{ \bh_{\bk} \}$ and $\{ \bx_{\bk} \}$ that maximize the 
rate at which $\delta_{x_{\ell}, h_{\ell}^{(k_{\ell})}}^2 u_{\bh_{\bk}}(\bx_{\bk})$ diverges.  

Suppose $\lim_{\min \bk \to \infty} \delta_{x_\ell, h_\ell^{(k_\ell)}}^2 u_{\bh_{\bk}} (\bx_{\bk}) = - \infty$. 
Choose the function $f: [0, \infty) \to [0, \infty]$ such that 
\begin{subequations} \label{f_rate}
\begin{align}
	& f(0) = 0 , \\ 
	& f \text{ is strictly increasing} , \\ 
	& 0 > \liminf_{\min \bk \to \infty} 
		f\left(h_\ell^{(k_\ell)}\right) \delta_{x_\ell, h_\ell^{(k_\ell)}}^2 u_{\bh_{\bk}} (\bx_{\bk}) > - \infty . 
\end{align}
\end{subequations}
Then, there exists a function $f_\ell$ satisfying \eqref{f_rate} such that 
\begin{equation} \label{f1_rate}
	\limsup_{h \to 0^+} \frac{f_\ell(h)}{f(h)} < \infty \text{ for all choices of $f$ satisfying \eqref{f_rate}} 
\end{equation} 
and all sequences $\bh_{\bk} \to \mathbf{0}^+$ and $\bx_{\bk} \to \bx_0$ such that 
$\delta_{x_{\ell}, h_{\ell}^{(k_{\ell})}}^2 u_{\bh_{\bk}}(\bx_{\bk}) \to - \infty$.
Thus, $(f_\ell)^{-1}$ represents the rate at which $\delta_{x_\ell, h_\ell^{(k_\ell)}}^2 u_{\bh_{\bk}} (\bx_{\bk})$ diverges 
for the choice of sequences $\bh_{\bk} \to \mathbf{0}^+$ and $\bx_{\bk} \to \bx_0$ that maximize the rate at which 
$\delta_{x_\ell, h_\ell^{(k_\ell)}}^2 u_{\bh_{\bk}} (\bx_{\bk})$ diverges.  
(Note that this may not be the same choice of sequences used to form $H^{(k)}$ since we now choose the sequence 
that blows up at the fastest rate.) 
Furthermore, we have there exists a constant $C_\ell > 0$ such that 
\begin{equation} \label{rate_ou}
	\liminf_{\min \bk \to \infty} 
		f_\ell \left(h_\ell^{(k_\ell)}\right) \delta_{x_\ell, h_\ell^{(k_\ell)}}^2 u_{\bh_{\bk}} (\bx_{\bk}) = - C_\ell . 
\end{equation}
Therefore, there exists subsequences (not relabelled) such that 
\begin{equation} \label{rate_uh}
	\lim_{\min \bk \to \infty} f_\ell \left(h_\ell^{(k_\ell)} \right) \delta_{x_{k_\ell}, h_\ell^{(k_\ell)}}^2 u_{\bh_{\bk}} (\bx_{\bk}) 
	= - C_\ell . 
\end{equation}

We can also show that the $(\ell, \ell)$ component of the numerical moment is nonnegative.  
By \eqref{rate_uh}, \eqref{rate_ou}, and \eqref{moment_ii_2h}, we have 
\begin{align}\label{momentii} 
	& \liminf_{\min \bk \to \infty} f_\ell \left(h_\ell^{(k_\ell)} \right) 
		\left[ \widetilde{D}_{\bh_{\bk}}^2 u_{\bh_{\bk}}(\bx_{\bk}) 
			- \widehat{D}_{\bh_{\bk}}^2 u_{\bh_{\bk}}(\bx_{\bk}) \right]_{\ell \ell} \\ 
	\nonumber & \qquad \liminf_{\min \bk \to \infty} f_\ell \left(h_\ell^{(k_\ell)} \right) 
		\bigg( \frac12 \delta_{x_{k_\ell}, h_\ell^{(k_\ell)}}^2 u_{\bh_{\bk}} (\bx_{\bk} - h_\ell^{(k_\ell)} \mathbf{e}_\ell)  
			- \delta_{x_{k_\ell}, h_\ell^{(k_\ell)}}^2 u_{\bh_{\bk}} (\bx_{\bk}) \\ 
			\nonumber & \qquad \qquad 
			+ \frac12 \delta_{x_{k_\ell}, h_\ell^{(k_\ell)}}^2 u_{\bh_{\bk}} (\bx_{\bk} + h_\ell^{(k_\ell)} \mathbf{e}_\ell) \bigg) \\ 
	\nonumber & \qquad \geq - \frac12 C_\ell + C_\ell  - \frac12 C_\ell = 0 . 
\end{align} 
Consequently, 
\begin{equation} \label{momentii_bound}
	f_\ell \left(h_\ell^{(k_\ell)} \right) 
		\left[ \widetilde{D}_{\bh_{\bk}}^2 u_{\bh_{\bk}}(\bx_{\bk}) 
			- \widehat{D}_{\bh_{\bk}}^2 u_{\bh_{\bk}}(\bx_{\bk}) \right]_{\ell \ell} 
	\geq - \frac{k_{**}}{4\gamma} C_\ell 
\end{equation}
for all $\min \bk$ sufficiently large.

For the remainder of case (ii) we will use the following strategy.  
By the definition of the scheme, there holds 
\begin{align*}
	0 = \hF [ u_{\bh_{\bk}} , \bx_{\bk} ] \implies 
	0 = f_\ell \left( h_\ell^{(k_\ell)} \right) \hF [ u_{\bh_{\bk}} , \bx_{\bk} ] . 
\end{align*}
We will exploit the blow-up in $\delta_{x_{k_\ell}, h_\ell^{(k_\ell)}}^2 u_{\bh_{\bk}} (\bx_{\bk})$ by 
letting $\min \bk$ be sufficiently large and sending $k_\ell \to \infty$.  
Using the structure of $\hF$ and the choice of the sequences $\{ \bh_{\bk} \}$ and $\{ \bx_{\bk} \}$, 
we will be able to show that, for $k_\ell >> k_j$ for all $j \neq \ell$, there holds 
\[
	f_\ell \left( h_\ell^{(k_\ell)} \right) \hF [ u_{\bh_{\bk}} , \bx_{\bk} ] 
	> f_\ell \left( h_\ell^{(k_\ell)} \right) F \left( D^2 \varphi(\bx_0) , \nabla \varphi(\bx_0) , \varphi(\bx_0) , \bx_{\bk} \right)
\] 
when $\min \bk$ is sufficiently large. 
The bound $0 \geq F[\varphi](\bx_0)$ follows since $f_\ell \left( h_\ell^{(k_\ell)} \right) > 0$ and $\bx_{\bk} \to \bx_0$.  

In order to move derivatives onto $\varphi$, we can simply use the elliptic structure and Lipschitz continuity of $F$ 
while exploiting \eqref{momentii} and the fact that 
\begin{equation} \label{f1d10}
	-f_\ell \left(h_\ell^{(k_\ell)} \right) \delta_{x_{k_\ell}, h_\ell^{(k_\ell)}}^2 u_{\bh_{\bk}} (\bx_{\bk}) \geq \frac12 C_\ell
\end{equation}
for all $\min \bk$ sufficiently large.   
Indeed, using \eqref{moment_ii_2h} and \eqref{viscosity_h} in \eqref{hatF}, there holds 
{\small 
\begin{align*}
	0 & = \hF [ u_{\bh_{\bk}} , \bx_{\bk} ] \\ 
	& = F \left( D^2 \varphi(\bx_0) , \nabla \varphi(\bx_0) , \varphi(\bx_0) , \bx_{\bk} \right) \\ 
	& \quad + \hF [ u_{\bh_{\bk}} , \bx_{\bk} ] 
		- F \left( D^2 \varphi(\bx_0) , \nabla \varphi(\bx_0) , \varphi(\bx_0) , \bx_{\bk} \right) \\ 
	& = F \left( D^2 \varphi(\bx_0) , \nabla \varphi(\bx_0) , \varphi(\bx_0) , \bx_{\bk} \right) \\ 
	& \quad + F \left( \overline{D}_{\bh_{\bk}}^2 u_{\bh_{\bk}}(\bx_{\bk}) , 
		\overline{\nabla}_{\bh_{\bk}} u_{\bh_{\bk}}(\bx_{\bk}) , u_{\bh_{\bk}}(\bx_{\bk}) , \bx_{\bk} \right) 		
		- F \left( D^2 \varphi(\bx_0) , \nabla \varphi(\bx_0) , \varphi(\bx_0) , \bx_{\bk} \right) \\ 
	& \quad - \beta \sum_{i \neq \ell}^d h_{i}^{(k_i)} \delta_{x_i, h_i^{(k_i)}}^2 u_{\bh_{\bk}}(\bx_{\bk}) 
		+ \gamma \sum_{i=1}^d \sum_{j=1 \atop (i,j) \neq (\ell,\ell)}^d \left( \left[ \widetilde{D}_{\bh_{\bk}}^2 \right]_{ij} 
				- \left[ \widehat{D}_{\bh_{\bk}}^2 \right]_{ij} \right) u_{\bh_{\bk}}(\bx_{\bk}) \\ 
	& \quad 
		- \beta h_{\ell}^{(k_\ell)} \delta_{x_\ell, h_\ell^{(k_\ell)}}^2 u_{\bh_{\bk}}(\bx_{\bk})
		+ \gamma \left( \left[ \widetilde{D}_{\bh_{\bk}}^2 \right]_{\ell \ell} 
				- \left[ \widehat{D}_{\bh_{\bk}}^2 \right]_{\ell \ell} \right) u_{\bh_{\bk}}(\bx_{\bk}) \\ 
	& \geq F \left( D^2 \varphi(\bx_0) , \nabla \varphi(\bx_0) , \varphi(\bx_0) , \bx_{\bk} \right) \\ 
	& \quad - 2(K^{**} + \gamma) \sum_{i=1}^d \sum_{j=1 \atop (i,j) \neq (\ell,\ell)}^d
		\left( \left| \left[ \widetilde{D}_{\bh_{\bk}}^2 u_{\bh_{\bk}}(\bx_{\bk}) \right]_{ij} \right| 
			+ \left| \left[ \widehat{D}_{\bh_{\bk}}^2 u_{\bh_{\bk}}(\bx_{\bk}) \right]_{ij} \right| 
			+ \left| \left[ D^2 \varphi(\bx_0) \right]_{ij} \right| \right) \\ 
	& \quad - K^{**} \left| \varphi_{x_\ell x_\ell} (\bx_0) \right| 
		- K^* \sum_{i=1}^d \left( \left| \delta_{x_i, h_i^{(k_i)}} u_{\bh_{\bk}}(\bx_{\bk}) \right| 
			+ \left| \varphi_{x_i} (\bx_0) \right| \right) \\ 
	& \quad - K^0 \left( \left| u_{\bh_{\bk}}(\bx_{\bk}) \right| + \left| \varphi(\bx_0) \right| \right) 
		- \beta \sum_{i \neq \ell}^d \left| h_{i}^{(k_i)} \delta_{x_i, h_i^{(k_i)}}^2 u_{\bh_{\bk}}(\bx_{\bk}) \right| 
		 \\ 
	& \quad 
		- \left( k_{**} + \beta h_{\ell}^{(k_\ell)} \right) \delta_{x_\ell, h_\ell^{(k_\ell)}}^2 u_{\bh_{\bk}}(\bx_{\bk})
		+ (2 \gamma - K^{**}) \left( \delta_{x_{k_\ell}, 2 h_\ell^{(k_\ell)}}^2 - \delta_{x_{k_\ell}, h_\ell^{(k_\ell)}}^2 \right) u_{\bh_{\bk}} (\bx_{\bk}) . 
\end{align*}
}
Let $h_{\bk}' = \min_{i \neq \ell} h_i^{(k_i)}$. 
Since $\| D^2 \varphi (\bx_0)\|_{\max}$, $\| \nabla \varphi(\bx_0) \|_{\max}$, $|\varphi(\bx_0)|$, 
and $| u_{\bh_{\bk}} (\bx_{\bk})|$ 
are uniformly bounded, 
there exists a constant $M > 0$ such that  
\begin{align} \label{case2_eq}
	0 & = f_\ell \left( h_\ell^{(k_\ell)} \right) \hF [ u_{\bh_{\bk}} , \bx_{\bk} ] \\ 
	\nonumber & 
	\geq f_\ell \left( h_\ell^{(k_\ell)} \right) F \left( D^2 \varphi(\bx_0) , \nabla \varphi(\bx_0) , \varphi(\bx_0) , \bx_{\bk} \right) 
		- f_\ell \left( h_\ell^{(k_\ell)} \right) \frac{M}{\left( h_{\bk}' \right)^2} \\ 
	\nonumber & \quad 
		- 4(K^{**} + \gamma) f_\ell \left( h_\ell^{(k_\ell)} \right) \sum_{i \neq \ell}^d
		\left( \left| \left[ \widetilde{D}_{\bh_{\bk}}^2 u_{\bh_{\bk}}(\bx_{\bk}) \right]_{\ell i} \right| 
			+ \left| \left[ \widehat{D}_{\bh_{\bk}}^2 u_{\bh_{\bk}}(\bx_{\bk}) \right]_{\ell i} \right| \right) \\ 
	\nonumber & \quad 
		- K^* f_\ell \left( h_\ell^{(k_\ell)} \right) \left| \delta_{x_\ell, h_\ell^{(k_\ell)}} u_{\bh_{\bk}}(\bx_{\bk}) \right| 
		+ \frac14 \left( k_{**} + \beta h_{\ell}^{(k_\ell)} \right) C_\ell  
\end{align}
for all $\min \bk$ sufficiently large by \eqref{momentii_bound} and \eqref{f1d10}. 
In order to bound the negative terms in \eqref{case2_eq} by $\frac14 \left( k_{**} + \beta h_{\ell}^{(k_\ell)} \right) C_\ell$, 
we will consider three subcases depending on how the function $f_\ell(h)$ behaves as $h \to 0^+$.

\smallskip 
\underline{\em Subcase {\rm (iia)}}: Suppose $\liminf_{k_\ell \to \infty} \frac{f_\ell \left( h_\ell^{(k_\ell)}\right)}{h_\ell^{(k_\ell)}} 
= 0$.  Then, by \eqref{discrete_hess_explicit} and the stability of the scheme, 
there exists a constant $M' > 0$ independent of $h_{\ell}^{(k_\ell)}$ and a subsequence (not relabelled) such that
\begin{subequations}
\begin{align}
	& \frac{f_\ell \left( h_\ell^{(k_\ell)} \right)}{h_\ell^{(k_\ell)}} 
		\left| h_\ell^{(k_\ell)} \left[ \widetilde{D}_{\bh_{\bk}}^2 u_{\bh_{\bk}}(\bx_{\bk}) \right]_{\ell i} \right| 
		\leq \frac{f_\ell \left( h_\ell^{(k_\ell)} \right)}{h_\ell^{(k_\ell)}} \frac{M'}{ h_{\bk}' } , \\ 
	& \frac{f_\ell \left( h_\ell^{(k_\ell)} \right)}{h_\ell^{(k_\ell)}} 
		\left| h_\ell^{(k_\ell)} \left[ \widehat{D}_{\bh_{\bk}}^2 u_{\bh_{\bk}}(\bx_{\bk}) \right]_{\ell i} \right| 
		\leq \frac{f_\ell \left( h_\ell^{(k_\ell)} \right)}{h_\ell^{(k_\ell)}} \frac{M'}{ h_{\bk}' } , \\ 
	& \frac{f_\ell \left( h_\ell^{(k_\ell)} \right)}{h_\ell^{(k_\ell)}} 
		\left| h_\ell^{(k_\ell)} \delta_{x_\ell, h_\ell^{(k_\ell)}} u_{\bh_{\bk}}(\bx_{\bk}) \right| 
		\leq \frac{f_\ell \left( h_\ell^{(k_\ell)} \right)}{h_\ell^{(k_\ell)}} M' 
\end{align}
\end{subequations}
for all $i \neq \ell$ and $\min \bk$ sufficiently large. 
Therefore, there exists indices $K_0$, $K_\ell$ with $K_\ell >> K_0$ such that for $k_i = K_0$ for all $i \neq \ell$, 
there holds 
\begin{align} \label{2abound}
	\frac{f_\ell \left( h_\ell^{(k_\ell)} \right)}{h_\ell^{(k_\ell)}} 
	\bigg[ 8 d \left( 1 + K^{**} + \gamma + K^* \right) \frac{M + M'}{ \left( h_{\bk}' \right)^2 } \bigg] 
	< \frac{k_{**} C_\ell}{8}
\end{align} 
for all $k_\ell > K_\ell$.  
Plugging \eqref{2abound} into \eqref{case2_eq}, it follows that, for some index $\bk$, 
\begin{align*} 
	0 & \geq f_\ell \left( h_\ell^{(k_\ell)} \right) F \left( D^2 \varphi(\bx_0) , \nabla \varphi(\bx_0) , \varphi(\bx_0) , \bx_{\bk} \right) 
		+ \frac18 \left( k_{**} + \beta h_\ell^{(k_\ell)} \right) C_\ell  \\ 
	& >  f_\ell \left( h_\ell^{(k_\ell)} \right) F \left( D^2 \varphi(\bx_0) , \nabla \varphi(\bx_0) , \varphi(\bx_0) , \bx_{\bk} \right) 
\end{align*}
from which we can conclude $0 \geq F \left( D^2 \varphi(\bx_0) , \nabla \varphi(\bx_0) , \varphi(\bx_0) , \bx_0 \right)$. 

\smallskip 
\underline{\em Subcase {\rm (iib)}}:  Suppose 
$0 < C_b \leq \liminf_{k_\ell \to \infty} \frac{f_\ell \left( h_\ell^{(k_\ell)}\right)}{h_\ell^{(k_\ell)}} \leq C_B < \infty$ 
for some constants $C_b$, $C_B$.  
  
Assume that there exists sequences $\eta_{\bk} \to \mathbf{0}^+$ and $\bz_{\bk} \to \bx_0$ such that 
\[
	\limsup_{\min \bk \to \infty} \big( 
	\left|  u_{\eta_{\bk}}(\bz_{\bk}) - u_{\eta_{\bk}}(\bz_{\bk} - \eta_{k_\ell} \mathbf{e}_\ell) \right| 
	+ \left|  u_{\eta_{\bk}}(\bz_{\bk}) - u_{\eta_{\bk}}(\bz_{\bk} + \eta_{k_\ell} \mathbf{e}_\ell) \right| \big) > 0 . 
\] 
Then, by the upper semi-continuity and definition of $\ou$, there exists a choice for the sequences 
$\eta_{\bk} \to \mathbf{0}^+$ and $\bz_{\bk} \to \bx_0$ and a constant $C_0 > 0$ such that 
\[
	\lim_{\min \bk \to \infty} \left[ u_{\eta_{\bk}}(\bz_{\bk} - \eta_{k_\ell} \mathbf{e}_\ell) - 2 u_{\eta_{\bk}}(\bz_{\bk}) 
		+ u_{\eta_{\bk}}(\bz_{\bk} + \eta_{k_\ell} \mathbf{e}_\ell) \right] = - C_0 , 
\]
and it follows that 
\begin{align*}
	& \limsup_{\min \bk \to \infty} f_\ell \left( \eta_\ell^{(k_\ell)} \right) 
		\delta_{x_\ell, \eta_\ell^{(k_\ell)}}^2 u_{\eta_{\bk}}(\bz_{\bk}) \\ 
	& \quad = \limsup_{\min \bk \to \infty} \frac{f_\ell \left( \eta_\ell^{(k_\ell)} \right)}{\eta_\ell^{(k_\ell)}} 
		\left[ \left(\eta_\ell^{(k_\ell)} \right)^2 \delta_{x_\ell, \eta_\ell^{(k_\ell)}}^2 u_{\eta_{\bk}}(\bz_{\bk}) \right] 
			\frac{1}{\eta_\ell^{(k_\ell)}} \\ 
	& \quad \leq - C_b \frac{C_0}{2} \liminf_{\min \bk \to \infty} \frac{1}{\eta_\ell^{(k_\ell)}} 
	= - \infty , 
\end{align*}
a contradiction to \eqref{rate_ou}.  
Thus, we must have 
\begin{equation} \label{ou_cont}
	\limsup_{\min \bk \to \infty} \big( 
	\left|  u_{\eta_{\bk}}(\bz_{\bk}) - u_{\eta_{\bk}}(\bz_{\bk} - \eta_{k_\ell} \mathbf{e}_\ell) \right| 
	+ \left|  u_{\eta_{\bk}}(\bz_{\bk}) - u_{\eta_{\bk}}(\bz_{\bk} + \eta_{k_\ell} \mathbf{e}_\ell) \right| \big) = 0 . 
\end{equation}

Applying \eqref{ou_cont}, we have 
\begin{align*}
	0 = \limsup_{\min \bk \to \infty} \bigg( & 
	\left| u_{\bh_{\bk}} (\bx_{\bk} + h_\ell^{(k_\ell)} \mathbf{e}_\ell \pm h_i^{(k_i)} \mathbf{e}_i ) 
		- u_{\bh_{\bk}} (\bx_{\bk} \pm h_i^{(k_i)} \mathbf{e}_i ) \right| \\ 
	& \qquad + \left| u_{\bh_{\bk}} (\bx_{\bk} - h_\ell^{(k_\ell)} \mathbf{e}_\ell \pm h_i^{(k_i)} \mathbf{e}_i ) 
		- u_{\bh_{\bk}} (\bx_{\bk}  \pm h_i^{(k_i)} \mathbf{e}_i ) \right| \\ 
	& \qquad + \left| u_{\bh_{\bk}} (\bx_{\bk} \pm h_\ell^{(k_\ell)} \mathbf{e}_\ell ) 
		- u_{\bh_{\bk}} (\bx_{\bk} ) \right| \bigg) . 
\end{align*}
Therefore, by \eqref{discrete_hess_explicit}, 
there exists indices $K_0$, $K_\ell$ with $K_\ell >> K_0$ such that for $k_i = K_0$ for all $i \neq \ell$, 
there holds 
\begin{subequations} \label{2bbound}
\begin{align}
	& \frac{f_\ell \left( h_\ell^{(k_\ell)} \right)}{h_\ell^{(k_\ell)}} 
		\left| h_\ell^{(k_\ell)} \left[ \widetilde{D}_{\bh_{\bk}}^2 u_{\bh_{\bk}}(\bx_{\bk}) \right]_{\ell i} \right| 
		\leq 2 C_B \left| h_\ell^{(k_\ell)} \left[ \widetilde{D}_{\bh_{\bk}}^2 u_{\bh_{\bk}}(\bx_{\bk}) \right]_{\ell i} \right| \\ 
		\nonumber & \qquad \leq \frac{k_{**}C_\ell}{64d(K^{**} + \gamma)} , \\ 
	& \frac{f_\ell \left( h_\ell^{(k_\ell)} \right)}{h_\ell^{(k_\ell)}} 
		\left| h_\ell^{(k_\ell)} \left[ \widehat{D}_{\bh_{\bk}}^2 u_{\bh_{\bk}}(\bx_{\bk}) \right]_{\ell i} \right| 
		\leq 2 C_B \left| h_\ell^{(k_\ell)} \left[ \widehat{D}_{\bh_{\bk}}^2 u_{\bh_{\bk}}(\bx_{\bk}) \right]_{\ell i} \right| \\ 
		\nonumber & \qquad \leq \frac{k_{**}C_\ell}{64d(K^{**} + \gamma)} , \\ 
	& \frac{f_\ell \left( h_\ell^{(k_\ell)} \right)}{h_\ell^{(k_\ell)}} 
		\left| h_\ell^{(k_\ell)} \delta_{x_\ell, h_\ell^{(k_\ell)}} u_{\bh_{\bk}}(\bx_{\bk}) \right| 
		\leq 2 C_B \left| h_\ell^{(k_\ell)} \delta_{x_\ell, h_\ell^{(k_\ell)}} u_{\bh_{\bk}}(\bx_{\bk}) \right| \\ 
		\nonumber & \qquad 	\leq \frac{k_{**}C_\ell}{64d K^{*}} , \\ 
	& f_\ell \left( h_\ell^{(k_\ell)} \right) \frac{M}{\left( h_{\bk}' \right)^2} 
		\leq \frac{k_{**}C_\ell}{64d} 
\end{align}
\end{subequations}
for all $k_\ell > K_\ell$ 
for some subsequence (not relabelled) such that 
\[
	\limsup_{k_\ell \to \infty} \frac{f_\ell \left( h_\ell^{(k_\ell)}\right)}{h_\ell^{(k_\ell)}} \leq C_B . 
\]  
Plugging \eqref{2bbound} into \eqref{case2_eq}, it follows that, for some index $\bk$, 
\begin{align*} 
	0 & \geq f_\ell \left( h_\ell^{(k_\ell)} \right) F \left( D^2 \varphi(\bx_0) , \nabla \varphi(\bx_0) , \varphi(\bx_0) , \bx_{\bk} \right) 
		+ \frac18 \left( k_{**} + \beta h_{\ell}^{(k_\ell)} \right) C_\ell  \\ 
	& >  f_\ell \left( h_\ell^{(k_\ell)} \right) F \left( D^2 \varphi(\bx_0) , \nabla \varphi(\bx_0) , \varphi(\bx_0) , \bx_{\bk} \right) 
\end{align*}
from which we can conclude $0 \geq F \left( D^2 \varphi(\bx_0) , \nabla \varphi(\bx_0) , \varphi(\bx_0) , \bx_0 \right)$. 

\smallskip 
\underline{\em Subcase {\rm (iic)}}: Suppose $\lim_{k_\ell \to \infty} \frac{f_\ell \left( h_\ell^{(k_\ell)}\right)}{h_\ell^{(k_\ell)}} 
= \infty$.  Without a loss of generality, we also assume 
$\lim_{k_j \to \infty} \frac{f_j \left( h_j^{(k_j)}\right)}{h_j^{(k_j)}} = \infty$ for all $j \neq \ell$ 
for which there exists sequences $\eta_{\bk} \to \mathbf{0}^+$ and $\bz_{\bk} \to \bx_0$ such that 
$\liminf_{\min \bk \to \infty} \delta_{x_j, \eta_j^{(k_j)}}^2 u_{\eta_{\bk}}(\bz_{\bk}) = - \infty$.  
Otherwise, we could repeat the same arguments in Subcase {\rm (iia)} or Subcase {\rm (iib)} 
with the index $\ell$ replaced by the index $j$ to show that 
$0 \geq F \left( D^2 \varphi(\bx_0) , \nabla \varphi(\bx_0) , \varphi(\bx_0) , \bx_0 \right)$ 
since we would analogously have 
\begin{align*} 
	0 & >  f_j \left( h_j^{(k_j)} \right) F \left( D^2 \varphi(\bx_0) , \nabla \varphi(\bx_0) , \varphi(\bx_0) , \bx_{\bk} \right) 
\end{align*}
for some index $\bk$ with $f_j$ defined by \eqref{f1_rate}.  
Consequently, by Subcase {\rm (iib)}, we can assume $\ou$ is continuous at $\bx_0$. 

By assumption, there exists a function $g_\ell$ such that 
\begin{subequations} \label{g_rate}
\begin{align}
	& g_\ell(0) = 0 , \\ 
	& g_\ell \text{ is strictly increasing} , \\ 
	& C_c \leq \liminf_{h \to 0^+} \frac{f_\ell(h) g_\ell(h)}{h} \leq \limsup_{h \to 0^+} \frac{f_\ell(h) g_\ell(h)}{h} \leq C_C 
\end{align}
\end{subequations} 
for some constants $0 < C_c \leq C_C < \infty$, 
i.e., $f_\ell g_\ell \in \mathcal{O}(h)$. 
We first show there holds  
$\limsup_{\min \bk \to \infty} 
\left| \delta_{x_\ell, h_\ell^{(k_\ell)}}^+ u_{\bh_{\bk}}(\bx_{\bk}) - \delta_{x_\ell, h_\ell^{(k_\ell)}}^- u_{\bh_{\bk}}(\bx_{\bk}) \right| = 0$.  

Suppose there exists sequences $\eta_{\bk} \to \mathbf{0}^+$ and $\bz_{\bk} \to \bx_0$ and a constant $C'$ such that 
\[
	\lim_{\min \bk \to \infty} \big( \delta_{x_\ell, \eta_\ell^{(k_\ell)}}^+ u_{\eta_{\bk}}(\bz_{\bk}) 
		- \delta_{x_\ell, \eta_\ell^{(k_\ell)}}^- u_{\eta_{\bk}}(\bz_{\bk}) \big) 
		= -C' \neq 0 , 
\]
i.e., there exists a path approaching $\bx_0$ over which $\ou$ would have a corner along the $x_\ell$ direction.  
Then, since $\ou$ can be touched from above by a smooth function, we must have $C' > 0$, 
and it follows that 
\begin{align*}
	& \limsup_{\min \bk \to \infty} 
		f_\ell \left( \eta_\ell^{(k_\ell)} \right) \delta_{x_\ell, \eta_\ell^{(k_\ell)}}^2 u_{\eta_{\bk}}(\bz_{\bk}) \\ 
	& \quad = \limsup_{\min \bk \to \infty} \frac{f_\ell \left( \eta_\ell^{(k_\ell)} \right)}{\eta_\ell^{(k_\ell)}} 
		\left[ \delta_{x_\ell, \eta_\ell^{(k_\ell)}}^+ u_{\eta_{\bk}}(\bz_{\bk}) 
		- \delta_{x_\ell, \eta_\ell^{(k_\ell)}}^-u_{\eta_{\bk}}(\bz_{\bk}) \right] \\ 
	& \quad = \limsup_{\min \bk \to \infty} 
		\frac{f_\ell \left( \eta_\ell^{(k_\ell)} \right) g_\ell \left( \eta_\ell^{(k_\ell)} \right)}{\eta_\ell^{(k_\ell)}} 
		\left[ \delta_{x_\ell, \eta_\ell^{(k_\ell)}}^+ u_{\eta_{\bk}}(\bz_{\bk}) 
		- \delta_{x_\ell, \eta_\ell^{(k_\ell)}}^- u_{\eta_{\bk}}(\bz_{\bk}) \right] 
		\frac{1}{g_\ell \left( \eta_\ell^{(k_\ell)} \right)} \\ 
	& \quad \geq - \frac{C_c C'}{2} \liminf_{\min \bk \to \infty} \frac{1}{g_\ell \left( \eta_\ell^{(k_\ell)} \right)} 
	= - \infty , 
\end{align*}
a contradiction to \eqref{rate_ou}.  
Therefore, we must have 
\begin{equation} \label{oux1}
	\limsup_{\min \bk \to \infty} 
	\left|  \delta_{x_\ell, \eta_\ell^{(k_\ell)}}^+ u_{\eta_{\bk}}(\bz_{\bk}) - \delta_{x_\ell, \eta_\ell^{(k_\ell)}}^- u_{\eta_{\bk}}(\bz_{\bk}) \right| = 0  
\end{equation}
for all sequences $\eta _{\bk} \to \mathbf{0}^+$ and $\bz_{\bk} \to \bx_0$.  

For the remainder of Subcase {\rm (iic)} we choose the maximizing sequences 
$\bh_{\bk}$, $\bx_{\bk}$, and $\bz_{\bk}$ constructed in \eqref{max_seq} and satisfying \eqref{ops_bounds} 
with $\liminf_{\min \bk \to \infty} \delta_{x_{\ell}, h_{\ell}^{(k_{\ell})}}^2 u_{\bh_{\bk}}(\bz_{\bk}) = - \infty$ 
as discussed at the beginning of Case {\rm (ii)}.  
The strategies in Subcases {\rm (iia)} and {\rm (iib)}
relied upon having a sequence diverge at a sufficiently high rate.  
Since no such sequence exists, we can exploit the extra structure of the maximizing sequence 
without having non-signed terms dominate in \eqref{case2_eq}.   
In particular, we will use the fact that $\ou$ is touched from above by a smooth function at $\bx_0$ 
to show $\nabla_{\bh_{\bk}} u_{\bh_{\bk}}(\bx_{\bk}) \to \nabla \varphi(\bx_0)$.  

Notationally, we let $\widetilde{f}_\ell$ denote the rate function and select a subsequence such that 
\begin{align*}
	\lim_{\min \bk \to \infty} \widetilde{f}_\ell \left( h_\ell^{(k_\ell)} \right) 
		\delta_{x_\ell, h_\ell^{(k_\ell)}}^2 u_{\bh_{\bk}}(\bz_{\bk}) 
	= \liminf_{\min \bk \to \infty} \widetilde{f}_\ell \left( h_\ell^{(k_\ell)} \right)  
		\delta_{x_\ell, h_\ell^{(k_\ell)}}^2 u_{\bh_{\bk}}(\bz_{\bk}) 
	= - \widetilde{C}_\ell 
\end{align*}
for some constant $\widetilde{C}_\ell > 0$.  
By \eqref{ops_bounds} and \eqref{oux1}, we must have 
\begin{equation} \label{iic-grad}
	\lim_{\min \bk \to \infty} 
	\delta_{x_i, h_i^{(k_i)}}^\pm u_{\bh_{\bk}}(\bz_{\bk} ) = \varphi_{x_i}(\bx_0)  
\end{equation} 
for all $i=1,2,\ldots,d$.  

We now show that 
$\left| \delta_{x_\ell, h_\ell^{(k_\ell)}}^+ u_{\bh_{\bk}}(\bz_{\bk} \pm h_i^{(k_i)} \mathbf{e}_i) \right|$ 
is uniformly bounded.  
Suppose there exists an index $i \neq \ell$ such that 
\[
	\limsup_{\min \bk \to \infty} 
	\left| \delta_{x_\ell, h_\ell^{(k_\ell)}}^+ u_{\bh_{\bk}}(\bz_{\bk} \pm h_i^{(k_i)} \mathbf{e}_i) \right| = \infty .  
\]
Assume 
$\lim_{\min \bk \to \infty} \delta_{x_\ell, h_\ell^{(k_\ell)}}^+ u_{\bh_{\bk}}(\bz_{\bk} + h_i^{(k_i)} \mathbf{e}_i) = \infty$ 
for some subsequence (not relabelled).  
The case when the sequence diverges to $- \infty$ will be discussed in the later.  

A simple computation reveals 
\begin{align}\label{iic_blowup}
	\delta_{x_\ell, h_\ell^{(k_\ell)}}^+ u_{\bh_{\bk}}(\bz_{\bk} + h_i^{(k_i)} \mathbf{e}_i) 
	& = \frac{u_{\bh_{\bk}}(\bz_{\bk} + h_i^{(k_i)} \mathbf{e}_i + h_\ell^{(k_\ell)} \mathbf{e}_\ell )  
		- u_{\bh_{\bk}}(\bz_{\bk} ) }{h_\ell^{(k_\ell)}} \\ 
		\nonumber & \qquad - \frac{h_i^{(k_i)}}{h_\ell^{(k_\ell)}} \delta_{x_i, h_i^{(k_i)}}^+ u_{\bh_{\bk}}(\bz_{\bk} ) . 
\end{align}
Choose a subsequence with a single index $k$ such that 
\[
	c_h \leq \frac{h_i^{(k)}}{h_\ell^{(k)}} \leq C_h 
\]
for some constants $0 < c_h < C_h$ independent of $k$, 
i.e., choose a quasi-uniform subsequence of $\bh_{\bk}$.  
Define $\delta_{\eta_\ell^i, \bh_k}^+$ by 
\[
	\delta_{\eta_\ell^i, \bh_k}^+ v(x) \equiv 
	\frac{v(\bz_{k} + h_i^{(k)} \mathbf{e}_i + h_\ell^{(k)} \mathbf{e}_\ell )  - v (\bz_{k} ) }{
		\sqrt{\left(h_\ell^{(k)} \right)^2 + \left( h_i^{(k)} \right)^2}} 
\]
and the constant $c_k \in \left[1 , \frac{\sqrt{1+C_h^2}}{c_h} \right]$ such that 
\[
	c_k h_\ell^{(k)} = \sqrt{\left(h_\ell^{(k)} \right)^2 + \left( h_i^{(k)} \right)^2} . 
\]
Then, by \eqref{iic_blowup}, there exists a constant $C' > 0$ such that  
\begin{align*}
	\delta_{\eta_\ell^i, \bh_k}^+ \left( u_{\bh_{k}} - \varphi \right)(\bz_{k}) 
	& = \frac{1}{c_k} \delta_{x_\ell, h_\ell^{(k)}}^+ \left( u_{\bh_{k}} - \varphi \right) (\bz_{k} + h_i^{(k)} \mathbf{e}_i) 
		+ \frac{h_i^{(k)}}{c_k h_\ell^{(k)}} \delta_{x_i, h_i^{(k)}}^+ \left( u_{\bh_{k}} - \varphi \right) (\bz_{k} ) \\ 
	& \geq 
		\delta_{x_\ell, h_\ell^{(k)}}^+ u_{\bh_{k}}(\bz_{k} + h_i^{(k)} \mathbf{e}_i) 
		- C' 
\end{align*} 
for all $k$ sufficiently large.  
Hence, 
\[
	\delta_{\eta_\ell^i, \bh_k}^+ \left( u_{\bh_{k}} - \varphi \right)(\bz_{k}) \to \infty , 
\]
a contradiction to \eqref{uk_local_max}.  
Therefore, 
$\limsup_{\min \bk \to \infty} \delta_{x_\ell, h_\ell^{(k_\ell)}}^+ u_{\bh_{\bk}}(\bz_{\bk} + h_i^{(k_i)} \mathbf{e}_i) < \infty$.  

We can analogously show 
$\liminf_{\min \bk \to \infty} \delta_{x_\ell, h_\ell^{(k_\ell)}}^+ u_{\bh_{\bk}}(\bz_{\bk} + h_i^{(k_i)} \mathbf{e}_i) > - \infty$ 
by assuming 
it diverges and arriving at the contradiction 
$\delta_{\xi_\ell^i, \bh_k}^- \left( u_{\bh_{k}} - \varphi \right)(\bz_{k}) \to - \infty$ 
for the analogous operator $\delta_{\xi_\ell^i, \bh_k}^-$ 
defined by 
\[
	\delta_{\xi_\ell^i, \bh_k}^- v(x) \equiv 
	\frac{v (\bz_{k} ) - v(\bz_{k} + h_i^{(k)} \mathbf{e}_i - h_\ell^{(k)} \mathbf{e}_\ell ) }{
		\sqrt{\left(h_\ell^{(k)} \right)^2 + \left( h_i^{(k)} \right)^2}} . 
\]
The argument now uses \eqref{oux1} to guarantee the existence of a subsequence such that  
$\delta_{x_\ell, h_\ell^{(k_\ell)}}^- u_{\bh_{\bk}}(\bz_{\bk} + h_i^{(k_i)} \mathbf{e}_i) \to - \infty$ 
with 
$\delta_{\xi_\ell^i, \bh_k}^-  u_{\bh_{\bk}}(\bz_{\bk}) 
\leq \delta_{x_\ell, h_\ell^{(k_\ell)}}^- u_{\bh_{\bk}}(\bz_{\bk} + h_i^{(k_i)} \mathbf{e}_i) + C'$.  
Therefore, we have 
\begin{equation} \label{oux2}
	\limsup_{\min \bk \to \infty} 
	\left| \delta_{x_\ell, h_{k_\ell}}^\mu u_{\bh_{\bk}}(\bz_{\bk} \pm h_i^{(k_i)} \mathbf{e}_i) \right| < \infty 
	\qquad \mu \in \{ + , - \}
\end{equation} 
as a consequence of \eqref{oux1}.

Applying \eqref{oux2}, there exists a constant $C'' > 0$ independent of $h_{\ell}^{(k_\ell)}$ 
and a subsequence (not relabelled) such that
\begin{align*}
	\left| \left[ D_{\bh_{\bk}}^{\mu \nu} u_{\bh_{\bk}}(\bz_{\bk}) \right]_{\ell i} \right| 
	& = \left| \delta_{x_i, h_i^{(k_i)}}^\nu \delta_{x_\ell, h_\ell^{(k_\ell)}}^\mu u_{\bh_{\bk}}(\bx_{\bk}) \right| 
	\leq \frac{4}{h_{\bk}'} C'' 
\end{align*}
for all $i \neq \ell$, $\mu, \nu \in \{+,-\}$, and $\min \bk$ sufficiently large. 
By \eqref{fdxy}, it follows that 
there exists a constant $M'' > 0$ independent of $h_{\ell}^{(k_\ell)}$ and a subsequence (not relabelled) such that
\begin{subequations}
\begin{align}
	& \widetilde{f}_\ell \left( h_\ell^{(k_\ell)} \right) 
		\left| \left[ \widetilde{D}_{\bh_{\bk}}^2 u_{\bh_{\bk}}(\bx_{\bk}) \right]_{\ell i} \right| 
		\leq \widetilde{f}_\ell \left( h_\ell^{(k_\ell)} \right) \frac{M''}{ h_{\bk}' } , \\ 
	& \widetilde{f}_\ell \left( h_\ell^{(k_\ell)} \right) 
		\left| \left[ \widehat{D}_{\bh_{\bk}}^2 u_{\bh_{\bk}}(\bx_{\bk}) \right]_{\ell i} \right| 
		\leq \widetilde{f}_\ell \left( h_\ell^{(k_\ell)} \right) \frac{M''}{ h_{\bk}' } , \\ 
	& \widetilde{f}_\ell \left( h_\ell^{(k_\ell)} \right) 
		\left| \delta_{x_\ell, h_\ell^{(k_\ell)}} u_{\bh_{\bk}}(\bx_{\bk}) \right| 
		\leq \widetilde{f}_\ell \left( h_\ell^{(k_\ell)} \right) M'' 
\end{align}
\end{subequations}
for all $i \neq \ell$ and $\min \bk$ sufficiently large. 
Therefore, there exists indices $K_0$, $K_\ell$ with $K_\ell >> K_0$ such that for $k_i = K_0$ for all $i \neq \ell$, 
there holds 
\begin{align} \label{2cbound}
	\widetilde{f}_\ell \left( h_\ell^{(k_\ell)} \right) 
	\bigg[ 8 d \left( 1 + K^{**} + \gamma + K^* \right) \frac{M + M''}{ \left( h_{\bk}' \right)^2 } \bigg] 
	< \frac{k_{**} C_\ell}{8}
\end{align} 
for all $k_\ell > K_\ell$.  
Plugging \eqref{2cbound} into \eqref{case2_eq} with $f_\ell$ replaced by $\widetilde{f}_\ell$, 
it follows that, for some index $\bk$, 
\begin{align*} 
	0 & \geq \widetilde{f}_\ell \left( h_\ell^{(k_\ell)} \right) 
		F \left( D^2 \varphi(\bx_0) , \nabla \varphi(\bx_0) , \varphi(\bx_0) , \bx_{\bk} \right) 
		+ \frac18 \left( k_{**} + \beta h_{\ell}^{(k_\ell)} \right) C_\ell  \\ 
	& >  \widetilde{f}_\ell \left( h_\ell^{(k_\ell)} \right) 
		F \left( D^2 \varphi(\bx_0) , \nabla \varphi(\bx_0) , \varphi(\bx_0) , \bx_{\bk} \right) 
\end{align*}
from which we can conclude $0 \geq F \left( D^2 \varphi(\bx_0) , \nabla \varphi(\bx_0) , \varphi(\bx_0) , \bx_0 \right)$. 

\smallskip 
All cases exhausted, we must have 
$0 \geq F \left( D^2 \varphi(\bx_0) , \nabla \varphi(\bx_0) , \varphi(\bx_0) , \bx_0 \right)$ 
whenever $\bx_0 \in \Omega$.  


\medskip
{\em Step 3b}:  We now consider the other scenario when $\mathbf{x}_0 \in \partial \Omega$ and show \eqref{step3ab} holds. Suppose (i) $\ou(\mathbf{x}_0) \leq g(\mathbf{x}_0)$. Then
\[
	B \left[ \varphi \right](\mathbf{x}_0) 
	= \varphi(\mathbf{x}_0) - g(\mathbf{x}_0) 
	= \ou(\mathbf{x}_0) - g(\mathbf{x}_0) 
	\leq 0 . 
\]
Hence, the assertion holds. 

Suppose (ii) there exists a sequence $\{ \bx_{\bk} \} \to \bx_0$ such that 
$\bx_{\bk} \in \cT_{\bk_{\bk}} \cap \partial \Omega$ 
for all $\bk$ and $u_{\bh_{\bk}} (\bx_{\bk}) \to \ou(\bx_0)$.  
Then $u_{\bh_{\bk}}(\bx_{\bk}) = u_{\bh_{\bk}} (\bx_{b_{\bk}}) = g(\bx_{b_{\bk}})$ 
for some boundary node $\bx_{b_{\bk}}$. By the continuity of $\varphi$ and $g$ 
and the fact $\ou(\bx_0) = \lim_{\min \bk \to \infty} u_{\bh_{\bk}}(\bx_{b_{\bk}})$, 
it follows that $\varphi(\bx_0) = \ou(\bx_0) = g(\bx_0)$ 
and we have $B[ \varphi](\bx_0) \leq 0$ by (i).

Suppose (iii) there exists a subsequence (not relabeled)
$\{\bx_{\bk}\} \in \cT_{\mathbf{h}_{\bk}} \cap \Omega$ 
for which $u_{\bh_{\bk}}(\bx_{\bk}) = u_{\bh_{\bk}}(\mathbf{x}_{\alpha_{\bk}})$ 
for some interior node $\mathbf{x}_{\alpha_{\bk}}$ for all $\bk$. 
If $\ou(\bx_0) \leq g(\bx_0)$, we have $B \left[ \varphi \right](\mathbf{x}_0) \leq 0$ by (i).  
Thus, we can assume $\ou(\mathbf{x}_0) > g(\mathbf{x}_0)$.  
Using the same argument as in {\em Step 3a} and the ellipticity of $F$, we can show 
\[
	0 \geq 
	F \bigl( D^2 \varphi(\mathbf{x}_0) , \nabla \varphi(\mathbf{x}_0) , \varphi(\mathbf{x}_0) , 
		\mathbf{x}_0 \bigr) 
\]
by using the scheme directly with $0 = \hF[u_{\bh_{\bk}}, \mathbf{x}_{\alpha_{\bk}}]$ 
for $u_{\bh_{\bk}}(\mathbf{x}_{\alpha_{\bk}}) \to \ou(\bx_0)$.

Combing (i), (ii), and (iii) which represent all (not mutually exclusive) scenarios, 
we have $\ou$ is a subsolution of the boundary condition in the viscosity sense 
since we either have 
\[
	B[ \varphi](\mathbf{x}_0) \leq 0 
\]
or 
\[
	F \bigl( D^2 \varphi(\mathbf{x}_0) , \nabla \varphi(\mathbf{x}_0) , \varphi(\mathbf{x}_0) , 
		\mathbf{x}_0 \bigr) \leq 0 . 
\]

\medskip
{\em Step 4}: We consider the case of a general test function $\varphi \in C^2(\oOme)$ which
is alluded to in {\em Step 3}. Recall that $\ou-\varphi$ is assumed to have 
a local maximum at $\mathbf{x}_0$. Using Taylor's formula we write
\begin{align*}
\varphi(\mathbf{x}) &=  \varphi(\mathbf{x}_0) + \nabla \varphi(\mathbf{x}_0) \cdot (\mathbf{x}-\mathbf{x}_0) 
+\frac12 ( \mathbf{x} - \mathbf{x}_0)^T D^2 \varphi(\mathbf{x}_0) (\mathbf{x}-\mathbf{x}_0) 
+ o(|\mathbf{x}-\mathbf{x}_0|^2) \\
&\equiv p(\mathbf{x}) + o(|\mathbf{x}-\mathbf{x}_0|^2). \nonumber
\end{align*}
For any $\sigma>0$, we define the following quadratic polynomial: 
\begin{align*}
p^\sigma(\mathbf{x}) &\equiv p(\mathbf{x}) + \sigma |\mathbf{x}-\mathbf{x}_0|^2 \\
&= \varphi(\mathbf{x}_0) + \nabla \varphi(\mathbf{x}_0) \cdot (\mathbf{x}-\mathbf{x}_0) 
+ (\mathbf{x} - \mathbf{x}_0)^T \bigl[\sigma \, I_{d \times d}+\frac12 D^2\varphi(\mathbf{x}_0)\bigr] 
 (\mathbf{x}-\mathbf{x}_0) .
\end{align*}

Trivially, 
$\nabla p^\sigma(\mathbf{x}) = \nabla \varphi(\mathbf{x}_0) 
+ [ 2 \sigma + D^2 \varphi(\mathbf{x}_0)] (\mathbf{x}-\mathbf{x}_0)$, 
$D^2 p^\sigma(\mathbf{x}) = 2\sigma I_{d \times d} + D^2 \varphi(x_0)$, and
$\varphi(\mathbf{x})-p^\sigma(\mathbf{x}) = o(|\mathbf{x}-\mathbf{x}_0|^2) 
- \sigma |\mathbf{x}-\mathbf{x}_0|^2 \leq 0$. 
Thus, $\varphi-p^\sigma$ 
has a local maximum at $\mathbf{x}_0$ and, therefore, $\ou- p^\sigma$ has a local
maximum at $\mathbf{x}_0$.  By the result of {\em Step 3} we have 
$F\bigl(D^2 p^\sigma(\mathbf{x}_0),
 \nabla p^\sigma(\mathbf{x}_0) , p^\sigma(\mathbf{x}_0) , \mathbf{x}_0\bigr)\leq 0$, 
 that is,  $F\bigl(2\sigma I_{d \times d} +D^2 \varphi(\mathbf{x}_0, 
\nabla \varphi(\mathbf{x}_0) , \varphi(\mathbf{x}_0) , \mathbf{x}_0 \bigr)\leq 0$.  
Taking $\liminf_{\sigma\to 0^+}$ and using the continuity of $F$ we obtain 
\begin{align*}
0 &\geq \liminf_{\sigma\to 0^+} F\bigl(2\sigma I_{d \times d} + D^2 \varphi(\mathbf{x}_0), 
\nabla \varphi (\mathbf{x}_0) , \varphi(\mathbf{x}_0) , \mathbf{x}_0 \bigr) \\
&\geq  F\bigl(D^2 \varphi (\bx_0), \nabla \varphi(\bx_0) , \varphi(\bx_0) , \bx_0\bigr).
\end{align*}
Thus, $\ou$ is a viscosity subsolution of \eqref{FD_problem}.

\medskip
{\em Step 5}: By following almost the same lines as those of {\em Steps 2-4 } 
we can show that if $\uu-\varphi$ takes a local minimum at $\mathbf{x}_0\in \overline{\Omega}$ 
for some $\varphi \in C^2(\oOme)$ with $\ou(\bx_0) = \varphi(\bx_0)$, then 
either 
\[
F\bigl(D^2 \varphi (\bx_0), \nabla \varphi (\bx_0) , \varphi(\bx_0) , \bx_0 \bigr) \geq 0 
\] 
if $\bx_0 \in \Omega$ 
or 
\begin{align*}
	F \bigl( D^2 \varphi(\mathbf{x}_0) , \nabla \varphi(\mathbf{x}_0) , \varphi(\mathbf{x}_0) , 
		\mathbf{x}_0 \bigr) &\geq 0 \\
\noalign{or}
	B \left[ \varphi \right](\mathbf{x}_0) &\geq 0 
\end{align*}
if $\bx_0 \in \partial \Omega$. 
Hence, $\uu$ is a viscosity supersolution of \eqref{FD_problem}. 

\medskip
{\em Step 6}: By the comparison principle (see Definition~\ref{comparison}), we
get $\ou\leq \uu$ on $\Omega$. On the other hand, by their definitions,
we have $\uu\leq \ou$ on $\Omega$. Thus, $\ou=\uu$, which coincides 
with the unique continuous viscosity solution $u$ of 
\eqref{FD_problem}. The proof is complete.
\end{proof}

We end this section with a couple of remarks concerning the convergence analysis.  

\begin{remark}
(a) The maximizing sequence defined by \eqref{max_seq} would be sufficient to prove convergence 
if the underlying FD scheme were monotone.  
When $\ou$ was sufficiently smooth but could be touched by a smooth function, 
we could use the maximizing sequence paired with the ellipticity of $F$ 
and the consistency of $\hF$ to move derivatives onto $\varphi$.  
When $\ou$ was not sufficiently smooth, we chose an alternative sequence/path that better 
exploited the lack of regularity for $\ou$ via the ellipticity of $F$ and the generalized monotonicity of $\hF$.   

(b) The numerical moment played the role of a nonnegative term (when choosing an appropriate subsequence) 
and allowed us to exploit the fact that $k_{**} > 0$.  
In particular, the generalized monotonicity property ensured by the numerical moment guaranteed that any 
divergent behavior corresponding to a mixed partial derivative approximation led to divergent behavior for 
an approximation of $u_{x_\ell x_\ell}$ for some Cartesian direction $x_\ell$.  
Thus, the monotonicity associated with the elliptic structure coming from the PDE could be exploited directly to 
move derivatives onto $\varphi$ via the divergence of $u_{x_\ell x_\ell}$.  

(c) Theorem \ref{FD_convergence_nd} is proved under the assumption that
the numerical scheme is admissible and $\ell^\infty$-norm stable giving potentially wider
applicability of the convergence result. 
We remark that this assumption has been 
verified for the narrow-stencil finite difference scheme proposed in this paper when $d=2,3$ 
by exploiting spectral properties of the numerical moment 
(see Sections \ref{stability_admissibility_sec} and \ref{H2-stability_sec}) 
and a discrete Sobolev embedding result. 
\end{remark}

\section{Numerical experiments} \label{FD_2D_numerics}

In this section we present several two-dimensional numerical tests 
to gauge the performance of the proposed narrow-stencil finite difference scheme for approximating 
viscosity solutions.  
All tests are performed using {\it Matlab} and the built-in nonlinear solver {\it fsolve} 
with an initial guess given by the zero function on the interior of the domain. 
The errors are measured in the $\ell^\infty$ norm.  
The examples will correspond to HJB operators independent of the gradient argument. 
We use the numerical operator $\widehat{F}$ with the numerical moment and the numerical viscosity 
given by \eqref{Abeta} for $\gamma > 0$ specified 
and the numerical viscosity coefficient $\beta = 0$. 
We note that both uniformly and degenerate elliptic cases are considered, 
and refer the interested reader to \cite{Lewis_dissertation} to see additional numerical tests 
for the Monge-Amp\`ere equation.

\subsection{Test 1: Finite Control Set} 

Consider the HJB equation \eqref{HJB} with $\Omega = (0,1)^2$, 
\begin{align*}
	A^{\theta}(x,y) \in & 
	\Bigg\{ \left[ \begin{array}{cc} -1 & 1 \\ 1 & -1 \end{array} \right] , 
	 \left[ \begin{array}{cc} -2 & 1 \\ 1 & -1 \end{array} \right] , 
	  \left[ \begin{array}{cc} -1 & - 1 \\ -1 & -1 \end{array} \right] , 
	   \left[ \begin{array}{cc} -1 & - 1 \\ -1 & -2 \end{array} \right] , \\ 
	    & \qquad 
	    \left[ \begin{array}{cc} -2 & 1 \\ 1 & -2 \end{array} \right] , 
	     \left[ \begin{array}{cc} -2 & - 1 \\ -1 & -2 \end{array} \right] , 
	      \left[ \begin{array}{cc} -2 & - 1 \\ -1 & -1 \end{array} \right] , 
	       \left[ \begin{array}{cc} -1 & - 1 \\ -1 & -2 \end{array} \right] \Bigg\} , 
\end{align*}
$\beta^\theta = \vec{0}$, $c^\theta = 0$, and $f_\theta$ and Dirichlet boundary data 
chosen such that the exact solution 
is given by $u(x,y) = \sin \bigl( 2 \pi (1.2x - y) \bigr)$.  
Then, the optimal control $A^{\theta^*}(x,y)$ corresponds to a discontinuous 
coefficient matrix.  
Furthermore, each choice for $A^\theta$ is symmetric negative semidefinite and 
the control set corresponds to a family of degenerate elliptic problems such as 
\[
	- u_{xx} - 2 u_{xy} - u_{yy} = f
\] 
when assuming $u_{xy} = u_{yx}$.  
Simple central difference methods are known to diverge for such examples as 
seen in \cite{Oberman06}.  
However, with the numerical moment, we recover convergent, non-monotone schemes.  
The results for $A = 4 \mathbf{1}$ can be found in Table~\ref{table_test1} 
which shows nearly second order convergence.  

\begin{table}[htb] 
{\small 
\begin{center}
\begin{tabular}{| c | c | c |}
		\hline
	 $h$ & Error & Order \\ 
		\hline
	3.63e-02 & 7.25e-01 & \\ 
		\hline
	2.40e-02 & 3.72e-01 & 1.61 \\ 
		\hline
	1.79e-02 & 2.25e-01 & 1.72 \\ 
		\hline
	1.19e-02 & 1.09e-01 & 1.78 \\ 
		\hline
	8.89e-03 & 6.41e-02 & 1.82 \\ 
		\hline
	7.11e-03 & 4.24e-02 & 1.85 \\ 
		\hline 
\end{tabular}
\end{center}
}
\caption{
Approximations for Test 1 with $A = 4 \mathbf{1}$.   
}
\label{table_test1}
\end{table}


\subsection{Test 2: Infinite Control Set} 

This example is adapted from \cite{Smears_Suli}.  
Let $\Lambda = [0,\pi/3] \times SO(2)$, where $SO(2)$ is the set of $2 \times 2$ rotation matrices 
and define $\sigma^{\theta}$ by 
\[
	\sigma^\theta \equiv R^T \left[ \begin{array}{cc} 1 & \sin(\phi) \\ 0 & \cos(\phi) \end{array} \right] , 
	\qquad \theta = (\phi, R) \in \Lambda . 
\]
Consider the HJB equation \eqref{HJB} with $\Omega = (0,1)^2$, 
$A^\theta = \frac12 \sigma^\theta \left( \sigma^\theta \right)^T$, 
$\beta^\theta = \vec{0}$, $c^\theta = \pi^2$, and $f_\theta = \sqrt{3} \sin^2(\phi/\pi^2) + g$, 
with $g$ chosen independent of $\theta$, and Dirichlet boundary data 
chosen such that the exact solution 
is given by $u(x,y) = e^{xy} \sin(\pi x) \sin( \pi y)$.  
Thus, the optimal controls vary significantly throughout the domain and the corresponding 
diffusion coefficient is not diagonally dominant in parts of $\Omega$.  
Furthermore, the coefficient matrix is degenerate for certain choices of $\theta$.  
The problem corresponds to optimizing over the choice of 
orientation and angle between two Wiener diffusions.  
The results for $A = 4 \mathbf{1}$ can be found in Table~\ref{table_test2} 
which shows nearly second order convergence.  

\begin{table}[htb] 
{\small 
\begin{center}
\begin{tabular}{| c | c | c |}
		\hline
	 $h$ & Error & Order \\ 
		\hline
	9.43e-02 & 2.60e-01 & \\  
		\hline
	6.15e-02 & 1.28e-01 & 1.66 \\ 
		\hline
	4.56e-02 & 7.32e-02 & 1.86 \\ 
		\hline
	3.63e-02 & 4.69e-02 & 1.94 \\ 
		\hline
	2.89e-02 & 2.99e-02 & 1.97 \\ 
		\hline 
\end{tabular}
\end{center}
}
\caption{
Approximations for Test 2 with $A = 4 \mathbf{1}$.   
}
\label{table_test2}
\end{table}


\subsection{Test 3: Low Regularity Solution} 

Consider the second order linear problem with non-divergence form 
\[
	A(x,y) : D^2 u = f \qquad \text{in } \Omega
\]
with   
$\Omega = (-0.5,0.5)^2$, 
\[
	A(x,y) = \frac{16}{9} \left[ \begin{array}{cc} 
	x^{2/3} & - x^{1/3} y^{1/3} \\ 
	- x^{1/3} y^{1/3} & y^{2/3} 
	\end{array} \right] , 
\] 
and $f(x)$ and Dirichlet boundary data chosen such that the solution is given by 
$u(x,y) = x^{4/3} - y^{4/3} \in W^{m,p}(\Omega)$ for $(4-3m)p > -1$.
Then, the problem is degenerate elliptic and the exact solution is not in $C^2(\Omega)$.  
The results for $A = 4 \mathbf{1}$ can be found in Table~\ref{table_test3}.  
Similar test results hold for approximating the viscosity solution of the infinite Laplacian 
\[
	- \Delta_\infty u \equiv - \nabla u \cdot D^2 u \nabla u 
	= - u_x^2 u_{xx} -  u_x u_y u_{xy} - u_y u_x u_{yx} - u_y^2 u_{yy} = 0 
\]
with Dirichlet boundary data chosen such that the solution is given by 
$u(x,y) = x^{4/3} - y^{4/3}$.  
For the test we have fixed the value of $\nabla u$ in the definition of $A(x,y)$.  

\begin{table}[htb] 
{\small 
\begin{center}
\begin{tabular}{| c | c | c |}
		\hline
	 $h$ & Error & Order \\ 
		\hline
	6.15e-02 & 3.38e-02 & \\ 
		\hline
	4.56e-02 & 3.01e-02 & 0.38 \\ 
		\hline
	3.63e-02 & 2.73e-02 & 0.42 \\ 
		\hline
	2.89e-02 & 2.50e-02 & 0.40 \\ 
		\hline
	2.24e-02 & 2.26e-02 & 0.40 \\ 
		\hline 
	1.43e-02 & 1.89e-02 & 0.40 \\ 
		\hline 
\end{tabular}
\end{center}
}
\caption{
Approximations for Test 3 with $A = 4 \mathbf{1}$.   
}
\label{table_test3}
\end{table}


\section{Conclusion} \label{conc_sec} 
In this paper we have constructed and analyzed a new narrow-stencil finite difference (FD) method for 
approximating the viscosity solution of fully nonlinear second order problems 
such as the Hamilton-Jacobi-Bellman (HJB) problem from stochastic optimal control.  
The new Lax-Friedrichs-like FD method is well-posed, $\ell^p$-norm stable (for $p=2,\infty$),  
and convergent. 
The fundamental building block of the scheme is a numerical moment, a discrete operator that 
corresponds to a high order linear perturbation of the problem.  
The conditions for choosing the numerical moment 
are easy to realize in practice.  
By removing the monotonicity condition, narrow-stencils can be used 
to design convergent schemes
for a much wider class of fully nonlinear problems.  

The numerical tests in Section~\ref{FD_2D_numerics} as well as the tests 
found in \cite{Feng_Kao_Lewis13} provide strong evidence that the 
stabilization technique based on adding a numerical moment  
can be used to remove the numerical artifacts that plague standard FD discretizations 
of fully nonlinear problems (see \cite{Feng_Glowinski_Neilan10}).  
By using a high order stabilization term, the scheme approximates a low-regularity function 
as a limit of smoother functions.  
Consequently, 
low-regularity artifacts are removed and/or destabilized by the scheme.  

The FD method proposed in this paper can be applied to 
various fully nonlinear second order elliptic problems 
as well as linear problems with non-divergence form.  
While the analysis is carried out for the case when the differential  
operators are globally Lipschitz, 
we expect most of the results still hold for locally Lipschitz operators 
when using adaptive numerical moments.  
The numerical moment may also be used as a low-regularity indicator that can 
be explored for designing adaptive schemes (which will be reported in a future work).  
The methods in this paper can easily be 
extended to parabolic problems with the form 
\[
	u_t + F( D^2 u , \nabla u , u , t , x ) = 0 
\]
using the method of lines. As such, these methods are suitable for many application problems.   
Both the theoretical analysis of this paper for approximating HJB equations 
and the positive numerical tests regarding 
the Monge-Amp\`ere equation found in \cite{Lewis_dissertation} 
hint at the robustness of the proposed narrow-stencil FD method 
and an underlying framework for designing narrow-stencil methods for   
fully nonlinear PDEs.  
As the method and results presented in this paper are exploited at the solver level,  
we expect this new FD method to be a significant step in the design of  
practical methods for approximating viscosity solutions.


\bibliographystyle{elsarticle-num}

\appendix
\section{Some auxiliary linear algebra results} \label{appendix_sec}
 
\medskip

Below we present a simple result for comparing two SPD matrices 
and several $2$-norm estimate results for perturbations of the identity matrix 
and SPD matrices. These results, which have independent interests, are crucially used to analyze the 
$\ell_2$-norm stabilities of the proposed narrow-stencil finite difference method in 
Sections \ref{stability_admissibility_sec} and \ref{H2-stability_sec}. 

\begin{lemma} \label{SPD_ordering_lemma}
Let $A, B, Q \in \mathbb{R}^{J \times J}$ such that 
$A, B$ are symmetric positive definite and $Q$ is orthogonal.  
Suppose $A - Q^T B Q$ is symmetric positive definite.  
Then $A - B$ is symmetric positive definite.  
\end{lemma}

\begin{proof} 
Since $Q$ is orthogonal, it has a complete set of eigenvectors 
$\{ \bv_1, \bv_2, \ldots, \bv_J \}$ that spans $\mathbb{R}^J$.  
Furthermore, each corresponding eigenvalue satisfies $| \lambda_k | = 1$ for all $k=1,2,\ldots,J$.  

Choose $k \in \{1,2,\ldots, J\}$.  
Then 
\[
	\bv_k^T Q^T B Q \bv_k = \lambda_k^2 \bv_k^T B \bv_k = \bv_k^T B \bv_k , 
\]
and it follows that 
\[
	\bv_k^T \left( A - Q^T B Q \right) \bv_k = \bv_k^T (A - B ) \bv_k . 
\]
Let $X \in \mathbb{R}^{J \times J}$ be defined by $[ X ]_k = \bv_k$.  
Then $X$ is nonsingular, and we have 
\begin{alignat*}{2}
	A - Q^T B Q \qquad \text{is SPD} 
	& \implies X^T \left( A - Q^T B Q \right) X && \qquad \text{is SPD} \\ 
	& \implies X^T \left( A - B \right) X && \qquad \text{is SPD} \\ 
	& \implies (X^{-1})^T X^T \left( A - B \right) X X^{-1} && \qquad \text{is SPD} \\ 
	& \implies  A - B && \qquad \text{is SPD} . 
\end{alignat*}
The proof is complete. 
\hfill 
\end{proof}

\begin{lemma}
Let $A, B \in \mathbb{R}^{J \times J}$ such that 
$A, B$ are symmetric positive definite.  
Then
\[
	\| (B - \tau A) \bx \|_2 \leq \| B \bx \|_2 \qquad \forall \bx \in \mathbb{R}^J 
\]
for all $\tau > 0$ sufficiently small.  
\end{lemma}

\begin{proof}
Let $B = Q \Lambda Q^*$ for $\Lambda > 0$.  
Choose $\tau > 0$ sufficiently small such that 
$\Lambda - \tau Q^* A Q$ is SPD.  
Then, $(\Lambda - \tau Q^* A Q)^2 \leq \Lambda^2$, and,  
for $\by = Q^* \bx$, there holds 
\begin{align*}
\| (B - \tau A)\bx\|_2^2 
& = \| ( Q \Lambda Q^* - \tau A) \bx \|_2^2 \\ 
& = \| Q ( \Lambda - \tau Q^* A Q ) Q^* \bx \|_2^2 \\ 
& = \| ( \Lambda - \tau Q^* A Q ) \by \|_2^2 \\ 
& = \by^* ( \Lambda - \tau Q^* A Q )^2 \by \\ 
& \leq \by^* \Lambda^2 \by \\ 
& = \| \Lambda \by \|_2^2 = \| \Lambda Q^* \bx \|_2^2 = \| Q \Lambda Q^* \bx \|_2^2 \\ 
& = \| B \bx \|_2^2 . 
\end{align*}
The proof is complete.  
\hfill 
\end{proof}

\begin{lemma}\label{lemma_symmetrization}
Let $A, F \in \mathbb{R}^{J \times J}$ such that 
$A, F$ are symmetric positive definite. 
Define $R \in \mathbb{R}^{J \times J}$ such that $R$ is upper triangular and $F = R^* R$.  
Then
\[
	\| \sigma I - FA \|_2 \leq \sigma 
\]
for all positive  constants $\sigma$ such that $\sigma I > R A R^*$.  

\end{lemma}

\begin{proof}
By assumption, $\sigma I - R A R^*$ is symmetric positive definite.  
Thus $R^* (\sigma I-R A R^*) R$ is symmetric positive definite, 
and it follows that 
\begin{align*}
	R^* (\sigma I-R A R^*) R
	& = \sigma  F - F A F  
	\leq \sigma F . 
\end{align*}
By symmetry, 
$\left( R^* (\sigma I-R A R^*) R \right)^2 < \sigma^2F^2$, 
and we have 
\begin{align*}
	& \bw^* R^* (\sigma I-R A R^*) R R^* (\sigma I-RAR^*) R \bw 
		- \bw^* \sigma^2 F^2 \bw \\ 
	& \qquad = \bw^* \left( R^* (\sigma I-R A R^*) R R^* (\sigma I-RAR^*) R - \sigma^2 F^2  \right) \bw \\ 
	& \qquad \leq 0 
\end{align*}
for all $\bw \in \mathbb{R}^J$.  
Let $\bz \in \mathbb{R}^J$ and $\bw = F^{-1} \bz$.  
Then 
\begin{align*}
	& \| R^* (\sigma I-R A R^*) (R^*)^{-1} \bz \|_2^2 \\ 
	& \qquad = \bz^* \left[ (R^*)^{-1} \right]^* (\sigma I-R A R^*) R R^* (\sigma I-R A R^*) (R^*)^{-1} \bz \\ 
	& \qquad = \bz^* \left[ R R^{-1} (R^*)^{-1} \right]^* (\sigma I-R A R^*) R R^* (\sigma I-R A R^*) R R^{-1} (R^*)^{-1} \bz \\ 
	& \qquad = \bz^* \left[ R (R^*R)^{-1} \right]^* (\sigma I-R A R^*) R R^* (\sigma I-R A R^*) R (R^* R)^{-1} \bz \\ 
	& \qquad = \left[ R F^{-1} \bz \right]^* (\sigma I-R A R^*) R R^* (\sigma I-R A R^*) R F^{-1} \bz \\ 
	& \qquad = \bw^* R^* (\sigma I-R A R^*) R R^* (\sigma I-R A R^*) R \bw \\ 
	& \qquad \leq \bw^* \sigma^2 F^2 \bw 
	= \| \sigma F \bw \|_2^2 
	= \| \sigma  F F^{-1} \bz \|_2^2 \\ 
	& \qquad = \sigma^2 \| \bz \|_2^2 , 
\end{align*}
and it follows that 
\begin{align*}
	\| \sigma I - FA \|_2 
	& = \sup_{\bz \neq \vec{0}} \frac{\| (\sigma I-FA)\bz \|_2}{\| \bz \|_2} \\ 
	& = \sup_{\bz \neq \vec{0}} \frac{\| (\sigma I-R^* RA)\bz \|_2}{\| \bz \|_2} \\ 
	& = \sup_{\bz \neq \vec{0}} \frac{\| R^*(\sigma I-RAR^*) (R^*)^{-1}\bz \|_2}{\| \bz \|_2} \\ 
	& \leq \sup_{\bz \neq \vec{0}} \frac{\sigma \| \bz \|_2}{\| \bz \|_2} 
	= \sigma . 
\end{align*}
The proof is complete.  
\hfill
\end{proof}

\begin{corollary}\label{corollary4.2}
Let $A, B, F \in \mathbb{R}^{J \times J}$ such that 
$A, B, F$ are symmetric positive definite.  
Then
\[
	\| (B - \tau FA) \bx \|_2 \leq \| B \bx \|_2 \qquad \forall \bx \in \mathbb{R}^J 
\]
for all $\tau > 0$ sufficiently small.  
\end{corollary}

\begin{proof}
Since $B$ is SPD, there exists matrices $Q, \Lambda \in \mathbb{R}^{J \times J}$ 
such that $B = Q \Lambda Q^*$ with 
$Q$ orthogonal and $\Lambda$ diagonal with $\lambda_{ii} > 0$ for all $i=1,2,\ldots,J$.  
Let $\lambda_*$ denote the smallest diagonal entry of $\Lambda$. 

Choose $\bx \in \mathbb{R}^J$.  
Observe that, by Lemma~\ref{lemma_symmetrization}, 
the singular values of the matrix $(\lambda_*/2) I - \tau ( Q^* F Q ) (Q^* A Q )$ are all bounded above by 
$\lambda_*/2$ when $\tau$ is sufficiently small since $Q^* F Q$ and $Q^* A Q$ are SPD.  
Let $(\lambda_*/2) I - \tau ( Q^* F Q ) (Q^* A Q ) = U \Sigma V^*$ be a singular value decomposition.  
Then, 
\begin{align*}
	& \left\| \big( (\lambda_*/2) I - \tau ( Q^* F Q ) (Q^* A Q ) \big) \bx \right\|_2 \\ 
	& \qquad = \left\| U \Sigma V^* \bx \right\|_2 
	\leq \left\| U (\lambda_*/2) I V^* \bx \right\|_2 
	= (\lambda_*/2) \left\| U V^* \bx \right\|_2 . 
\end{align*}
Observe that $U V^*$ is unitary and, consequently, normal.    
Thus, there exists a diagonal matrix $\widetilde{\Lambda}$ and a unitary matrix $\widetilde{U}$ such that 
$U V^* = \widetilde{U} \widetilde{\Lambda} \widetilde{U}^*$ with $\widetilde{\lambda}_{ii} \in \{-1,1\}$ 
for all $i$.  
Then, there holds $| \widetilde{\Lambda} | = I$ 
so that 
$\widetilde{U} | \widetilde{\Lambda} | \widetilde{U}^* = \widetilde{U} \widetilde{U}^* = I$.  

Let $\by = Q^* \bx$. 
Since each component $[ \Lambda - (\lambda_*/2) I ]_{ii} > 0$ for all $i$, 
we have the length of the vector 
$\big( \Lambda - (\lambda_*/2) I \big) \by + \big( (\lambda_*/2) I - \tau ( Q^* F Q ) (Q^* A Q ) \big) \by$ 
is maximized when each component of $\big( (\lambda_*/2) I - \tau ( Q^* F Q ) (Q^* A Q ) \big) \by$ 
is scaled by the maximal positive amount, i.e., $\lambda_*/2$.  
Thus, 
\begin{align*}
	\| \left( B - \tau F A \right) \bx \|_2 
	& = \| \left( Q \Lambda Q^* - \tau F A \right) \bx \|_2 \\ 
	& = \| Q \left( \Lambda Q^* - \tau Q^* F A \right) \bx \|_2 
		= \| \left( \Lambda Q^* - \tau Q^* F A \right) \bx \|_2 \\ 
	& = \| \left( \Lambda - \tau Q^* F A Q \right) Q^* \bx \|_2 
		= \| \left( \Lambda - \tau Q^* F A Q \right) \by \|_2 \\ 
	& = \| \left( \Lambda - \tau ( Q^* F Q ) (Q^* A Q ) \right) \by \|_2 \\ 
	& = \| \left( \Lambda - (\lambda_*/2) I \right) \by + \left( (\lambda_*/2) I - \tau ( Q^* F Q ) (Q^* A Q ) \right) \by \|_2 \\ 
	& = \| \left( \Lambda - (\lambda_*/2) I \right) \by + U \Sigma V^* \by \|_2 \\ 
	& \leq \| \left( \Lambda - (\lambda_*/2) I \right) \by 
		+ (\lambda_*/2) \widetilde{U} | \widetilde{\Lambda} | \widetilde{U}^* \by \|_2 \\ 
	& = \| \left( \Lambda - (\lambda_*/2) I \right) \by 
		+ (\lambda_*/2) \by \|_2 \\ 
	& =  \| \Lambda \by \|_2 
		= \| \Lambda Q^* \bx \|_2 
		= \| Q \Lambda Q^* \bx \|_2 
		= \| B \bx \|_2 
\end{align*}
for all $\tau$ sufficiently small. 
The proof is complete.
\hfill 
\end{proof}

\end{document}